\documentclass{article}
\pdfoutput=1

\usepackage[utf8]{inputenc}
\usepackage[american]{babel}
\usepackage{amsmath, amsfonts, dsfont}
\usepackage{array, multirow, hhline, tabularx, wrapfig}
\usepackage{comment}
\usepackage{stmaryrd}
\usepackage{color}
\usepackage{geometry}
\usepackage{hyperref}
\usepackage{authblk}
\usepackage{chngcntr}
\usepackage{tikz}
\usetikzlibrary{matrix,arrows, shapes, snakes}
\usepackage{kbordermatrix}
\usepackage[inline]{enumitem}
\usepackage{algpseudocode, algorithm, algorithmicx}
\usepackage{natbib}
\usepackage{amsthm}
\usepackage{amssymb}

\hypersetup{colorlinks=true, citecolor = blue}
\definecolor{light-gray}{gray}{0.95}

\newcommand\Sym[1]{\mathfrak{S}_{#1}}
\newcommand\Rank[1]{\Gamma(#1)}
\newcommand\Sn{\mathfrak{S}_{n}}
\newcommand\1[1]{\mathds{1}_{#1}}
\newcommand\n{\llbracket n \rrbracket}
\newcommand\set[1]{\llbracket #1 \rrbracket}

\newcommand{\argmax}{\operatornamewithlimits{argmax}}
\newcommand{\supp}{\operatorname{supp}}
\newcommand{\Supp}{\operatorname{\textbf{supp}}}
\newcommand\Subsets[1]{\mathcal{P}(#1)}
\newcommand\SubsetsWE[1]{\bar{\mathcal{P}}(#1)}
\newcommand{\EOD}{\widehat{\mathcal{A}}_{N}}
\newcommand{\Gn}{\bar{\Gamma}_{n}}
\newcommand{\GGn}{\Gamma^{\ast}_{n}}
\newcommand\Hn{\mathbb{H}_{n}}
\newcommand{\Space}[1]{L(#1)}
\newcommand{\Part}{\operatorname{Part}}
\newcommand{\eig}{\operatorname{eig}}
\newcommand{\shape}{\operatorname{shape}}
\newcommand{\SYT}{\text{SYT}_{n}}
\newcommand{\Span}{\operatorname{span}}
\newcommand{\image}{\operatorname{Im}}
\newcommand{\noninv}{\operatorname{noninv}}
\newcommand{\Bij}{\operatorname{Bij}}

\theoremstyle{plain}
\newtheorem{theorem}{Theorem}

\newtheorem{lemma}[theorem]{Lemma}
\newtheorem{proposition}[theorem]{Proposition}
\theoremstyle{remark}
\newtheorem{example}{Example}
\newtheorem{remark}[theorem]{Remark}
\theoremstyle{definition}
\newtheorem{definition}[theorem]{Definition}

\begin{document}

\title{A Multiresolution Analysis Framework\\ for the Statistical Analysis of Incomplete Rankings}

\author[1]{Eric Sibony\thanks{Corresponding author - email: esibony@gmail.com - postal address: Telecom ParisTech 46 rue Barrault, 75013 Paris, France.}}
\author[1]{St\'ephan Cl\'emen\c{c}on}
\author[2]{J\'er\'emie Jakubowicz}
\affil[1]{LTCI UMR No. 5141 Telecom ParisTech/CNRS, Institut Mines-Telecom, Paris, 75013, France}
\affil[2]{SAMOVAR UMR No. 5157 Telecom SudParis/CNRS, Institut Mines-Telecom, Paris, 75013, France}

\date{}



\maketitle

\begin{abstract}%
Though the statistical analysis of ranking data has been a subject of interest over the past centuries, especially in economics, psychology or social choice theory, it has been revitalized in the past 15 years by recent applications such as recommender or search engines and is receiving now increasing interest in the machine learning literature. Numerous modern systems indeed generate ranking data, representing for instance ordered results to a query or user preferences. Each such ranking usually involves a small but varying subset of the whole catalog of items only. The study of the variability of these data, \textit{i.e.} the statistical analysis of \textit{incomplete rankings}, is however a great statistical and computational challenge, because of their heterogeneity and the related combinatorial complexity of the problem. Whereas many statistical methods for analyzing \textit{full rankings} (orderings of all the items in the catalog) are documented in the dedicated literature, \textit{partial rankings} (full rankings with ties) or \textit{pairwise comparisons}, only a few approaches are available today to deal with incomplete ranking, relying each on a strong specific assumption.

It is the purpose of this article to introduce a novel general framework for the statistical analysis of incomplete rankings. It is based on a representation tailored to these specific data, whose construction is also explained here, which fits with the natural multi-scale structure of incomplete rankings and provides a new decomposition of rank information with a \textit{multiresolution analysis} interpretation (MRA). We show that the MRA representation naturally allows to overcome both the statistical and computational challenges without any structural assumption on the data. It therefore provides a general and flexible framework to solve a wide variety of statistical problems, where data are of the form of incomplete rankings.
\end{abstract}


\section{Introduction}
\label{intro}

As they represent observations of ordinal comparisons, rankings naturally arise in a wide variety of situations, especially when the data are related to human activities: ballots in political elections, survey answers, expert judgments, sports results, competition rankings, customer buying behaviors or user preferences among others. Initiated in social choice theory, the statistical analysis of ranking data has been the subject of much attention in the statistics literature, mostly in relation with psychological and economics applications. The last decade has seen a revival of interest for this topic, in particular in the machine learning and artificial intelligence literature, motivated by various modern applications such as recommendation  systems, search engines or crowdsourced annotation, producing or fed by massive ranking data and bringing new statistical and computational challenges. The goal of this paper is twofold: to explain the major limitations of the state-of-the-art in the domain of statistical analysis of ranking data and to introduce a novel framework for overcoming them.

From a broad perspective, rankings are defined as (strict) partial orders ``$\prec$'' on a set of $n\geq 1$ items $\n := \{ 1,\dots, n\}$ \citep[see for instance][for a rigorous definition]{Stanley86}, where $a\succ b$ means that item $a$ is preferred to / ranked higher than item $b$. A dataset of rankings is thus a collection of partial orders $(\prec_{1}, \dots, \prec_{N})$ modeled as IID samples of a probability distribution over the set of partial orders on $\n$. Several typical problems can then be considered. \textit{Ranking aggregation} consists in finding a ranking $\prec$ that best ``summarizes'' the dataset. It corresponds to finding the outcome of an election, the final ranking in a sports competition or the global ordering of items in the presence of several experts or even in a crowdsourced judgment setting. Statistical analysis is then used to define the notion of ``summary'' and to study different notions of variability in this context. Another issue of major interest is the statistical estimation of the model underlying the observations in order to interpret the data or predict new observations. It is applied for instance to analyze survey answers or buying behaviors of customers and take business decisions or to predict individual preferences in a recommendation setting. Clustering individuals based on the rankings they express on alternatives is another important task, used for instance to segment the population of customers based on their tastes, for marketing purposes.


As the set of partial orders on $\n$ exhibits an extremely rich mathematical structure, the vast majority of the approaches introduced in the literature focus on a certain type of rankings. A widely considered type is that of ``full rankings'', defined as strict total orders on $\n$, of the form $a_{1}\succ a_{2}\succ\dots\succ a_{n}$, where $a_{1}$ and $a_{n}$ are respectively the items ranked first and last. Such an order is usually described as the permutation $\sigma$ on $\n$ that maps an item to its rank: $\sigma(a_{i}) = i$ for all $i\in\n$. Statistical analysis of full rankings thus rely on probabilistic modeling on the symmetric group $\Sn$, the set of permutations on $\n$ namely. Approaches include ``parametric'' models based on a psychological interpretation, such as the Mallows model and its extensions \citep[see for instance][]{Mallows57, FV86, CM93, DPR04, MM2014}, the Plackett-Luce model and its extensions \citep[see for instance][]{Luce59, Plackett75, Henery81, FV88, Xu2000, GM08} with several fitting methods \citep[see for instance][]{Hunter04, Guiver09, CD12} or the Thurstone model and its extensions  \citep[see for instance][]{Thurstone27, MO99, WBA02}, applied for instance to label ranking \citep[][]{CHH09, CDH10}, ranking aggregation \citep[]{MPPB07,ACPX13}, ranking prediction \citep[][]{LL02,ACX13} or clustering and data analysis \citep[][]{GM09}. Many other approaches use a ``non-parametric'' model backed by a classic mathematical assumption, such as distance-based assumptions \citep[see for instance][]{FA86, LL03, Sun2012}, independence modeling \citep[see for instance][]{CFV91, Csiszar09b, HG2012}, embedding in Euclidean spaces \citep[see for instance][]{YC01, HW09, PMC11}, pairwise decomposition and modeling of pairwise comparisons \citep[see for instance][]{HFCB08, VZ14}, sparsity assumptions \citep[see for instance]{Jagabathula2011}, sampling-based models \citep[see for instance][]{Diaconis1998,Diaconis06, Ailon08, Ailon14}, algebraic toric models \citep[see for instance][]{Csiszar09a,SW12}, or harmonic analysis \citep[see for instance][]{Diaconis1988, Diaconis89, HGG09, Kondor2010, Kakarala2011, Irurozki2011, Kondor2012}.

In many applications however, observations are not total orders on $\n$ and cannot be represented by permutations. Most of the approaches for the statistical analysis of full rankings thus cannot be applied and either they must be adapted or new ones must be invented, with respect to the considered types of rankings. The literature distinguishes two main types of rankings: the \textit{partial rankings} (also referred to as \textit{bucket orders}), and the incomplete rankings (sometimes called \textit{subset rankings} or \textit{listwise rankings}), see for instance \citet{Marden96} or \citet{AY14}. Partial rankings are orders of the form $a_{1,1},\dots,a_{n_{1},1}\succ\dots\succ a_{1,r},\dots,a_{n_{r},r}$ with $r\geq 1$ and $\sum_{i=1}^{r}n_{i} = n$. They correspond to full rankings with ties and include the particular case of top-$k$ rankings, of the form $a_{1}\succ\dots\succ a_{k}\succ \textit{the rest}$. If some approaches for full rankings directly apply to partial rankings \citep[see for instance][]{Diaconis89, LL03, HFCB08, VZ14}, the extension of other methods has motivated many contributions in the literature, to extend the Mallows or Plackett-Luce models \citep[see for instance][]{BOB07,MLB10, QGL10, CTM14}, to define and study proper distances \citep[see for instance][]{Critchlow1985, FKMSV03, FKMSV06} or to extend nonparametric methods \citep[see for instance][]{Huang12, Kakarala2012}. Other approaches have also been introduced specifically for partial rankings, for different applications such as estimation \citep{LM08}, prediction \citep{CHWW12} or ranking aggregation \citep{AS12}.

Incomplete rankings are defined as partial orders of the form $a_{1}\succ a_{2}\succ\dots\succ a_{k}$ with $2 \leq k < n$. The fundamental difference with full or partial rankings is that each incomplete ranking only involves a (possibly small) subset of items, which can vary a lot among observations. The statistical challenge of the analysis of incomplete rankings is then to handle their heterogeneity with accuracy. Incomplete rankings include the specific case of pairwise comparisons (for $k=2$), which has attracted much attention in the literature. The impressive surveys of \citet{DF76} and \citet{Cattelan12} already show the abundance of the developments from the Thurstone model (with the additional insights from \citet{Mosteller51}) and from the Bradley-Terry model \citep[][]{BT52}, at the origin of the Plackett-Luce model, that keep growing with new results, such as the theoretical guarantees provided in \citet{SPBRBW15}. Other approaches use the Mallows model \citep[see][]{BFHS14,LuB14}, matrix approximation \citep{KO97}, entropy maximization methods \citep{AS11} or graphical models \citep{DIS15}. Recovering a full ranking on $\n$ from pairwise comparisons has been a topic of special interest \citep[see for instance][]{CSS99, BM08, JN11, GL11, Ailon12, NOS12, WJJ13, CBCTH13, RA14}, in particular with the introduction and development of the HodgeRank framework \citep[see][]{JLYY11,XHJYLY12,DSS2012, OBO13} to exploit the topological structure of the pairwise comparisons graph. Another subject of interest concerns the case where items have features, and the task of learning how to rank them can be cast as an ``ordinal regression'' problem, which many off-the-shelf supervised learning algorithms can be applied to \citep[see for instance][]{HGO00, FISS03, BSRLDHH05, CG05}.

Much less contributions however have been devoted to the analysis of incomplete rankings of arbitrary and variable size. Yet in many applications, observed rankings involve subsets of items: customers usually choose their products among the subsets they were presented, races/games involve different subsets of competitors in many racing/gaming competitions, users express their preferences only on a small part of the catalog of items. Ties might be present in these rankings, which do not thus have the exact form of incomplete rankings, but the greatest challenge in their analysis remains to handle the heterogeneity of their sizes and of the subsets of items they are related to. In many practical applications, the number $n$ of items can be very large, around $10^{4}$ say, adding a tremendous computational challenge to the mathematical formulation problem. Among parametric approach, the Plackett-Luce model is well-known to handle such incomplete rankings easily \citep[see for instance][]{CDH10, WL11}. By contrast, only the method introduced in \citet{LuB11} allows to use the Mallows model with incomplete rankings. Besides from that, we are only aware of three nonparametric approaches to handle incomplete rankings, namely those introduced in \citet{YLA02}, \citet{Kondor2010} and \citet{Sun2012} in order to perform tests, estimation and prediction respectively. The principles underlying these approaches are described at length in Subsection \ref{subsec:existing-approaches}.

\subsection{Our contributions}

In this article we introduce a novel general framework for the statistical analysis of incomplete rankings. Our contributions are both methodological and theoretical: we establish a new decomposition of functions of rankings that has a standalone interest and introduce a new approach to analyze ranking data based on this decomposition. Some of the results of this article are already proven in the unpublished manuscript \citet{CJS2014} or in the conference paper \citet{SCJ2015}, though formulated in a different manner. In any case, the present article is fully self-contained.
\begin{enumerate}
	\item We first define a rigorous setting for the statistical analysis of incomplete rankings accounting for their multi-scale structure. This includes a thorough discussion about the assumption of the existence of one single ranking model that explains all possible observations and the data generating process that produces the observations. We also clearly explicit the challenges of the statistical analysis of incomplete rankings and show why the existing approaches either do not overcome them or rely on restrictive assumptions.
	\item Exploiting recent results from algebraic topology, we establish the construction of the MRA representation, a novel decomposition for functions of incomplete rankings that fits with the natural multi-scale structure of incomplete rankings. We detail its multiresolution interpretation and show its strong localization properties.
	\item We use the MRA representation to define the MRA framework for the statistical analysis of incomplete rankings. It provides a general method to tackle many statistical problem on a dataset composed of incomplete rankings. As it uses the MRA representation, it naturally overcomes the challenges aforementioned and at the same time offers a general and flexible sandbox to design many procedures. 
	\item Finally we establish several connections between the MRA representation and other mathematical constructions on rankings or permutations. In particular we explain that the MRA representation decomposes rank information into pieces of ``relative rank information'' whereas $\Sn$-based harmonic analysis decomposes rank information into pieces of ``absolute rank information'', and highlight the relationship between these two decompositions.
\end{enumerate}

In statistical signal and image processing, novel harmonic analysis tools such as wavelet bases and their extensions have completely revitalized structured data analysis these last decades and lead to sparse representations and efficient algorithms for a wide variety of statistical tasks: estimation, prediction, denoising, compression, clustering, etc.  Directly inspired by the seminal contributions of P. Diaconis, where harmonic analysis tools have been first used to analyze ranking data, we believe that the MRA representation introduced in this paper may lead to a novel and powerful way of processing ranking data, in the same way as recent advances in computational harmonic analysis produced successful methods for high-dimensional data analysis. As will be seen throughout the paper, even if the analogy with MRA on the real line and standard wavelet theory has its limitations, it sheds light onto the rationale of our proposal.

\subsection{Related work}

As we have previously tried to give an overview of the general ranking literature and the existing approaches for the statistical analysis of incomplete rankings are recalled in Subsection \ref{subsec:existing-approaches}, we focus here on contributions that inspired the present work, harmonic and multiresolution analysis playing an important role.

Harmonic analysis for rankings was introduced in the seminal contributions \citet{Diaconis1988} and \citet{Diaconis89}, and then developed in several contributions \citep[see for instance][]{Clausen1993, Maslen1998, HGG09, Kondor2010, Irurozki2011, Kakarala2011}.  Its principle is to decompose functions of rankings into projections onto subspaces that are invariant under $\Sn$-based translations (see  Subsection \ref{subsec:background-harmonic-analysis} for the details), computed with the representations of the symmetric group. It has been applied with success to full and partial rankings, but it is by nature not fitted for the analysis of incomplete rankings. As shall be seen below, the MRA representation we introduce decomposes instead functions of rankings into projections that localize the effects of specific items, and has a natural multiresolution interpretation.

Our work is of course inspired by the first multiresolution analysis constructed for rankings, introduced in \citet{Kondor2012}. The latter provides a decomposition of functions on the symmetric group that refines in some way that of $\Sn$-based harmonic analysis as it allows to localize the effects of items inside the projections onto invariant subspaces. Its tree-structure however induces that the projections localize information conditioned upon those of lower scale, and does not fit with the multi-scale structure of subsets of items. More generally, several constructions for multiresolution analysis on discrete data have been introduced in the literature, see for instance \citet{Coifman06}, \citet{Gavish2010}, \citet{Hammond2011}, \citet{RG2013} or \citet{KTG14}. Though they all constitute great sources of inspiration, none of them leads to the MRA representation we introduce. The latter indeed has a different mathematical nature and involves objects from algebraic topology.

The HodgeRank framework is the first to use tools from algebraic topology for the purpose of ranking analysis. It was introduced in \citet{JLYY11} and then developed in several contributions such as \citet{XHJYLY12}, \citet{DSS2012} or \citet{OBO13}. Its principle also relies on decomposing a function of rankings into as a sum of meaningful projections. This decomposition is different from the MRA representation but some connection exists in particular cases, which we detail in Subsection \ref{subsec:pairwise-comparisons}. The HodgeRank framework however only applies to pairwise comparisons, whereas the MRA representation does to incomplete rankings of any size.

\subsection{Outline of the paper}

The paper is organized as follows:
\begin{itemize}
	\item A rigorous setting for the statistical analysis of incomplete rankings is defined in Section \ref{sec:setting}. After describing the classic statistical problems on ranking data, we discuss in depth the ``consistency assumption'', which stipulates the existence of one ranking model to explain all observations, and propose a generic data generating process that produces incomplete ranking observations. Then we explain the statistical and computational challenges of the analysis of incomplete rankings, and show that the existing approaches either do not fully overcome them or rely on a strong assumption on the form of the data. We finish the section with a discussion about the impact of the observation design on the complexity of the analysis.
	\item In Section \ref{sec:MRA-representation} we introduce the notations, concepts and main properties of the MRA representation. The construction and the related proofs are postponed to Section \ref{sec:MRA-construction}. We develop at length the multiresolution interpretation and show that the MRA representation allows to characterize the solutions to linear systems that involve marginals on incomplete rankings. At last we describe a fast procedure to compute the representation (a ``fast wavelet transform'') and give bounds for its complexity.
	\item The MRA framework is introduced in Section \ref{sec:MRA-framework}. After characterizing the parameters of the ranking model that can be inferred from observations, we introduce a general method that uses the MRA representation to do it efficiently. Several examples are displayed in order to show how this method can be combined with other procedures to tackle many statistical problem involving incomplete rankings. Then we demonstrate how this method naturally overcomes the statistical and computational challenges, while still offering many possibilities of fine-tuning and combinations with other methods.
	\item Section \ref{sec:MRA-construction} mainly contains the construction of the MRA representation and the proofs of the properties claimed in Section \ref{sec:MRA-representation}. It also provides some more insights about why the embedding operator used make the construction work whereas a classic, more intuitive, embedding operator would not make it work.
	\item In Section \ref{sec:connections} we establish several connections between the MRA representation and other mathematical constructions. The connection with $\Sn$-based harmonic analysis in particular is treated in depth. We show why the latter can be considered to decompose rank information into pieces of ``absolute rank information'', whereas the MRA representation to decompose rank information into pieces of ``relative rank information'', and we establish a precise relationship between the two decompositions. At last we explicit the connection with card shuffling, generalized Kendall's tau distances, and in the particular case of pairwise comparisons with social choice theory and HodgeRank.
	\item At last, Section \ref{sec:discussion} is devoted to additional discussion and the description of possible directions for further research. Regularity assumptions and regularization procedures in the feature space of the MRA framework are discussed in depth. Other developments of the MRA framework are also considered: refinement of the MRA representation with a wavelet basis or generalization to the analysis of incomplete rankings with ties.
\end{itemize}
We are aware that the length of the present article may make it difficult to approach. This is why we propose the following reading guide:
\begin{itemize}
	\item The reader mainly interested in the statistical analysis of incomplete rankings may focus on Sections \ref{sec:setting}, \ref{sec:MRA-representation} and \ref{sec:MRA-framework}. They contain all the tools to apply the MRA framework to statistical problems of interest on incomplete rankings.
	\item The reader mainly interested in the MRA decomposition and its connection with other mathematical constructions may focus on Sections \ref{sec:MRA-representation}, \ref{sec:MRA-construction} and \ref{sec:connections}. They contain the main fundamental results of the article and provide for each one of them as much insight as possible.
\end{itemize}

\section{Setting for the statistical analysis of incomplete rankings}
\label{sec:setting}

We first introduce the general setting for the statistical analysis of incomplete rankings, formulating the main definitions and assumptions that will be used in the sequel. Here and throughout the article, a probability distribution on a finite set is identified with its probability mass function. For a set $E$ of finite cardinality $\vert E\vert <\infty$, we set $\Subsets{E} = \{ A\subset E \;\vert\; \vert A\vert \geq 2\}$ and denote by $\Space{E} = \{ f : E \rightarrow \mathbb{R} \}$ the linear space of real-valued functions on $E$. It is equipped with the canonic inner product $\left\langle f,g\right\rangle_{E} = \sum_{x\in E}f(x)g(x)$ and the associated Euclidean norm $\Vert\cdot\Vert_{E}$. The indicator function of a subset $S\subset E$ is denoted by $\1{S}$ in general and by $\delta_{x}$ when $S$ is the singleton $\{x\}$, in which case it is called a Dirac function. The indicator function of any event $\mathcal{E}$ is denoted by $\mathbb{I}\{\mathcal{E}\}$. The support of a function $f\in\Space{E}$ is the set $\supp(f) := \{x\in E \;\vert\; f(x)\neq 0\}$.

\subsection{Classic statistical problems}
\label{subsec:problems}

Most ranking applications correspond to unsupervised learning problems on the set $\mathcal{X}$ of all rankings on $\n$ of a certain type (for instance $\mathcal{X} = \Sn$ when data are assumed to be full rankings). One observes a dataset of $N$ rankings $\pi_{1}, \dots, \pi_{N}\in\mathcal{X}$ drawn IID from a probability distribution over $\mathcal{X}$ and seeks to recover some part of the structure of this statistical population. Although the problems of this type that have been considered in the literature are much too numerous to be listed in an exhaustive manner, we may mention the following ones. Some of them are specific to ranking data while others apply to usual vector data but their extension to ranking data requires the definition of specific concepts and the design of dedicated methods.
\begin{itemize}
	\item { \bf Estimation:} The goal is to estimate the probability distribution on $\mathcal{X}$ that generates the observations, either assuming a parametric form or through a nonparametric approach.
	\item { \bf Clustering:} The goal is to divide the statistical population of rankings into groups such that elements in a same cluster are more ``similar'' to each other than to those in other groups.
	\item { \bf Ranking aggregation:} The goal is to find one ranking that best ``summarizes'' the statistical population.
	\item { \bf Best $k$ items recovery:} The goal is to find the $k$ items in $\n$ that are ``the most preferred'' with respect to the statistical population.
	\item { \bf Prediction on a subset:} The goal is, for any subset of items, to find the ``best ranking'' of these items with respect to the statistical population.
	\item { \bf Hypothesis testing / rule mining:} The goal is to test some statistical hypothesis or to identify some logical rules that are ``mostly satisfied'' by the statistical population.
\end{itemize}
All these problems can be considered for a statistical population of full, partial or incomplete rankings (refer to Subsection \ref{subsec:general-method} for a more detailed description of some of them applied to incomplete rankings). In each case they require a different approach but they all rely on a common modeling assumption, which we call the \textit{consistency assumption}.

\subsection{Projectivity: the consistency assumption}
\label{subsec:consistency-assumption}

A \textit{ranking model} is a family of probability distributions that characterize the variability of a statistical population of rankings. In the case of full rankings, the statistical population is only composed of random permutations, and a ranking model reduces to one probability distribution $p$ over the symmetric group $\Sn$. But when one considers partial or incomplete rankings, they usually are of various types, and the global variability of the statistical population is characterized by a family of probability distributions, one over the rankings of each type. In the case of top-$k$ rankings for instance, the number $k$ usually varies from $1$ to $n-1$ between observations, and the global variability of the statistical population is characterized by a family $(P_{k})_{1\leq k\leq n-1}$ where for each $k\in\{1,\dots,n-1\}$, $P_{k}$ is a probability distribution over the set of $k$-tuples with distinct elements \citep[see][for instance]{BOB07}.

Incomplete rankings are rankings on subsets of items. The varying parameter in a statistical population of incomplete rankings is thus the subset of items involved in each ranking. Let us introduce some notations. For distinct items $a_{1},\dots,a_{k}$ with $2\leq k\leq n$, we simply denote the incomplete ranking $a_{1}\succ\dots\succ a_{k}$ by the expression $\pi = a_{1}\dots a_{k}$. Such an expression is called an \textit{injective word}, its content is the set $c(\pi) = \{a_{1},\dots, a_{k}\}$ and its length or size is the number $\vert\pi\vert = k$. The rank of the item $i\in c(\pi)$ in the ranking $\pi$ is denoted by $\pi(i)$. We denote by $\Gamma_{n}$ the set of all incomplete rankings on $\n$ and by $\Rank{A} = \{ \pi\in\Gamma_{n} \;\vert\; c(\pi) = A\}$ the set of incomplete rankings with content $A$, for any $A\in\Subsets{\n}$. Notice that $\Rank{\n}$ corresponds to $\Sn$ and that $\Gamma_{n} = \bigsqcup_{A\in\Subsets{\n}}\Rank{A}$. Equipped with these notations, a ranking model for incomplete rankings is a family $(P_{A})_{A\in\Subsets{\n}}$ where for each $A\in\Subsets{\n}$, $P_{A}$ is a probability distribution over the set $\Rank{A}$ of rankings on $A$.

\begin{example}
For $n = 3$,
\begin{equation*}
\begin{array}{ccccccccc}
\Gamma_{3} 	& = & \{12, 21\} 	 &\sqcup & \{13, 31\} 		& \sqcup & \{23, 32\} 		&\sqcup & \{123, 132, 213, 231, 312, 321\}\\
\vspace{0.05cm}\\
			&	& \Rank{\{1,2\}} & 	 	 & \Rank{\{1,3\}} 	& 		 & \Rank{\{2,3\}} 	&		& \Rank{\{1,2,3\}} \equiv \Sym{3}
\end{array}
\end{equation*}
\end{example}

If there were no relationship between the different probability distributions of a ranking model, the statistical analysis of partial and/or incomplete rankings would boil down to independent analyses for each type of ranking. Yet one should be able to transfer information from the observation of one type of ranking to another. In a context of top-$k$ rankings analysis, if for instance item $a$ appears very frequently in top-$1$ rankings, it is natural to expect that it be ranked in high position in top-$k$ rankings with larger values of $k$, and reciprocally, if it is usually ranked high in top-$k$ rankings, then its probability of being top-$1$ should be high. The same intuition holds for incomplete rankings. If item $a$ is usually preferred to item $b$ in pairwise comparisons then rankings on $\{a,b,c\}$ that place $a$ before $b$ should have higher probabilities than the others. Reciprocally if such rankings appear more frequently than the others, then item $a$ should be preferred to item $b$ with high probability in a pairwise comparison.

The ranking literature thus relies on one fundamental assumption: the observed rankings in a statistical population of interest are induced by full rankings drawn from a single probability distribution $p$ over $\Sn$ \citep[see for instance][]{Luce77}. Permutation $\sigma\in\Sn$ induces ranking $\prec$ or equivalently is a linear extension of ranking $\prec$ if for all $a,b\in\n$, $a\succ b \Rightarrow \sigma(a) < \sigma(b)$. The probability that a random permutation $\Sigma$ drawn from $p$ induces a ranking $\prec$ is thus equal to 
\begin{equation}
\label{eq:linear-extension}
\mathbb{P}\left[\Sigma\in \Sn(\prec)\right] = \sum_{\sigma\in \Sn(\prec)}p(\sigma),
\end{equation}
where $\Sn(\prec)$ is the set of linear extensions of $\prec$. The \textit{consistency assumption} then stipulates that the probability distributions of a ranking model are all given by Eq. \eqref{eq:linear-extension}, forming thus a projective family of distributions. For instance, the set of linear extensions of the top-$k$ ranking $a_{1}\succ\dots\succ a_{k}\succ\textit{the rest}$, where $k\in\{1,\dots,n-1\}$ and $a_{1},\dots,a_{k}$ are distinct items in $\n$, is equal to $\{\sigma\in\Sn \;\vert\; \sigma^{-1}(1) = a_{1},\dots,\sigma^{-1}(k) = a_{k} \}$. The probability $P_{k}(a_{1},\dots,a_{k})$ is thus given by 
\[
P_{k}(a_{1},\dots, a_{k}) = \mathbb{P}\left[\Sigma^{-1}(1) = a_{1}, \dots, \Sigma^{-1}(k) = a_{k}\right] = \sum_{\substack{\sigma\in\Sn \\ \sigma^{-1}(1) = a_{1},\dots,\,\sigma^{-1}(k) = a_{k}}}p(\sigma).
\]
A permutation $\sigma$ induces an incomplete ranking $\pi$ on $A\in\Subsets{\n}$ if it ranks the items of $A$ in the same order as $\pi$, that is if $\sigma(\pi_{1}) < \dots < \sigma(\pi_{\vert\pi\vert})$. More generally, we say that word $\pi^{\prime}$ is a subword of word $\pi$ if there exist indices $1\leq i_{1} < \dots < i_{\vert\pi^{\prime}\vert} \leq \vert\pi\vert$ such that $\pi^{\prime} = \pi_{i_{1}}\dots\pi_{i_{\vert\pi^{\prime}\vert}}$, and we write $\pi^{\prime}\subset\pi$. Hence, permutation $\sigma$ induces ranking $\pi$ if and only if $\pi\subset\sigma$. In addition, it is clear that for a word $\pi\in\Gamma_{n}$ and a subset $A\in\Subsets{c(\pi)}$, there exists a unique subword of $\pi$ of content $A$. We denote it by $\pi_{\vert A}$ and call it the induced ranking of $\pi$ on $A$. The set of linear extensions of a ranking $\pi\in\Rank{A}$ with $A\in\Subsets{\n}$ is then $\Sn(\pi) = \{\sigma\in\Sn \;\vert\; \pi\subset\sigma\} = \{\sigma\in\Sn \;\vert\; \sigma_{\vert A} = \pi \}$ and the probability $P_{A}(\pi)$ is given by
\begin{equation}
\label{eq:consistency-assumption} 
P_{A}(\pi) = \mathbb{P}\left[\Sigma(\pi_{1}) < \dots < \Sigma(\pi_{\vert\pi\vert})\right] = \sum_{\sigma\in\Sn(\pi)}p(\pi) = \sum_{\substack{\sigma\in\Sn\\ \pi\subset\sigma}}p(\sigma) = \sum_{\substack{\sigma\in\Sn\\ \sigma_{\vert A} = \pi}}p(\sigma).\tag{$\ast$}
\end{equation}

\begin{example}
Let $n = 3$. For $\sigma = 231$, one has $\sigma_{\vert\{1,2\}} = 21$, $\sigma_{\vert\{1,3\}} = 31$ and $\sigma_{\vert\{2,3\}} = 23$. For $A = \{1,3\}$ and $\pi = 31$, one has
\[
P_{\{1,3\}}(31) = \mathbb{P}\left[\Sigma(3) < \Sigma(1)\right] = p(231) + p(321) + p(312).
\]
\end{example}

We call Eq. \eqref{eq:consistency-assumption} the \textit{consistency assumption} for the statistical analysis of incomplete rankings. It implies that all the $P_{A}$'s in the ranking model are \textit{marginal distributions} of the same probability distribution $p$ over $\Sn$. Abusively, $p$ is also called the ranking model in the sequel. 

We also extend the definition of a marginal to any function of incomplete rankings. As $\Gamma_{n} = \bigsqcup_{A\in\Subsets{\n}}\Rank{A}$, we embed all the spaces $\Space{\Rank{A}}$ into $\Space{\Gamma_{n}}$, identifying a function $F$ on $\Rank{A}$ to the function $f$ on $\Gamma_{n}$ equal to $F$ on $\Rank{A}$ and to $0$ outside $\Rank{A}$. One thus has $\Space{\Gamma_{n}} = \bigoplus_{A\in\Subsets{\n}}\Space{\Rank{A}}$. We then define $M_{A} : \Space{\Gamma_{n}} \rightarrow \Space{\Rank{A}}$, the \textit{marginal operator} on $A\in\Subsets{\n}$, for any $f\in \Space{\Gamma_{n}}$ by
\begin{equation}
\label{eq:marginal-operator}
M_{A}f(\pi) = \sum_{\sigma\in\Gamma_{n},\ \pi\subset\sigma}f(\sigma) \qquad\text{for }\pi\in\Rank{A}.
\end{equation}
In particular, $M_{A}p = P_{A}$ for all $A\in\Subsets{\n}$ and $M_{A}f = 0$ if $f\in \Space{\Rank{B}}$ with $A\not\in\Subsets{B}$.

\subsection{Probabilistic setting for the observation of incomplete rankings}

A dataset of full rankings is naturally modeled as a collection of random permutations $(\Sigma_{1},\dots,\Sigma_{N})$ drawn IID from a ranking model $p$. The latter thus fully characterizes the statistical population as well as its observation process. This property does not hold true in the statistical analysis of incomplete rankings, where the ranking model characterizes the statistical population, but it does not entirely characterize the generating process of this population. More specifically, it characterizes the variability of the observations on each subset of items $A\in\Subsets{\n}$, but it does not account for the variability of the observed subsets of items. 

\begin{example}
A ranking model $p$ for incomplete rankings on $\set{3}$ induces the probability distributions $P_{\{1,2\}}$, $P_{\{1,3\}}$, $P_{\{2,3\}}$ and $P_{\{1,2,3\}} = p$. For each $A\in\Subsets{\set{3}}$, a random ranking on $A$ can thus be drawn from the probability distribution $P_{A}$. But the $P_{A}$'s do not induce a probability distribution on $\Subsets{\set{3}}$ that would generate the samplings of the subsets $A$.
\end{example}

To model this double variability, we represent the observation of an incomplete ranking by a couple of random variables $(\mathbf{A},\Pi)$, where $\mathbf{A}\in\Subsets{\n}$ is the observed subset of items and $\Pi\in\Rank{A}$ is the observed ranking \textit{per se} on this subset of items. Let $\nu$ be the distribution of $\mathbf{A}$ over $\Subsets{\n}$. A dataset of incomplete rankings is then a collection $((\mathbf{A}_{1},\Pi^{(1)}),\dots,(\mathbf{A}_{N},\Pi^{(N)}))$ of IID samples of $(\mathbf{A},\Pi)$ drawn from the following process:
\begin{equation}
\label{eq:scheme}
\mathbf{A}\sim\nu\qquad\text{then}\qquad \Pi\vert (\mathbf{A} = A) \sim P_{A},
\end{equation}
where, for $\mathbf{X}$ and $P$ respectively a random variable and a probability distribution on a measurable space $\mathcal{X}$, $\mathbf{X}\sim P$ means that $\mathbf{X}$ is drawn from $P$. The interpretation of probabilistic setting \eqref{eq:scheme} is that first the subset of items $\mathbf{A}\in\Subsets{\n}$ is drawn from $\nu$ and then the ranking $\Pi\in\Rank{A}$ is drawn from $P_{A}$. It can be reformulated by exploiting the consistency assumption \eqref{eq:consistency-assumption}. The latter stipulates that for $A\in\Subsets{\n}$, the distribution of the random variable $\Pi$ on $\Rank{A}$ is the same as that of the induced ranking $\Sigma_{\vert A}$ of a random permutation $\Sigma$ drawn from $p$. A drawing of $(\mathbf{A},\Pi)$ can thus be reformulated as 
\begin{equation}
\label{eq:censoring-process}
\Sigma\sim p\qquad\text{then}\qquad\mathbf{A}\sim\nu\qquad\text{and}\qquad\Pi = \Sigma_{\vert \mathbf{A}}.
\end{equation}
Reformulation \eqref{eq:censoring-process} leads to the following interpretation: first a random permutation $\Sigma\in\Sn$ is drawn from $p$ then the subset of items $\mathbf{A}$ is drawn from $\nu$ and the ranking $\Pi$ is set equal to $\Sigma_{\vert \mathbf{A}}$. The permutation $\Sigma$ can then be seen as a latent variable that expresses the full preference of a user in the statistical population but its observation is censored by $\mathbf{A}$. We point out that this interpretation motivates the broader probabilistic setting introduced in \citet{Sun2012}. The authors model more generally the observation of any partial and/or incomplete ranking as the drawing of a latent random permutation $\Sigma$ from $p$ followed by a censoring process that \textit{can depend} on $\Sigma$. In the context of incomplete rankings observation, their probabilistic setting can be defined as: $\Sigma\sim p$ then $\mathbf{A}\sim\nu_{\Sigma}$ and $\Pi = \Sigma_{\vert \mathbf{A}}$, where $\nu_{\sigma}$ is a probability distribution over $\Subsets{\n}$ for each $\sigma\in\Sn$. Probabilistic setting \eqref{eq:scheme} fits into this broader one by setting all distributions $\nu_{\sigma}$ equal to $\nu$ or, equivalently, assuming that $\Sigma$ and $\mathbf{A}$ are \textit{independent}.

The independence of between $\Sigma$ and $\mathbf{A}$ corresponds to the \textit{missing at random assumption} in the general context of learning from incomplete data \citep[see][]{GJ95}. This assumption is not realistic in all situations, particularly in settings where the users choose the items on which they express their preferences, their choices being naturally biased by their tastes \citep[see][for instance]{MZRS07}. It remains however realistic in many situations where the subset of items proposed to the user is determined by the context: the available items in stock in a specific store or the possible recommendations in a specific area for instance. This assumption is thus made in many contributions of the literature \citep[see for instance][]{LuB14, RA14, DIS15}. 

\begin{remark}
We maintain furthermore that making a dependence assumption is incompatible with the principle of the statistical analysis of incomplete rankings. Indeed, the purpose of assuming that $\mathbf{A}$ and the latent variable $\Sigma$ are not independent is to infer from the observation of $\Pi = \Sigma_{\vert\mathbf{A}}$ some more information on $\Sigma$ than just $\Sigma_{\vert\mathbf{A}}$. For instance, to model the fact that the expression of a user's preferences would be biased by her tastes, one can assume that the full ranking $\Sigma$ is censored to items that have a low expected rank (meaning that they have a high probability to be ranked in the first positions). The subset of items $\mathbf{A}$ could then be obtained by sampling items without replacement from a distribution over $\n$ of the form $\eta_{\sigma}(i) \varpropto e^{-\alpha\sigma(i)}$, where $\alpha\in\mathbb{R}$ is a spread parameter, conditioned upon $\Sigma = \sigma$. The observed ranking $\Pi$ on a subset $A=\{a_{1},\dots,a_{k}\}\in\Subsets{\n}$ would then not only provide information on the relative ordering $\Sigma_{\vert A}$ of the items of $A$ but even more on their absolute ranks $(\Sigma(a_{1}),\dots,\Sigma(a_{k}))$ in the latent full ranking $\Sigma$. Exploiting this additional information requires to analyze $\Pi$ as a partial ranking. Thus it cannot be done in a setting of statistical analysis of incomplete rankings.
\end{remark}

\subsection{Challenges of the statistical analysis of incomplete rankings}
\label{subsec:challenges}

The general setting for the statistical analysis of incomplete rankings is now formalized. For any application mentioned in subsection \ref{subsec:problems}, we suppose we observe a dataset $\mathcal{D}_{N} = ((\mathbf{A}_{1},\Pi^{(1)}),\dots,(\mathbf{A}_{N},\Pi^{(N)}))$ of $N$ incomplete rankings drawn IID from the process \eqref{eq:scheme}. The goal is then to recover some part of the ranking model $p$: it can be $p$ itself or only the marginals involved in the generation of the dataset (estimation), the partition of $\Gamma_{n}$ that best fits with $p$ (clustering), the ranking that best summarizes $p$ (ranking aggregation), the $k$ items that are most preferred according to $p$, the best ranking $\pi\in\Rank{A}$ according to $P_{A}$ for any subset of items $A\in\Subsets{\n}$, some logical rules that are highly probable, such as $a\succ b \Rightarrow c\succ d$ if the mutual information of the events $\{a\succ b\}$ and $\{c \succ d\}$ is high (rule mining / hypothesis testing).

\begin{remark}
We point out that it may be desirable to estimate the probability distribution $\nu$ in addition to recovering some parts of $p$. This problem can however be treated independently and is thus not addressed in the present paper. The distribution $\nu$ remains however the censoring process that generates the design of observations and therefore has a major impact on the parts of $p$ that can be inferred from the dataset $\mathcal{D}_{N}$. A deeper analysis of the impact of $\nu$ is provided in Subsection \ref{subsec:impact-of-nu}.
\end{remark}

Characterizing separately the variability of the observed subset of item $\mathbf{A}$ leads to an unexpected analogy with supervised learning: in the couple $(\mathbf{A}, \Pi)$, the subset $\mathbf{A}$ can be seen as an input generated by the distribution $\nu$ and the ranking $\Pi$ can be seen as the output generated by the ranking model $p$ given the input $\mathbf{A}$. Analyzing incomplete rankings data thus requires to face two classical issues in statistical learning, which can be easily formulated in the context of binary classification, the flagship problem in machine-learning theory.
\begin{itemize}
	\item {\bf Consolidate knowledge on already seen subsets of items.} For an observed subset of items $A\in\Subsets{\n}$, one must consolidate all the observations on $A$ in order to recover a maximum amount of information about $P_{A}$. The corresponding task in binary classification is to consolidate all the outputs $y$ for a given input $x$ (or very close inputs) that was observed many times, where $x$ and $y$ are the values taken by IID samples of a random couple $(X,Y)$. Its difficulty depends on how much the value $\mathbb{P}\left[ Y = 1 \vert X = x\right]$ is close to $1/2$: the closer the more difficult. Analogously, the difficulty of consolidating observations on a given subset of items $A$ depends on the complexity of the marginal $P_{A}$. If $P_{A}$ is a Dirac function, it is easy to recover. If $P_{A}$ is more complex, its recovery is more challenging.
	\item {\bf Transfer knowledge to unseen subsets of items.} For a new unseen subset of items, one needs to transfer a maximum amount of acquired information from the observed subsets of items. In binary classification, one faces an analogous problem when trying to predict the output $y$ related to an input value $x$ never observed before and potentially far from all previously observed inputs. The difficulty of this task then depends on the ``regularity'' of the function $\eta : x \mapsto \mathbb{P}\left[Y = 1 \vert X = x\right]$: it is easier to infer the value of $\mathbb{P}\left[Y = 1 \vert X = x\right]$ for an unobserved $x$ when $\eta$ is ``regular'', in the sense that $\eta(x)$ does not vary unexpectedly when $x$ varies. Similarly for incomplete rankings, it is easier to transfer information to an unobserved subset of items $A$ when the function $B\mapsto P_{B}$ does not vary unexpectedly when $B$ varies in $\Subsets{\n}$.
\end{itemize}
These two tasks require to cope with two different sources of variability and can be tackled independently in a theoretical setting. But in a statistical setting, they must be handled simultaneously in order to best reduce the sampling noise of a dataset $\mathcal{D}_{N}$. It is better indeed to transfer between subsets information that has been consolidated on each subset and conversely, it is better to consolidate information on a subset with information transferred from other subsets. A major difficulty however remains: incomplete rankings are heterogeneous. They can have different sizes and for a given size they can be observed on different subsets of items. Consolidating and transferring information for incomplete rankings is thus far from being obvious, and represents the main challenge of the statistical analysis of incomplete rankings.

\begin{example}
Let $n = 4$ and assume that one observes rankings on $\{1,3\}$, $\{1,3,4\}$ and $\{2,4\}$. Information could be consolidated on each of these three subsets independently and then transferred to unobserved subsets. It would certainly be more efficient however to consolidate information on these subsets simultaneously, transferring at the same time information between them. The question is now to find a way to achieve this.
\end{example}

The consistency assumption \eqref{eq:consistency-assumption} defines the base structure to transfer information between subsets of items. Namely for two subsets of items $A,B\in\Subsets{\n}$ with $B\subset A$, it stipulates that $M_{B}P_{A} = P_{B}$. The knowledge of $P_{A}$ thus implies the knowledge of $P_{B}$. Information must therefore be transferred from $A$ to $B$ through the marginal operator $M_{B}$. The condition is a slightly more subtle in the other direction: information must be transferred from $B$ to $A$ through the constraint on $P_{A}$ to satisfy $M_{B}P_{A} = P_{B}$. Hence, the knowledge of $P_{B}$ does not imply the knowledge of $P_{A}$, but it provides some part of it. More generally, the knowledge of any marginal $P_{A}$ provides some information on $p$ through the constraint $M_{A}p = P_{A}$. How to transfer information from $A$ to a subset $C$ such that neither $C\subset A$ nor $A\subset C$ is however a priori unclear. 

\begin{example}
Coming back to the previous example, information on $\{1,3,4\}$ should be used to consolidate information on $\{1,3\}$ through the relation $M_{\{1,3\}}P_{\{1,3,4\}} = P_{\{1,3\}}$. Information on $\{1,3\}$ should be used to enforce a constraint in consolidating information on $\{1,3,4\}$ through the same relationship. Information on each subset can be used to enforce a constraint on the global ranking model $p$. It is however unclear if or how information should be transferred between $\{2,4\}$ and $\{1,3\}$ or $\{1,3,4\}$.
\end{example}

In addition to this major statistical challenge, practical applications also raise a great computational challenge. The analysis of incomplete rankings always involve at some point the computation of a marginal of a ranking model. Performed naively using the definition \eqref{eq:marginal-operator}, the computation of $M_{A}p(\pi)$ for $A\in\Subsets{\n}$ and $\pi\in\Rank{A}$ requires $n!/\vert A\vert !$ operations. This is by far intractable in practical applications where $\vert A\vert$ is around $10$ and $n$ is around $10^{4}$.

\subsection{Limits of existing approaches}
\label{subsec:existing-approaches}

We now review the existing approaches in the literature for the statistical analysis of incomplete rankings and outline their limits.\\

\noindent
{\bf Parametric models.} The most widely used approaches rely on parametric modeling. One considers a family of models $\{p_{\theta} \;\vert\; \theta\in\Theta\}$, where $\Theta$ is a parameter space, and assumes that $p = p_{\theta^{\ast}}$ for a certain $\theta^{\ast}\in\Theta$. The goal is then to recover $\theta^{\ast}$ from the dataset $\mathcal{D}_{N}$. One standard method is to take the parameter that maximizes the likelihood of the model on the dataset. For $\theta\in\Theta$, let $\mathbb{P}_{\theta}$ be the distribution of a random permutation $\Sigma$ corresponding to the ranking model $p_{\theta}$, that is to say the distribution defined by $\mathbb{P}_{\theta}\left[\Sigma\in S\right] = \sum_{\sigma\in S}p_{\theta}(\sigma)$ for any subset $S\subset\Sn$. The relevance of a model candidate $p_{\theta}$ on the dataset $\mathcal{D}_{N}$ is thus measured through the conditional likelihood
\[
\mathcal{L}(\theta \vert \mathbf{A}_{1},\dots, \mathbf{A}_{N}) = \prod_{i=1}^{N}\mathbb{P}_{\theta}\left[\Sigma_{\vert\mathbf{A}_{i}} = \Pi^{(i)}\right] = \prod_{i=1}^{N}M_{\mathbf{A}_{i}}p_{\theta}\left(\Pi^{(i)}\right).
\]
One then compute $\widehat{\theta}_{N} = \argmax_{\theta\in\Theta}\mathcal{L}(\theta \vert \mathbf{A}_{1},\dots, \mathbf{A}_{N})$ exactly or approximately and uses the ranking model $\widehat{p}_{N} := p_{\widehat{\theta}_{N}}$. In this approach, the consolidation of information is performed implicitly through the selection of the ranking model from the family $\{p_{\theta} \;\vert\; \theta\in\Theta\}$ that best explains the data. It is then transferred to any subset of items $B\in\Subsets{\n}$ through the marginal $M_{B}\widehat{p}_{N}$. The computational challenge is easily overcome when using the Plackett-Luce model because the marginals of the latter have a closed-form expression. It is much less straightforward for the Mallows model, but a dedicated method was introduced in \citet{LuB11}. From a global point of view, approaches based on a parametric model have the advantage to offer a simple framework for all applications of the statistical analysis of incomplete rankings. Their major drawback however is to rely on a very rigid assumption on the form of the ranking model, which is rarely satisfied in practice.\\

\noindent
{\bf Nonparametric methods based on identifying an incomplete ranking with the set of its linear extensions.} The three nonparametric methods introduced in the literature to analyze incomplete rankings all face the heterogeneity of incomplete rankings the same way: they represent an incomplete ranking $\pi\in\Gamma_{n}$ by the set of its linear extensions $\Sn(\pi)\subset\Sn$. \citet{YLA02} generalize a distance $d$ on $\Sn$ to a distance $d^{\ast}$ on $\Gamma_{n}$ by setting $d^{\ast}(\pi,\pi')$ proportional to $\sum_{\sigma\in\Sn(\pi)}\sum_{\sigma'\in\Sn(\pi')}d(\sigma,\sigma')$ for two incomplete rankings $\pi,\pi'\in\Gamma_{n}$ and use it to perform statistical tests. In \citet{Sun2012}, the Kendall's tau distance\footnote{The Kendall's tau distance on $\Sn$ is defined as the number of pairwise disagreements between two permutations: $d(\sigma,\sigma') = \sum_{1\leq i < j \leq n}\mathbb{I}\{\sigma_{\vert\{i,j\}}\neq\sigma'_{\vert\{i,j\}}\}$.} is generalized in the same way and then used to define a kernel-based estimator of $p$. Finally, \citet{Kondor2010} define kernels on $\Gamma_{n}$ based, for two incomplete rankings $\pi,\pi'\in\Gamma_{n}$, on the Fourier transform of the indicator functions of the sets $\Sn(\pi)$ and $\Sn(\pi')$. Broadly speaking, these three approaches transfer information between different incomplete rankings through a given \textit{similarity measure} between their sets of linear extensions. They overcome some part of the computational challenge through explicit simplifications of the extended distance $d^{\ast}$ or the Fourier transform of the indicator function of an incomplete ranking. They are however fundamentally biased. To best illustrate this point, let us consider the following estimator:
\begin{equation}
\label{eq:biased-empirical-estimator}
\widehat{p}_{N} = \frac{1}{N}\sum_{i=1}^{N}\frac{\vert\mathbf{A}_{i}\vert!}{n!}\1{\Sn(\Pi^{(i)})}.
\end{equation}
It corresponds to the natural empirical estimator of $p$ when one represents an incomplete ranking by the set of its linear extensions. In this representation indeed, one considers that the observation of an incomplete ranking $\Pi$ indicates that the underlying permutation $\Sigma$ should belong to $\Sn(\Pi)$. The amount of knowledge about $\Sigma$ is thus modeled by the uniform distribution on $\Sn(\Pi)$. The estimator $\widehat{p}_{N}$ is then the average of the uniform distributions over the sets $\Sn(\Pi^{(i)})$ for $i\in\{1,\dots,N\}$. As stated in the following proposition, it is always strongly biased, except in a few specific situations, irrelevant in practice.

\begin{proposition}
\label{prop:bias}
Let $N\geq 1$ and $\widehat{p}_{N}$ be the estimator defined by equation \eqref{eq:biased-empirical-estimator}. Then for any $\sigma\in\Sn$,
\[
\mathbb{E}\left[\widehat{p}_{N}(\sigma)\right] = \sum_{\sigma'\in\Sn}\left(\sum_{A\in\Subsets{\n} }\nu(A)\frac{\vert A\vert !}{n!}\mathbb{I}\{\sigma'_{\vert A} = \sigma_{\vert A}\}\right)p(\sigma').
\]
\end{proposition}

\begin{proof}
Using the reformulation \eqref{eq:censoring-process} of the data generating process producing the observations, one has for any $\sigma\in\Sn$
\[
\mathbb{E}\left[\widehat{p}_{N}(\sigma)\right] = \frac{1}{N}\sum_{i=1}^{N}\mathbb{E}\left[\frac{\vert\mathbf{A}_{i}\vert!}{n!}\mathbb{I}\{\Sigma_{\vert\mathbf{A}_{i}} = \sigma_{\vert\mathbf{A}_{i}}\}\right] = \sum_{A\in\Subsets{\n}}\nu(A)\frac{\vert A\vert !}{n!}\sum_{\sigma'\in\Sn}p(\sigma')\mathbb{I} \{ \sigma'_{\vert A} = \sigma_{\vert A} \}.
\]
A simple sum inversion concludes the proof.
\end{proof}

Proposition \ref{prop:bias} says that unless $p$ is a Dirac distribution (which is a too restrictive assumption) or $\nu$ is solely concentrated on $\n$ (which boils down to statistical analysis on full rankings), $\mathbb{E}\left[\widehat{p}_{N}(\sigma)\right]$ is fundamentally different from $p(\sigma)$ for $\sigma\in\Sn$.

\begin{example}
Let $n = 4$ and $\nu$ with support $\{ \{1,3\}, \{2,4\}, \{1,3,4\} \}$. Then for any $N\geq 1$,
\begin{align*}
\mathbb{E}\left[\widehat{p}_{N}(2134)\right] &= \frac{\nu(\{1,3\})}{12}\Big[p(2413) + p(4213) + p(2134) + p(4132) + p(1324) + p(1342)\Big]\\
&+ \frac{\nu(\{2,4\})}{12}\Big[ p(1324) + p(3124) + p(1243) + p(3241) + p(2413) + p(2431)\Big]\\
&+ \frac{\nu(\{1,3,4\})}{4}\Big[ p(2134) + p(1342) \Big].
\end{align*}
\end{example}
We point out that Proposition \ref{prop:bias} says more specifically that $\widehat{p}_{N}$ is actually an unbiased estimator of $T_{\nu}p$, where $T_{\nu}$ is the matrix of similarity defined by
\[
T_{\nu}(\sigma,\sigma') = \sum_{A\in\Subsets{\n}}\nu(A)\frac{\vert A\vert !}{n!}\mathbb{I}\{\sigma_{\vert A} = \sigma'_{\vert A}\}\qquad\text{for } \sigma,\sigma'\in\Sn,
\]
In particular, if $\nu$ is the uniform distribution over the pairs of $\n$, $T_{\nu}(\sigma,\sigma')$ simply reduces to an affine transform of the Kendall's tau distance between $\sigma$ and $\sigma'$.\\

\noindent
{\bf Learning as a regularized inverse problem.} A general framework for the statistical analysis of incomplete rankings could take the paradigmatic form of a regularized inverse problem. Assume first that one knows exactly some of the marginals of the ranking model $p$, for a collection of subsets $\mathcal{A}\subset\Subsets{\n}$. She could try to recover $p$ through the minimization problem
\begin{equation}
\label{eq:theoretic-inverse-problem}
\min_{\substack{q:\Sn\rightarrow\mathbb{R}\\ q\geq 0\\ \sum_{\sigma\in\Sn}q(\sigma) = 1}}\Omega(q) \qquad\text{subject to}\qquad M_{A}q = P_{A} \text{ for all } A\in\mathcal{A},
\end{equation}
where $\Omega$ is a penalty function that measures a certain level of regularity, so that $p$ should be a solution of \eqref{eq:theoretic-inverse-problem}. Information from the $P_{A}$'s would then be transferred to an unknown subset $B\in\Subsets{\n}$ through the computation of $M_{B}p^{\ast}$, where $p^{\ast}$ is an exact or approximate solution of \eqref{eq:theoretic-inverse-problem}. In a statistical setting, one cannot know exactly the marginals of $p$. The natural extension is then to consider the naive empirical estimator defined for an observed subset $A$ by
\begin{equation}
\label{eq:empirical-marginal}
\widehat{P_{A}}(\pi) = \frac{\vert\{1\leq i\leq N \;\vert\; \Pi^{(i)} = \pi \}\vert}{N_{A}} \qquad\text{for }\pi\in\Rank{A},
\end{equation}
where $N_{A}$ is the number of times that $A$ was observed in $\mathcal{D}_{N}$, and to consider the following generic minimization problem
\begin{equation}
\label{eq:empirical-inverse-problem}
\min_{\substack{q:\Sn\rightarrow\mathbb{R}\\ q\geq 0\\ \sum_{\sigma\in\Sn}q(\sigma) = 1}} \quad \sum_{\substack{A\in\Subsets{\n}\\ N_{A} > 0}}\frac{N_{A}}{N}\Delta_{A}\left( M_{A}q,\widehat{P_{A}} \right) \quad + \quad \lambda_{N}\Omega(q),
\end{equation}
where $\Delta_{A}$ is a dissimilarity measure between two probability distributions over $\Rank{A}$\footnote{One can take for instance an $l^{p}$ norm or the Kullback-Leibler divergence.} and $\lambda_{N}$ is a regularization parameter. Information is then simultaneously consolidated on the observed subsets into an exact or approximate solution $\widehat{p}_{N}$ and can then be transferred to unobserved subsets by computing $M_{B}\widehat{p}_{N}$. 

Though this approach is quite common in the machine learning literature, where $\Omega(q)$ typically enforces the sparsity of $q$ in a certain basis, it has been applied to the ranking literature only in a few contributions. In \citet{Jagabathula2011} for instance, the problem of recovering the ranking model $p$ from the observation of its first-order marginals $\mathbb{P}\left[\Sigma(i) = j\right]$ for $i\in\n$ and $j\in\{1,\dots,n\}$ is considered under a sparsity assumption over $\Sn$. A maximal entropy assumption is made in \citet{AS12} in order to recover the ranking model $p$ either from its first-order marginals or from its pairwise marginals $\mathbb{P}\left[a \succ b\right]$ for $a,b\in\n$, $a\neq b$.

In the setting of the statistical analysis of incomplete rankings, this approach has the advantage to allow less restrictive assumptions than parametric modeling and to avoid the bias of the aforementioned nonparametric approaches. It suffers however from a major drawback: it requires to compute the marginal operators. It is therefore inapplicable on practical datasets if this computation is performed naively through definition \eqref{eq:marginal-operator}.

\ \\

All the existing approaches follow the same two steps: first, information from the dataset $\mathcal{D}_{N}$ is consolidated and transferred into a ranking model $\widehat{p}_{N}\in \Space{\Sn}$. Then it can be transferred to any subset of items $B\in\Subsets{\n}$ through the marginal $M_{B}\widehat{p}_{N}$. This is of course the most natural method to exploit the consistency assumption \eqref{eq:consistency-assumption} and try to overcome the statistical challenge of the analysis of incomplete rankings. It does not provide however any help to overcome the challenge of the computation of the marginal. This is why each approach requires a specific trick to be applicable. 

The MRA framework we introduce in this paper provides a general method to handle both the computational and statistical challenges of the analysis of incomplete rankings. Instead of consolidating information into a ranking model $\widehat{p}_{N}\in \Space{\Sn}$, observations are first represented into a feature space, which we call the MRA representation, fitted to exploit the consistency assumption and to compute the marginal operator efficiently. The framework then provides many possibilities to consolidate and transfer information in this feature space. The MRA representation is entirely model-free, it simply arises from the natural multi-scale structure of the marginal operators and its algebraic and topological properties. Before we describe its objects and properties in details, we present a brief analysis of the impact of the probability distribution $\nu$.

\subsection{The impact of the observation design}
\label{subsec:impact-of-nu}

Depending on the application, the probability distribution $\nu$ involved in the statistical process \eqref{eq:scheme} may or may not be known. In any case, as explained in Subsection \ref{subsec:challenges}, it is not the goal of the statistical analysis of incomplete rankings to learn it. It is rather seen as a parameter that adds some noise to the observations through the censoring process \eqref{eq:censoring-process}. It has nonetheless a direct impact on the complexity of the analysis, both on the statistical and computational points of view, especially through its support $\mathcal{A} = \{A\in\Subsets{\n} \;\vert\; \nu(A) > 0\}$, that we call the \textit{observation design}.

The impact of the distribution $\nu$ naturally occurs on the number of parameters required to store a dataset $\mathcal{D}_{N}$. Let $\widehat{\nu}_{N}$ be the empirical probability distribution over $\Subsets{\n}$ defined for $A\in\Subsets{\n}$ by $\widehat{\nu}_{N}(A) = N_{A}/N$, and let $\EOD = \{A\in\Subsets{\n} \;\vert\; N_{A} > 0\}$ be its support. Notice that one necessarily has $\EOD\subset\mathcal{A}$. A dataset $\mathcal{D}_{N}$ is then fully characterized by the probability distribution $\widehat{\nu}_{N}$ and the collection of empirical estimators $(\widehat{P_{A}})_{A\in\EOD}$. 

\begin{lemma}
	\label{lem:dataset-storage}
The number of parameters required to store the dataset $\mathcal{D}_{N}$ is upper bounded by $\min(N,\sum_{A\in\mathcal{A}}\vert A\vert!)$.
\end{lemma}

\begin{proof}
On the one hand, the number of parameters to store $\mathcal{D}_{N}$ is obviously bounded by $N$. On the other hand, the number of parameters required to store $\widehat{\nu}_{N}$ is $\vert\EOD\vert$, thus at most equal to $\vert\mathcal{A}\vert$. The number of parameters required to store $(\widehat{P_{A}})_{A\in\EOD}$ is equal to $\sum_{A\in\EOD}\vert\supp(\widehat{P_{A}})\vert$, thus at most equal to $\sum_{A\in\mathcal{A}}(\vert A\vert!-1)$. Summing these two quantities gives the desired result.
\end{proof}

The number $\min(N,\sum_{A\in\mathcal{A}}\vert A\vert!)$ given by Lemma \ref{lem:dataset-storage} is a measure of the ``complexity'' of the dataset $\mathcal{D}_{N}$, in the sense that any procedure that exploits all the information contained in $\mathcal{D}_{N}$ will necessarily require at least as many operations. Notice that the number $\sum_{A\in\mathcal{A}}\vert A\vert!$ is entirely characterized by $\mathcal{A}$, the observation design. It increases both with its ``spread'' $\vert\mathcal{A}\vert$ and its ``depth'' $K = \max_{A\in\mathcal{A}}\vert A\vert$, and is bounded by $\vert\mathcal{A}\vert\times K!$. In particular if $\mathcal{A} = \{A\in\Subsets{\n} \;\vert\; \vert A\vert \leq K\}$ then this bound is of order $O(K!\,n^{K})$. Figures \ref{fig:nu-example-1} and \ref{fig:nu-example-2} show examples of two different observation designs for $n = 5$ with the associated number $\sum_{A\in\mathcal{A}}\vert A\vert!$. The elements in $\mathcal{A}$ are in black whereas the elements of $\Subsets{\n}\setminus\mathcal{A}$ are in gray.

\begin{figure}[h!]
\centering
	\begin{tikzpicture}
	\node at (0,3) {\color{gray}$\{1,2,3,4,5\}$};
	\node at (-4,2) {\color{gray}$\{1,2,3,4\}$};
	\node at (-2,2) {\color{gray}$\{1,2,3,5\}$};
	\node at (0,2) {\color{gray}$\{1,2,4,5\}$};
	\node at (2,2) {\color{gray}$\{1,3,4,5\}$};
	\node at (4,2) {\color{gray}$\{2,3,4,5\}$};
	\node at (-6.75,1) {\color{gray}$\{1,2,3\}$};
	\node at (-5.25,1) {\color{gray}$\{1,2,4\}$};
	\node at (-3.75,1) {$\{1,2,5\}$};
	\node at (-2.25,1) {\color{gray}$\{1,3,4\}$};
	\node at (-0.75,1) {\color{gray}$\{1,3,5\}$};
	\node at (0.75,1) {\color{gray}$\{1,4,5\}$};
	\node at (2.25,1) {$\{2,3,4\}$};
	\node at (3.75,1) {\color{gray}$\{2,3,5\}$};
	\node at (5.25,1) {\color{gray}$\{2,4,5\}$};
	\node at (6.75,1) {\color{gray}$\{3,4,5\}$};
	\node at (-6.75,0) {$\{1,2\}$};
	\node at (-5.25,0) {$\{1,3\}$};
	\node at (-3.75,0) {$\{1,4\}$};
	\node at (-2.25,0) {\color{gray}$\{1,5\}$};
	\node at (-0.75,0) {$\{2,3\}$};
	\node at (0.75,0) {\color{gray}$\{2,4\}$};
	\node at (2.25,0) {$\{2,5\}$};
	\node at (3.75,0) {$\{3,4\}$};
	\node at (5.25,0) {$\{3,5\}$};
	\node at (6.75,0) {$\{4,5\}$};
	\end{tikzpicture}
	\caption{Example of an observation design $\mathcal{A}$ for $n = 5$, $\sum_{A\in\mathcal{A}}\vert A\vert ! = 28$}
	\label{fig:nu-example-1}
\end{figure}

\begin{figure}[h!]
	\centering
	\begin{tikzpicture}
	\node at (0,3) {\color{gray}$\{1,2,3,4,5\}$};
	\node at (-4,2) {\color{gray}$\{1,2,3,4\}$};
	\node at (-2,2) {$\{1,2,3,5\}$};
	\node at (0,2) {\color{gray}$\{1,2,4,5\}$};
	\node at (2,2) {\color{gray}$\{1,3,4,5\}$};
	\node at (4,2) {$\{2,3,4,5\}$};
	\node at (-6.75,1) {\color{gray}$\{1,2,3\}$};
	\node at (-5.25,1) {\color{gray}$\{1,2,4\}$};
	\node at (-3.75,1) {\color{gray}$\{1,2,5\}$};
	\node at (-2.25,1) {\color{gray}$\{1,3,4\}$};
	\node at (-0.75,1) {\color{gray}$\{1,3,5\}$};
	\node at (0.75,1) {$\{1,4,5\}$};
	\node at (2.25,1) {\color{gray}$\{2,3,4\}$};
	\node at (3.75,1) {\color{gray}$\{2,3,5\}$};
	\node at (5.25,1) {\color{gray}$\{2,4,5\}$};
	\node at (6.75,1) {\color{gray}$\{3,4,5\}$};
	\node at (-6.75,0) {\color{gray}$\{1,2\}$};
	\node at (-5.25,0) {\color{gray}$\{1,3\}$};
	\node at (-3.75,0) {\color{gray}$\{1,4\}$};
	\node at (-2.25,0) {\color{gray}$\{1,5\}$};
	\node at (-0.75,0) {\color{gray}$\{2,3\}$};
	\node at (0.75,0) {\color{gray}$\{2,4\}$};
	\node at (2.25,0) {\color{gray}$\{2,5\}$};
	\node at (3.75,0) {\color{gray}$\{3,4\}$};
	\node at (5.25,0) {\color{gray}$\{3,5\}$};
	\node at (6.75,0) {\color{gray}$\{4,5\}$};
	\end{tikzpicture}
	\caption{Example of an observation design $\mathcal{A}$ for $n = 5$, $\sum_{A\in\mathcal{A}}\vert A\vert ! = 54$}
	\label{fig:nu-example-2}
\end{figure}

The observation design also impacts the amount of information about the ranking model $p$ it gives access to. Indeed, if one makes a structural assumption on $p$ and seeks to recover some part of it from the observation of incomplete rankings drawn from \eqref{eq:scheme}, the complexity of this task will significantly depend on the interplay between $p$ and $\nu$, especially through the observation design $\mathcal{A}$. 

\begin{example}
\label{ex:toy-example}
As a toy example, consider the very simple case where one observes the exact induced rankings of one full ranking $\pi^{\ast}$ on $\set{5}$, on the subsets $\{1,2,3\}$, $\{3,4\}$ and $\{4,5\}$. The goal is then to recover the ranking model $p = \delta_{\pi^{\ast}}$ through the observation design $\mathcal{A} = \{\{1,2,3\}, \{3,4\}, \{4,5\} \}$. If $\pi^{\ast} = 12345$, then the observed induced rankings are $123$, $34$ and $45$. It happens that there is only one full ranking on $\set{5}$ that induces these three rankings, namely $12345$, and $\pi^{\ast}$ is recovered with certainty. Now, if $\pi^{\ast} = 24153$ for instance, the observed rankings are $213$, $43$ and $45$. In that case, there are twelve full rankings on $\set{5}$ that can induce these three rankings. The amount of information provided by these observations is therefore not sufficient for recovering $\pi^{\ast}$.
\end{example}

In a general context, one may assume that $p$ has a more general structure than a Dirac function on $\Rank{\n}$ and that the observations are made in the presence of a statistical noise. But the principle illustrated by the Example \ref{ex:toy-example} remains valid. Quantifying the amount of accessible information with respect to the interplay between $p$ and $\mathcal{A}$ is however not obvious because the latter is of a complex combinatorial nature. Some results have been provided in this sense in \citet{SPBRBW15}, when $p$ is assumed to be a Plackett-Luce or a Thurstone model and the observations are pairwise comparisons. When no structural assumption is made on $p$, the accessible information can be characterized exactly through the MRA representation described in the next section. The result is provided in \citet{SCJ2015} and recalled in the present paper by Theorem \ref{th:accessible-information} in Section \ref{sec:MRA-framework}.

\section{The MRA representation}
\label{sec:MRA-representation}

In this section, we introduce the MRA representation, describe its main properties and provide insight on how to interpret it. The terminology used here and throughout the article is borrowed from wavelet theory. Though it can appear peculiar at first reading, the analogy is explained at length in Subsection \ref{subsec:interpretation}.

\subsection{Definitions}
\label{subsec:definitions}

We first introduce the main objects of the MRA representation but postpone their explicit construction to Section \ref{sec:MRA-construction} for clarity. For any finite set $E$, we set $\SubsetsWE{E} := \Subsets{E}\cup\{\emptyset\}$.
\begin{itemize}
	\item {\bf Signal space.} The MRA representation applies to functions of incomplete rankings, which are seen as ``signals'' in order to borrow the language of standard MRA in wavelet theory (refer to \citet{Mallat}) for interpretation purpose. Let $\bar{0}$ denote by convention the unique injective word of content $\emptyset$ and length $0$. We set $\Gn:=\Gamma_{n}\cup\{\bar{0}\}$. Any space $\Space{\Rank{A}}$ for $A\in\SubsetsWE{\n}$ is seen as \textit{local} signal space and they are all embedded into the \textit{global} signal space defined by
	\[
	\Space{\Gn} = \bigoplus_{A\in\SubsetsWE{\n}}\Space{\Rank{A}}.
	\]
	Elements of the signal space are seen as collections of functions $F = (F_{A})_{A\in\SubsetsWE{\n}}$. The global support of an element $F$ is the set $\Supp(F) = \{A\in\SubsetsWE{\n} \;\vert\; F_{A} \neq 0 \}$, and we usually identify an element $F$ with the collection restricted to its global support $(F_{A})_{A\in\Supp(F)}$. We extend naturally the marginal operator to the space $\Space{\Gn}$ and define by convention the marginal operator on $\emptyset$ by $M_{\emptyset}:\Space{\Gn}\rightarrow\Space{\Rank{\bar{0}}},\ F\mapsto(\sum_{\pi\in\Gn}F(\pi))\delta_{\bar{0}}$.
	\item {\bf Feature space.} The feature space is defined by
	\[
	\Hn = \bigoplus_{B\in\SubsetsWE{\n}}H_{B},
	\]
	where for each $B\in\SubsetsWE{\n}$, $H_{B}$ is a linear space with dimension equal to $d_{\vert B\vert}$, the number of fixed-point free permutations on a set of $\vert B\vert$ elements (see Section \ref{sec:MRA-construction} for the definition and proof of the dimension). Hence $\dim\Hn = n!$, by elementary combinatorial arguments. Elements of the feature space are viewed as collections of vectors $\mathbf{X} = (X_{B})_{B\in\SubsetsWE{\n}}$. The global support of an element $\mathbf{X}$ is the set $\Supp(\mathbf{X}) = \{B\in\SubsetsWE{\n} \;\vert\; X_{B} \neq 0 \}$, and we usually identify an element $\mathbf{X}$ with the collection restricted to its global support $(X_{B})_{B\in\Supp(\mathbf{X})}$.
	\item {\bf Wavelet transform.} The wavelet transform is an operator $\Psi$ that maps a signal to its features:
	\[
	\Psi :\qquad \Space{\Gn}\rightarrow\Hn,\qquad F\mapsto \left(\Psi_{B}F\right)_{B\in\SubsetsWE{\n}},
	\]
	where for each $B\in\SubsetsWE{\n}$, $\Psi_{B}: \Space{\Gn} \rightarrow H_{B}$ is the wavelet projection on $H_{B}$. We precise that, writing $F = (F_{A})_{A\in\SubsetsWE{\n}}$, this definition means that for all $B\in\SubsetsWE{\n}$,
	\begin{equation}
	\label{eq:wavelet-transform}
	\Psi_{B}(F) = \sum_{A\in\SubsetsWE{\n}}\Psi_{B}F_{A}.
	\end{equation}
	\item {\bf Synthesis operators.} The synthesis operators are a family of operators $(\phi_{A})_{A\in\SubsetsWE{\n}}$ where for each $A\in\SubsetsWE{\n}$, $\phi_{A}:\Hn\rightarrow \Space{\Rank{A}}$ allows to reconstruct a signal in the local space $\Space{\Rank{A}}$ from its features. It satisfies the following properties:
	\begin{align*}
	\phi_{A}\Psi (F) &=F \qquad\text{for any } F\in \Space{\Rank{A}},\\
	\text{and}\qquad\Psi\phi_{\n}(\mathbf{X}) &= \mathbf{X} \qquad\text{for any } \mathbf{X} = (X_{B})_{B\in\SubsetsWE{\n}}\in\Hn.
	\end{align*}
	We precise that, writing $\mathbf{X} = (X_{B})_{B\in\SubsetsWE{\n}}$, this definition means that for all $A\in\SubsetsWE{\n}$,
	\begin{equation}
	\label{eq:operator-phi}
	\phi_{A}(\mathbf{X}) = \sum_{B\in\SubsetsWE{\n}}\phi_{A}X_{B}.
	\end{equation}
\end{itemize}
The spaces $H_{B}$, the operators $\phi_{A}$ and the wavelet transform $\Psi$ are all constructed explicitly in Section \ref{sec:MRA-construction}. They satisfy, for all $A,B\in\SubsetsWE{\n}$, $F_{A}\in \Space{\Rank{A}}$ and $X_{B}\in H_{B}$,
\begin{align*}
\Psi_{B}F_{A} = 0 &\qquad\text{if } B\not\subset A\\
\phi_{A}X_{B} = 0 &\qquad\text{if } B\not\subset A.
\end{align*}
For $B\in\SubsetsWE{\n}$, we define the set $\mathcal{Q}(B) = \{A\in\SubsetsWE{\n} \;\vert\; B\subset A\}$. Equations \eqref{eq:wavelet-transform} and \eqref{eq:operator-phi} then become, for any $A,B\in\Subsets{\n}$, $F\in \Space{\Gn}$ and $\mathbf{X}\in\Hn$,
\[
\Psi_{B}(F) = \sum_{A\in\mathcal{Q}(B)}\Psi_{B}F_{A} \qquad\text{and}\qquad \phi_{A}(\mathbf{X}) = \sum_{B\in\SubsetsWE{A}}\phi_{A}X_{B}.
\]

\subsection{Main properties} 

The strength of the MRA representation comes from the properties of the wavelet transform, the synthesis operator and the marginal operator, summarized in the following theorem. For any collection of subsets $\mathcal{S}\subset\SubsetsWE{\n}$, we define the subspace of $\Hn$:
\[
\mathbb{H}(\mathcal{S}) = \bigoplus_{B\in\mathcal{S}}H_{B}.
\]

\begin{theorem}[Fundamental properties of the MRA representation]
\label{th:MRA-general}
Let $A\in\SubsetsWE{\n}$ and $F\in \Space{\Rank{A}}$. The MRA representation satisfies the following properties.
\begin{itemize}
	\item $\Psi F$ is the unique element in $\mathbb{H}(\SubsetsWE{A})$ such that 
	\begin{equation}
	\label{eq:MRA-general-1}
	F = \phi_{A}\Psi F = \sum_{B\in\SubsetsWE{A}}\phi_{A}\Psi_{B}F.
	\end{equation}
	\item For any $A'\in\SubsetsWE{A}$,
	\begin{equation}
	\label{eq:MRA-general-2}
	M_{A'}F = \phi_{A'}\Psi F \qquad\text{or equivalently }\qquad \Psi_{B}M_{A'}F = \Psi_{B}F \text{ for all } B\in\SubsetsWE{A'}.
	\end{equation}
\end{itemize}
\end{theorem}

Theorem \ref{th:MRA-general} is proved in Section \ref{sec:MRA-construction}. It has several implications in practice. First, Property \eqref{eq:MRA-general-1} says that a function $F\in \Space{\Rank{A}}$ with $A\in\Subsets{\n}$ can be reconstructed from its wavelet transform $\Psi F$. The latter thus contains all information related to $F$ or in other words, the knowledge of $\Psi F$ implies the knowledge of $F$. In addition, this information is decomposed between all the wavelet projections $\Psi_{B}F$, and Property \eqref{eq:MRA-general-2} says that this decomposition is consistent with the marginal operator: the marginal $M_{A'}F$ of $F$ on any subset $A'\in\Subsets{A}$ can be reconstructed from the wavelet transform of $F$ restricted to the subsets $B\in\SubsetsWE{A'}$. Figure \ref{fig:th1-illustration} illustrates these properties for a ranking model $p$ over $\Sym{3}$ with marginals $P_{\{1,2\}}$, $P_{\{1,3\}}$ and $P_{\{2,3\}}$.

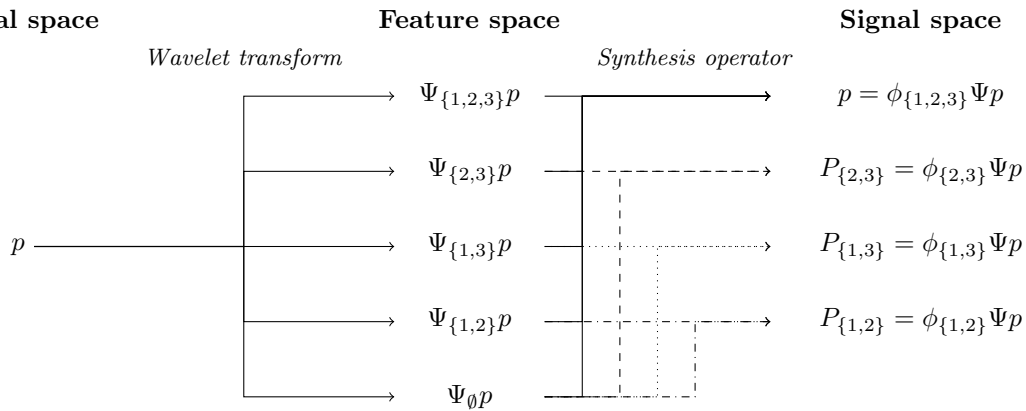
\begin{figure}[h]
\centering
	\begin{tikzpicture}
	\node at (-3,3) {\textbf{Signal space}};
	\node (a) at (-3,0) {$p$};
	\coordinate (i123) at (2,2);
	\coordinate (i23) at (2,1);
	\coordinate (i13) at (2,0);
	\coordinate (i12) at (2,-1);
	\coordinate (i0) at (2,-2);
	\node at (0,2.5) {\small{\textit{Wavelet transform}}};
	\draw[->] (a) -- ++(3,0) |- (i123);
	\draw[->] (a) -- ++(3,0) |- (i23);
	\draw[->] (a) -- ++(3,0) |- (i13);
	\draw[->] (a) -- ++(3,0) |- (i12);
	\draw[->] (a) -- ++(3,0) |- (i0);
	\node at (3,3) {\textbf{Feature space}};
	\node at (3,2) {$\Psi_{\{1,2,3\}}p$};
	\node at (3,1) {$\Psi_{\{2,3\}}p$};
	\node at (3,0) {$\Psi_{\{1,3\}}p$};
	\node at (3,-1) {$\Psi_{\{1,2\}}p$};
	\node at (3,-2) {$\Psi_{\emptyset}p$};
	\coordinate (o123) at (4,2);
	\coordinate (o23) at (4,1);
	\coordinate (o13) at (4,0);
	\coordinate (o12) at (4,-1);
	\coordinate (o0) at (4,-2);
	\node at (9,3) {\textbf{Signal space}};
	\node at (9,2) {$p = \phi_{\{1,2,3\}}\Psi p$};
	\coordinate (b1) at (7,2);
	\node at (6,2.5) {\small{\textit{Synthesis operator}}};
	\draw[->] (o123) -- ++(0.5,0) |- (b1);
	\draw[->] (o23) -- ++(0.5,0) |- (b1);
	\draw[->] (o13) -- ++(0.5,0) |- (b1);
	\draw[->] (o12) -- ++(0.5,0) |- (b1);
	\draw[->] (o0) -- ++(0.5,0) |- (b1);
	\node at (9,1) {$P_{\{2,3\}} = \phi_{\{2,3\}}\Psi p$};
	\coordinate (b2) at (7,1);
	\draw[dashed, ->] (o23) -- ++(1,0) |- (b2);
	\draw[dashed, ->] (o0) -- ++(1,0) |- (b2);
	\node at (9,0) {$P_{\{1,3\}} = \phi_{\{1,3\}}\Psi p$};
	\coordinate (b3) at (7,0);
	\draw[dotted, ->] (o13) -- ++(1.5,0) |- (b3);
	\draw[dotted, ->] (o0) -- ++(1.5,0) |- (b3);
	\node at (9,-1) {$P_{\{1,2\}} = \phi_{\{1,2\}}\Psi p$};
	\coordinate (b4) at (7,-1);
	\draw[dashdotted, ->] (o12) -- ++(2,0) |- (b4);
	\draw[dashdotted, ->] (o0) -- ++(2,0) |- (b4);
	\end{tikzpicture}
\caption{Illustration of Theorem \ref{th:MRA-general} for $n = 3$}
\label{fig:th1-illustration}
\end{figure}

\subsection{Multiresolution interpretation}\label{subsec:interpretation}

We now show that the MRA representation exploits the natural multi-scale structure of the marginals, justifying the use of terms ``MRA representation'' and ``wavelet transform''. The definition of the marginal operator \eqref{eq:marginal-operator} leads to the following relations for any subsets $A,B\in\Subsets{\n}$ with $B\subset A$:
\begin{equation}
\label{eq:projective-system}
\left(M_{A}\right)_{\Space{\Rank{A}}} = Id_{\Space{\Rank{A}}} \qquad\text{and}\qquad M_{B}M_{A} = M_{B},
\end{equation}
where $\left(M_{A}\right)_{\Space{\Rank{A}}}$ denotes the restriction of $M_{A}$ to $\Space{\Rank{A}}$ and $Id_{\Space{\Rank{A}}}$ is the identity operator on $\Space{\Rank{A}}$. Relations \eqref{eq:projective-system} actually mean that the collection of linear spaces $(\Space{\Rank{A}})_{A\in\Subsets{\n}}$ together with the collection of linear operators $(M_{A})_{A\in\Subsets{\n}}$ form a \textit{projective system}. The partial order associated with this projective system is the inclusion order on $\Subsets{\n}$. It is canonically graded with the rank function $A\mapsto\vert A\vert$. This defines a notion of \textit{scale} for the marginals, and this is why we call the projective system defined by relations \eqref{eq:projective-system} the \textit{multi-scale structure} of the marginals. Figure \ref{fig:multiscale-structure} provides an illustration for $n = 4$.

\begin{figure}[h!]
	\centering
	\begin{tikzpicture}
	\node (scale-4) at (-7,2) {\bf Scale 4};
	\node (scale-3) at (-7,0) {\bf Scale 3};
	\node (scale-2) at (-7,-2) {\bf Scale 2};
	\node (1) at (0,2) {$F$};
	\node (2) at (-4.5,0) {$M_{\{1,2,3\}}F$};
	\node (3) at (-1.5,0) {$M_{\{1,2,4\}}F$};
	\node (4) at (1.5,0) {$M_{\{1,3,4\}}F$};
	\node (5) at (4.5,0) {$M_{\{2,3,4\}}F$};
	\node (6) at (-5,-2) {$M_{\{1,2\}}F$};
	\node (7) at (-3,-2) {$M_{\{1,3\}}F$};
	\node (8) at (-1,-2) {$M_{\{1,4\}}F$};
	\node (9) at (1,-2) {$M_{\{2,3\}}F$};
	\node (10) at (3,-2) {$M_{\{2,4\}}F$};
	\node (11) at (5,-2) {$M_{\{3,4\}}F$};	
	\draw 
	(1) -- (2)
	(1) -- (3)
	(1) -- (4)
	(1) -- (5)
	(2) -- (6)
	(2) -- (7)
	(2) -- (9)
	(3) -- (6)
	(3) -- (8)
	(3) -- (10)
	(4) -- (7)
	(4) -- (8)
	(4) -- (11)
	(5) -- (9)
	(5) -- (10)
	(5) -- (11)
	;
	\end{tikzpicture}
	\caption{Multi-scale structure of the marginals of a function $F\in \Space{\Rank{\set{4}}}$}
	\label{fig:multiscale-structure}
\end{figure}
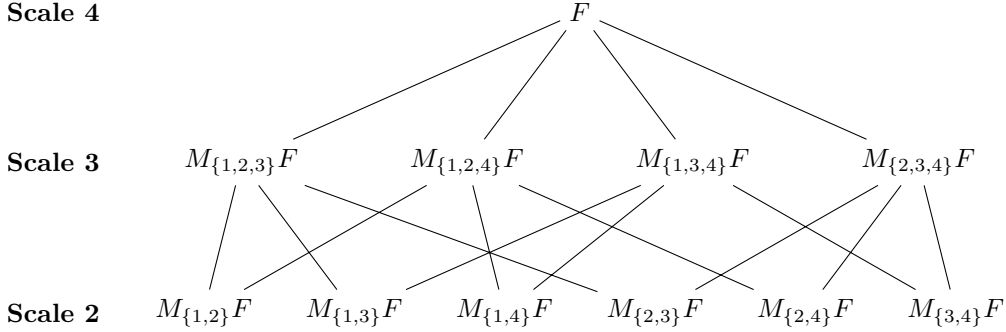

From a practical point of view, the scale of a marginal corresponds to the number of items in the subset on which the marginal is considered. By equation \eqref{eq:projective-system}, a marginal on a subset $B\in\Subsets{A}$ induces the marginals on all the subsets $C\in\Subsets{B}$. The collection of marginals $(M_{B}F)_{B\subset A,\ \vert B\vert = k}$ for $F\in \Space{\Rank{A}}$ and $k\in\{2,\dots,\vert A\vert\}$ thus induces all the marginals on subsets $C\subset A$ with $\vert C\vert \leq k-1$. Hence we say that $(M_{B}F)_{B\subset A,\ \vert B\vert = k}$ contains all the information of $F$ at scale up to $k$. This notion of scale can be naturally compared to the usual notion in image analysis: its version in low resolution can be recovered from a higher resolution. The version of the image in the higher resolution thus contains more information than the version in low resolution.

The same as for images, the piece of information gained when increasing the scale corresponds to an additional level of details. For instance, if one has access to the triple-wise marginals of a ranking model $p$ then one has access to the information contained in the pairwise marginals plus the piece of information of scale 3. This decomposition can be further refined: marginals of the same scale on different subsets provide different additional pieces of information. For instance, compared to the marginal on $\{1,4\}$, the marginals on $\{1,2,4\}$ and $\{1,3,4\}$ both provide an additional but different level of details. Pursuing the analogy with images, this decomposition of the information into pieces related to subsets of items can be compared with the space decomposition of an image: for each resolution level, an image can be spatially decomposed into different components. Therefore, through their multi-scale structure, the marginals of a function $F\in \Space{\Rank{A}}$ for $A\in\Subsets{\n}$ each contain a part of its total information, both delimited in \textit{scale} and in \textit{space}. 

The multiresolution representation allows to localize, in each of these parts, the component that is specific to the marginal. First, one has
\[
\Psi_{\emptyset}F = \left(\sum_{\pi\in\Rank{A}}F(\pi)\right)\delta_{\bar{0}},
\]
this is proven in Section \ref{sec:MRA-construction}. The projection $\Psi_{\emptyset}F$ can thus be seen as containing the piece of information of $F$ at scale $0$. Then for a pair $\{a,b\}\subset A$, applying Eq. \eqref{eq:MRA-general-1} to $M_{\{a,b\}}F$ combined with \eqref{eq:MRA-general-2} gives
\begin{equation}
\label{eq:recursion-level-2}
\Psi_{\{a,b\}}F = M_{\{a,b\}}F - \phi_{\{a,b\}}\Psi_{\emptyset}F.
\end{equation}
Hence, starting from $\Psi_{\emptyset}F$, $\Psi_{\{a,b\}}F$ contains the exact additional piece of information to recover $M_{\{a,b\}}F$. This is the part of information that is specific to $M_{\{a,b\}}F$. For a triple $\{a,b,c\}\subset A$, the same calculation gives
\begin{equation}
\label{eq:recursion-level-3}
\Psi_{\{a,b,c\}}F = M_{\{a,b,c\}}F - \phi_{\{a,b,c\}}\left[\Psi_{\emptyset}F + \Psi_{\{a,b\}}F + \Psi_{\{a,c\}}F + \Psi_{\{b,c\}}F\right].
\end{equation}
The projection $\Psi_{\{a,b,c\}}F$ of $F$ thus contains all the additional piece of information to get from the pairwise marginals $M_{\{a,b\}}F$, $M_{\{a,c\}}F$ and $M_{\{b,c\}}F$ to the triple-wise marginal $M_{\{a,b,c\}}F$. More generally, for $B\in\Subsets{A}$, $\Psi_{B}F$ contains the piece of information that is specific to $M_{B}F$, or equivalently the part of the information of $F$ that is localized on scale $\vert B\vert$ on the subset $B$. 

\begin{example}
Let $p$ be a ranking model over $\Rank{\set{3}}$ and $\Sigma$ a random permutation drawn from $p$. For clarity's sake, we denote by $\mathbb{P}\left[a_{1}\succ \dots \succ a_{k}\right]$ the probability of the event $\{\Sigma(a_{1}) < \dots < \Sigma(a_{k})\}$. One has for instance (see Section \ref{sec:MRA-construction} for the general formulas):
\[
\arraycolsep=2pt
\def\arraystretch{1.5}
\begin{array}{ccccccc}
\mathbb{P}\left[2\succ 1\succ 3\right] & = & p(213) & & & & \\
& = & \phi_{\set{3}}\Psi_{\emptyset} p(213) & + & \phi_{\set{3}}\left[\Psi_{\{1,2\}}p + \Psi_{\{1,3\}}p + \Psi_{\{2,3\}}p\right](213) & + & \phi_{\set{3}}\Psi_{\set{3}}p(213)\\
& = & \frac{1}{6} & + & \frac{1}{2}\left[\left(\mathbb{P}\left[2\succ 1\right] - \frac{1}{2}\right) + \left(\mathbb{P}\left[1\succ 3\right] - \frac{1}{2}\right)\right] & + & \Psi_{\set{3}}p(213).
\end{array}
\]
In this decomposition, the first term is the value of the uniform distribution over $\Rank{\set{3}}$, it represents information at scale $0$. The second term represents the part of information at scale 2 of $p$ that is involved in the probability of the ranking $2\succ 1\succ 3$. The last term represents the part of information involved at scale 3, it can be interpreted as a residual.
\end{example}

In wavelet analysis over a Euclidean space, each wavelet coefficient of a function $f$ contains a specific part of information, localized in scale and space. In the present context, for $F\in \Space{\Rank{A}}$ with $A\in\Subsets{\n}$ and $B\in\SubsetsWE{A}$, the coefficient $\Psi_{B}F$ contains the part of information that is specific to the marginal $M_{B}F$, or in other words localized at scale $\vert B\vert$ and subset $B$. This is why we call the operator $\Psi$ the \textit{wavelet transform} and more generally the construction the \textit{MRA representation}.

\begin{example}
\label{ex:illustration-MRA}
Here we provide a graphical illustration of the MRA representation applied to a real dataset, obtained from \citet{Croon1989} and studied for example in \citet{Diaconis1998} or \citet{YB99}. It consists of 2,262 full rankings of four political goals for the government collected from a survey conducted in Germany. Let $p$ be the normalized histogram of the results, which we consider as a ranking model over $\Sym{4}$.

The ranking model and all its marginals are represented on the left of Figure \ref{fig:German} whereas all its wavelet projections are represented on the right of Figure \ref{fig:German}. For each $B\in\Subsets{\n}$, we represent the wavelet projection $\Psi_{B}p\in H_{B}$ as an element of $\Space{\Rank{B}}$. We however point out that $\dim H_{B} = d_{\vert B\vert}$ by virtue of Theorem \ref{th:MRA-decomposition}. This means that for each $B\in\Subsets{\n}$, the wavelet projection $\Psi_{B}p$ actually characterizes $d_{\vert B\vert}$ degrees of freedom of $p$ and not $\vert B\vert !$. The graphical representation of $\Psi_{B}p$ as an element of $\Space{\Rank{B}}$ can thus be misleading. Table \ref{tab:factorial-derangements} provides a comparison of $k!$ and $d_{k}$ for different values of $k$.

\begin{figure}[h!]
\centering
\begin{tabular}{cccccccccccc}
\multicolumn{12}{@{}c@{}}{\tiny $p$}\\
\multicolumn{12}{@{}c@{}}{\includegraphics[scale=0.4]{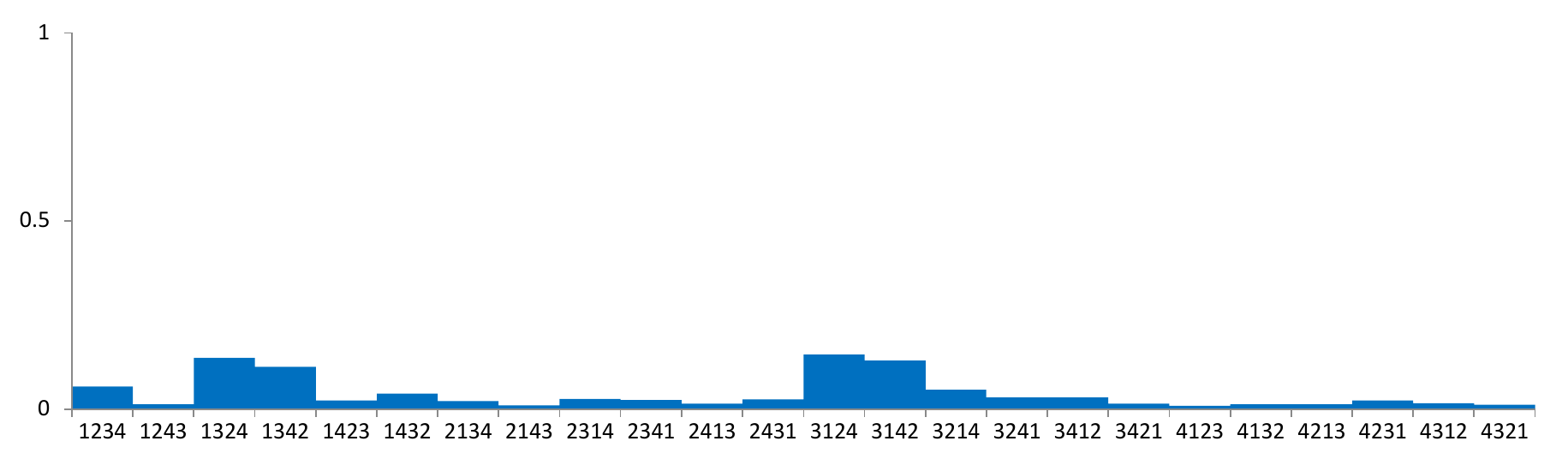}}\\
\multicolumn{3}{p{1.2cm}}{\tiny $M_{\{1,2,3\}}p$} & \multicolumn{3}{p{1.2cm}}{\tiny $M_{\{1,2,4\}}p$} & \multicolumn{3}{p{1.2cm}}{\tiny $M_{\{1,3,4\}}p$} & \multicolumn{3}{p{1.2cm}}{\tiny $M_{\{2,3,4\}}p$}\\
\multicolumn{3}{p{1.2cm}}{\includegraphics[scale=0.4]{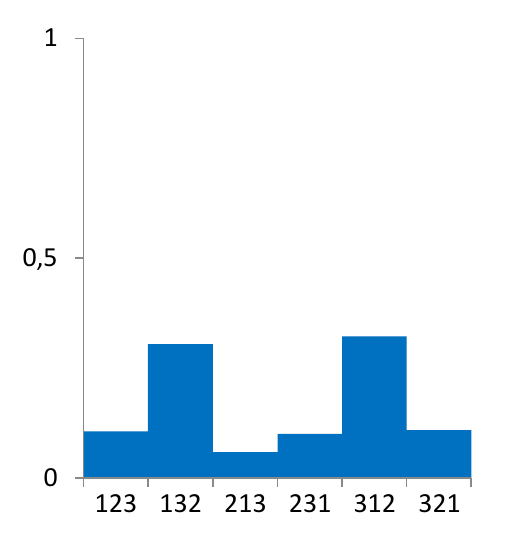}} & \multicolumn{3}{p{1.2cm}}{\includegraphics[scale=0.4]{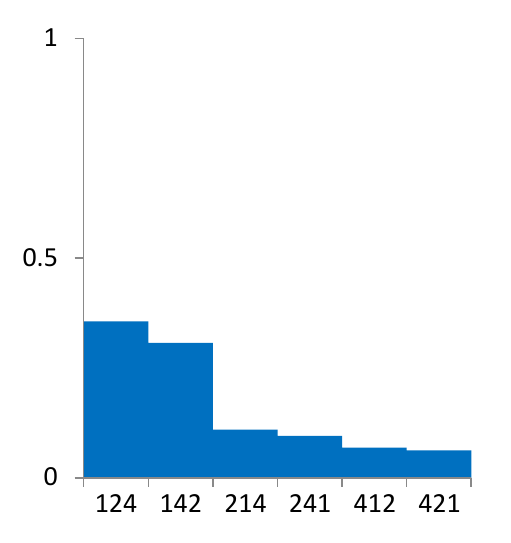}} & \multicolumn{3}{p{1.2cm}}{\includegraphics[scale=0.4]{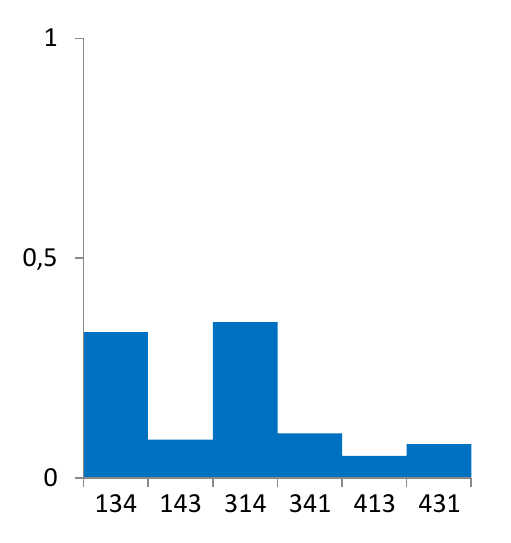}} & \multicolumn{3}{p{1.2cm}}{\includegraphics[scale=0.4]{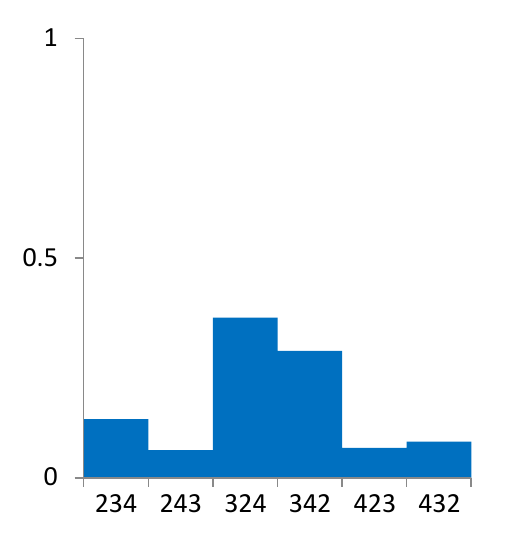}}\\
\multicolumn{2}{p{0.8cm}}{\tiny $M_{\{1,2\}}p$} & \multicolumn{2}{p{0.8cm}}{\tiny $M_{\{1,3\}}p$} & \multicolumn{2}{p{0.8cm}}{\tiny $M_{\{1,4\}}p$} & \multicolumn{2}{p{0.8cm}}{\tiny $M_{\{2,3\}}p$} & \multicolumn{2}{p{0.8cm}}{\tiny $M_{\{2,4\}}p$} & \multicolumn{2}{p{0.8cm}}{\tiny $M_{\{3,4\}}p$}\\
\multicolumn{2}{p{0.8cm}}{\includegraphics[scale=0.4]{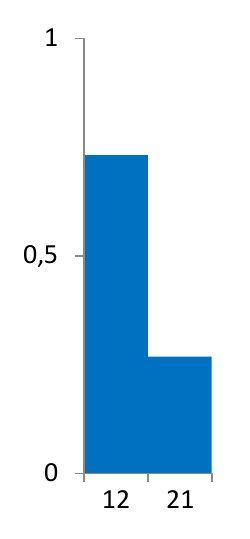}} & \multicolumn{2}{p{0.8cm}}{\includegraphics[scale=0.4]{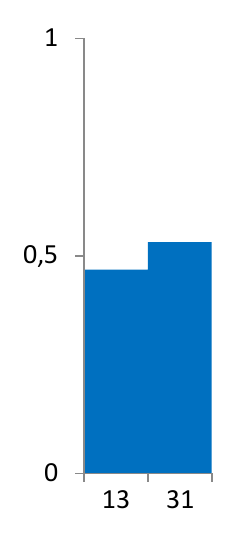}} & \multicolumn{2}{p{0.8cm}}{\includegraphics[scale=0.4]{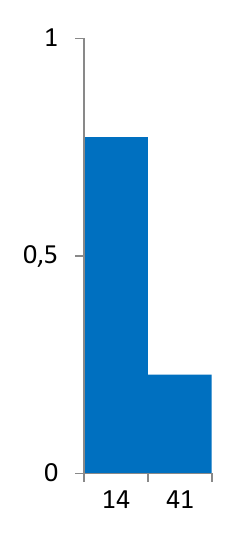}} & \multicolumn{2}{p{0.8cm}}{\includegraphics[scale=0.4]{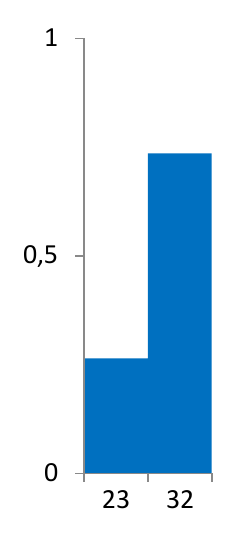}} & \multicolumn{2}{p{0.8cm}}{\includegraphics[scale=0.4]{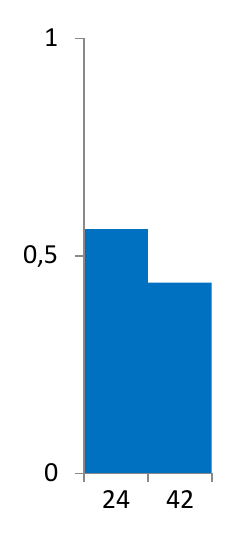}} & \multicolumn{2}{p{0.8cm}}{\includegraphics[scale=0.4]{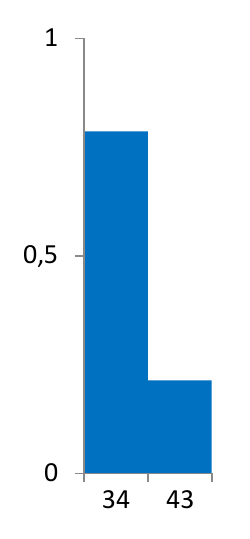}}
\end{tabular}
\ 
\begin{tabular}{cccccccccccc}
\multicolumn{12}{@{}c@{}}{\tiny $X_{\{1,2,3,4\}}p$}\\
\multicolumn{12}{@{}c@{}}{\includegraphics[scale=0.4]{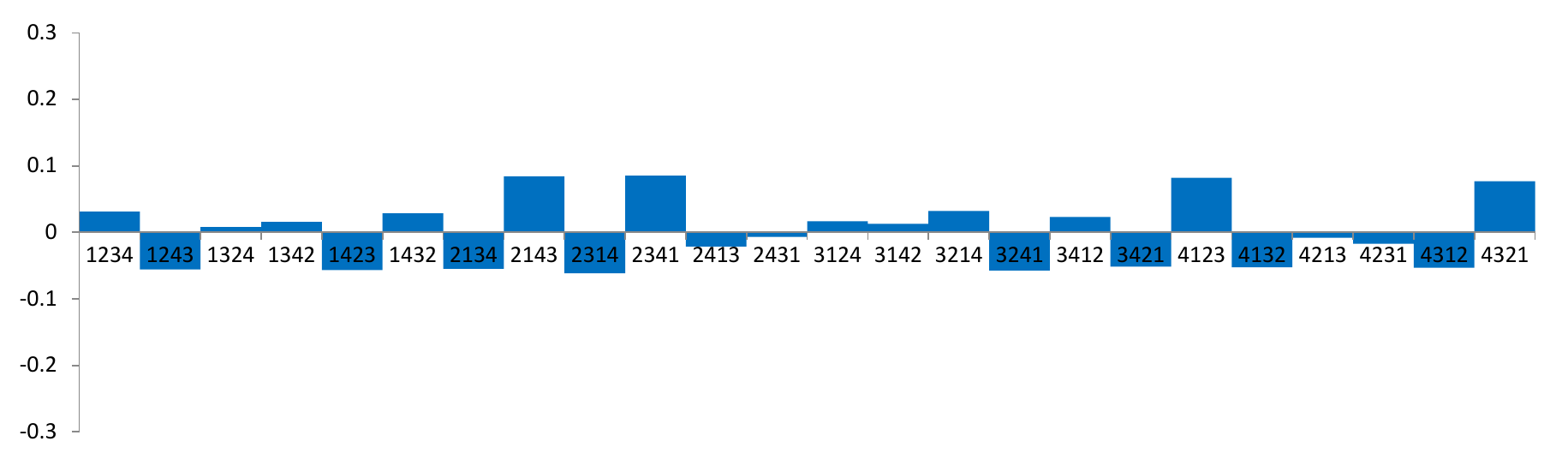}}\\
\multicolumn{3}{p{1.2cm}}{\tiny $X_{\{1,2,3\}}p$} & \multicolumn{3}{p{1.2cm}}{\tiny $X_{\{1,2,4\}}p$} & \multicolumn{3}{p{1.2cm}}{\tiny $X_{\{1,3,4\}}p$} & \multicolumn{3}{p{1.2cm}}{\tiny $X_{\{2,3,4\}}p$}\\
\multicolumn{3}{p{1.2cm}}{\includegraphics[scale=0.4]{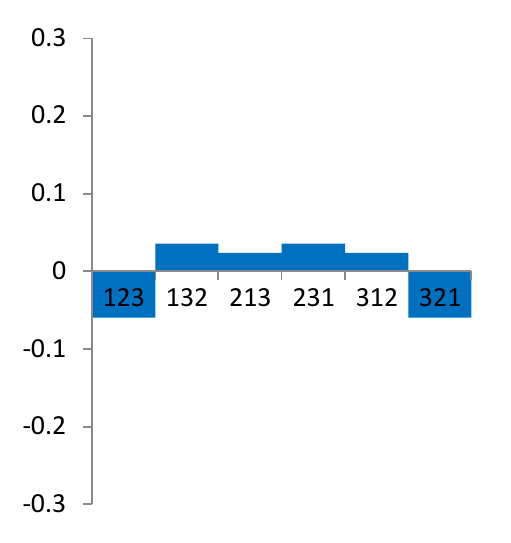}} & \multicolumn{3}{p{1.2cm}}{\includegraphics[scale=0.4]{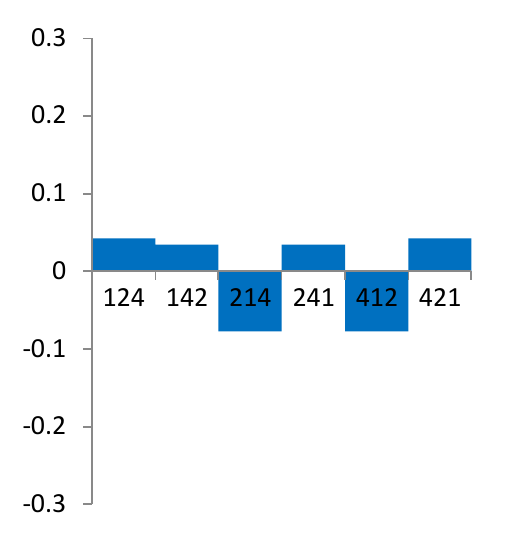}} & \multicolumn{3}{p{1.2cm}}{\includegraphics[scale=0.4]{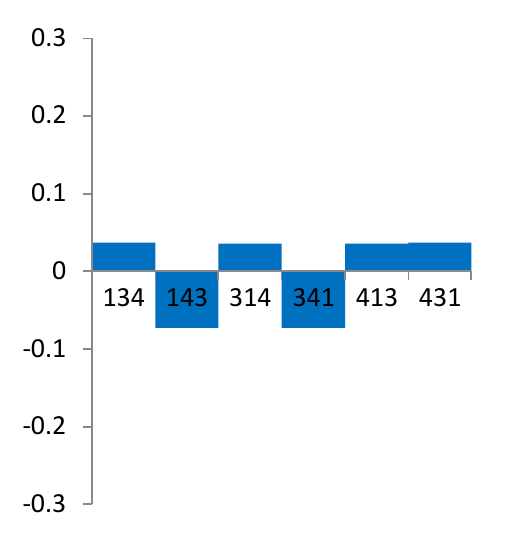}} & \multicolumn{3}{p{1.2cm}}{\includegraphics[scale=0.4]{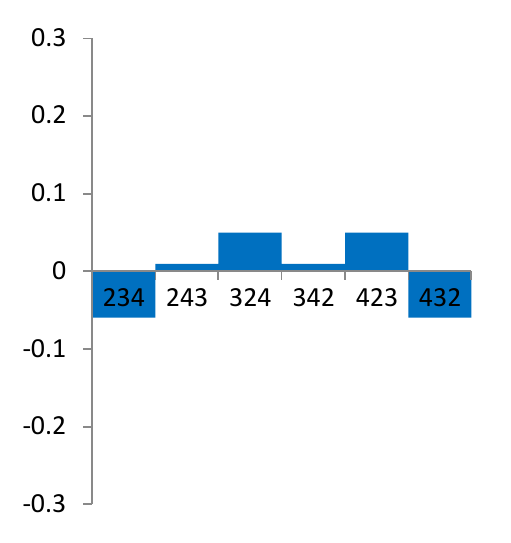}}\\
\multicolumn{2}{p{0.8cm}}{\tiny $X_{\{1,2\}}p$} & \multicolumn{2}{p{0.8cm}}{\tiny $X_{\{1,3\}}p$} & \multicolumn{2}{p{0.8cm}}{\tiny $X_{\{1,4\}}p$} & \multicolumn{2}{p{0.8cm}}{\tiny $X_{\{2,3\}}p$} & \multicolumn{2}{p{0.8cm}}{\tiny $X_{\{2,4\}}p$} & \multicolumn{2}{p{0.8cm}}{\tiny $X_{\{3,4\}}p$}\\
\multicolumn{2}{p{0.8cm}}{\includegraphics[scale=0.4]{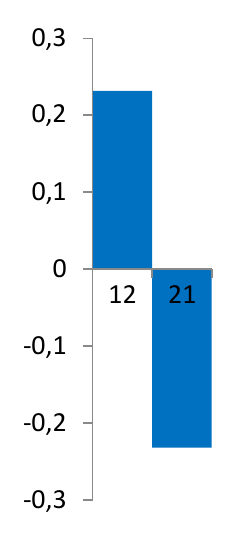}} & \multicolumn{2}{p{0.8cm}}{\includegraphics[scale=0.4]{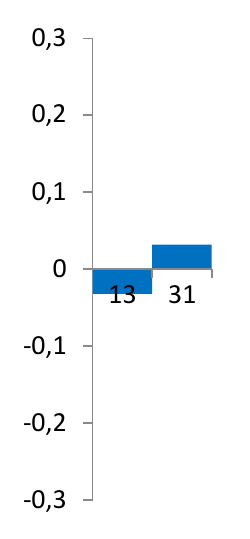}} & \multicolumn{2}{p{0.8cm}}{\includegraphics[scale=0.4]{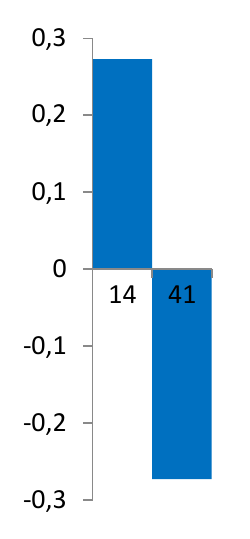}} & \multicolumn{2}{p{0.8cm}}{\includegraphics[scale=0.4]{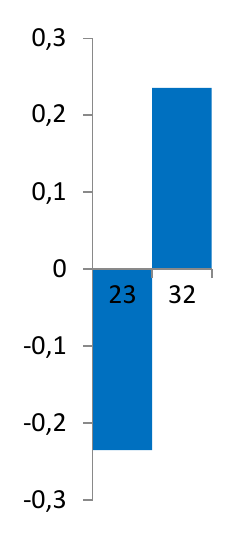}} & \multicolumn{2}{p{0.8cm}}{\includegraphics[scale=0.4]{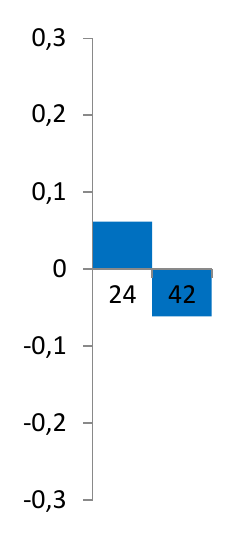}} & \multicolumn{2}{p{0.8cm}}{\includegraphics[scale=0.4]{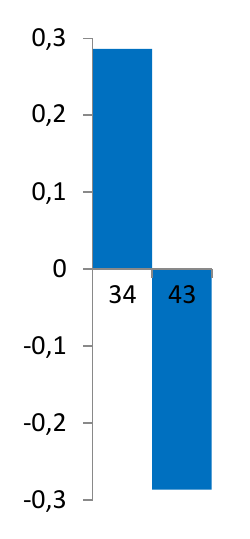}}
\end{tabular}
\caption{Ranking model $p$ and its marginals on the left. Wavelet projections of the ranking model $p$ on the right}
\label{fig:German}
\end{figure}
\begin{table}[h!]
\centering
\begin{tabular}{c|ccccccc}
$k$ & 0 & 1 & 2 & 3 & 4 & 5 & 6 \\
\hline
$k!$ & 1 & 1 & 2 & 6 & 24 & 120 & 720 \\
$d_{k}$ & 1 & 0 & 1 & 2 & 9 & 44 & 265
\end{tabular}
\caption{Values of $k!$ and $d_{k}$}
\label{tab:factorial-derangements}
\end{table}
\end{example}

\begin{remark}
Beyond the useful analogy, we point out several differences between the MRA representation for incomplete rankings and standard wavelet analysis over a Euclidean space.
\begin{itemize}
	\item The signal space $\Space{\Gn} = \bigoplus_{A\in\SubsetsWE{\n}}\Space{\Rank{A}}$ being heterogeneous, it is usually required in applications to reconstruct a signal only in a local signal space $\Space{\Rank{A}}$. This is why the MRA representation comes with a family of synthesis operators $\phi_{A}$ and not just one. 
	\item The wavelet transform $\Psi$ maps a function $F$ to a collection of \textit{vector coefficients} $\Psi_{B}F$ and not \textit{scalar coefficients} as it is the case in Euclidean harmonic analysis. This means that each wavelet projection localizes a part of information with several degrees of freedom.
	\item The subspace decomposition associated to $\Psi$ is not orthogonal and $\Psi$ is not an isometry. More specifically, for $F\in \Space{\Rank{A}}$ with $A\in\Subsets{\n}$, one has in general 
	\[
	\Vert F\Vert_{A}^{2}\quad \neq \sum_{B\in\SubsetsWE{A}}\Vert \Psi_{B}F\Vert_{B}^{2},
	\] 
	where $\Vert\cdot\Vert_{B}$ is an abbreviated notation for the Euclidean norm $\Vert\cdot\Vert_{\Rank{B}}$ on $\Space{\Rank{B}}$ for any $B\in\Subsets{\n}$. This last fact implies in particular that the classic nonlinear approximation theory based on wavelet analysis is not applicable: keeping only the wavelet projections with highest $l^{2}$ norm will not necessarily provide a good approximation of the signal for the $l^{2}$ norm (nor any $l^{p}$ norm).
\end{itemize}
\end{remark}

\subsection{Solving linear systems involving the marginal operator}

One of the main consequences of Theorem \ref{th:MRA-general} is that the MRA representation ``simultaneously block-diagonalize'' the marginal operators $M_{A}$. To be more specific, let $\mathcal{M}_{A}:\Hn\rightarrow\Hn$ be the operator defined by $\mathcal{M}_{A} = \Psi M_{A}\phi_{\n}$ for $A\in\Subsets{\n}$. The following proposition is a direct consequence of Theorem \ref{th:MRA-general}.

\begin{proposition}[Marginal operator in the feature space]
\label{prop:marginal-operator-feature-space}
Let $A\in\Subsets{\n}$. For all $F\in \Space{\Gn}$,
\[
\Psi M_{A} F = \mathcal{M}_{A}\Psi F.
\]
In other words, the operator $\mathcal{M}_{A}$ is such that the following diagram is commutative.
\begin{center}
	\begin{tikzpicture}
		\node (a) at (0,2) {$\Space{\Gn}$};
		\node (b) at (3,2) {$\Hn$};
		\node (c) at (3,0) {$\Hn$};
		\node (d) at (0,0) {$\Space{\Rank{A}}$};
		\path[->] (a) edge node[above] {$\Psi$} (b);
		\path[->] (a) edge node[left] {$M_{A}$} (d);
		\path[->] (b) edge node[right] {$\mathcal{M}_{A}$} (c);
		\path[->] (d) edge node[below] {$\Psi$} (c);
	\end{tikzpicture}
\end{center}
\end{proposition}

\begin{proof}
Let $A\in\Subsets{\n}$ and $F\in \Space{\Gn}$. By definition of the operator $\mathcal{M}_{A}$,
\[
\mathcal{M}_{A}\Psi F = \Psi M_{A}\phi_{\n}\Psi F.
\]
Now, applying Property \eqref{eq:MRA-general-2} successively to $\phi_{\n}\Psi F$ and $F$ gives
\[
M_{A}\phi_{\n}\Psi F = \phi_{A}\Psi\phi_{\n}\Psi F = \phi_{A}\Psi F = M_{A}F,
\]
where we recall that $\Psi\phi_{\n}\mathbf{X} = \mathbf{X}$ for any $\mathbf{X}\in\Hn$. Hence $\mathcal{M}_{A}\Psi F = \Psi M_{A}F$.
\end{proof}

Proposition \ref{prop:marginal-operator-feature-space} says that applying the operator $\mathcal{M}_{A}$ in the feature space is equivalent to applying the marginal operator $M_{A}$ in the signal space. This is why we call $\mathcal{M}_{A}$ the \textit{marginal operator in the feature space}. Now, Theorem \ref{th:MRA-general} also implies that this operator is actually a simple projection.

\begin{proposition}[Simultaneous block-diagonalization]
\label{prop:projector}
For $A\in\Subsets{\n}$, $\mathcal{M}_{A}$ is the projection on $\mathbb{H}(\SubsetsWE{A})$: for any $(X_{B})_{B\in\SubsetsWE{\n}}\in\Hn$,
\[
\mathcal{M}_{A}\left((X_{B})_{B\in\SubsetsWE{\n}}\right) = (X_{B})_{B\in\SubsetsWE{A}}.
\]
Equivalently, the matrix of $\mathcal{M}_{A}$ in any basis of $\Hn$ consistent with the decomposition $\bigoplus_{B\in\SubsetsWE{\n}}H_{B}$ is of the form
\[
\kbordermatrix{
	& H_{\emptyset} & \cdots & H_{\n}\\
	H_{\emptyset} & \mathbf{m}_{\emptyset} & \cdots & 0\\
	\vdots & \vdots & \ddots & 0 \\
	H_{\n} & 0 & \cdots & \mathbf{m}_{\n}
},
\]
where for $B\in\SubsetsWE{\n}$, $\mathbf{m}_{B} = \mathbf{I}_{B}$, the matrix of the identity operator $Id_{H_{B}}$ on $H_{B}$, if $B\subset A$ and $\mathbf{m}_{B} = 0$ otherwise.
\end{proposition}

\begin{proof}
Let $A\in\Subsets{\n}$ and $\mathbf{X}\in\Hn$. Applying Property \eqref{eq:MRA-general-2} to $\phi_{\n}\mathbf{X}$ one obtains
\[
\mathcal{M}_{A}(\mathbf{X}) = \Psi M_{A}\phi_{\n}(\mathbf{X}) = \Psi \phi_{A}(\mathbf{X}) = \Psi \sum_{B\in\SubsetsWE{A}}\phi_{A}X_{B} = \left(X_{B}\right)_{B\in\SubsetsWE{A}},
\]
which concludes the proof.
\end{proof}

\begin{example}
	The matrix of $\mathcal{M}_{\{1,2,4\}}$ in any basis of $\mathbb{H}_{4}$ consistent with the decomposition $\bigoplus_{B\in\SubsetsWE{\set{4}}}H_{B}$ is equal to
	\[
	\kbordermatrix{
		& H_{\emptyset} & H_{\{1,2\}} & H_{\{1,3\}} & H_{\{1,4\}} & H_{\{2,3\}} & H_{\{2,4\}} & H_{\{3,4\}} & H_{\{1,2,3\}} & H_{\{1,2,4\}} & H_{\{1,3,4\}} & H_{\{2,3,4\}} & H_{\set{4}}\\
		H_{\emptyset} & \mathbf{I}_{\emptyset} & 0 & 0 & 0 & 0 & 0 & 0 & 0 & 0 & 0 & 0 & 0\\
		H_{\{1,2\}}   & 0 & \mathbf{I}_{\{1,2\}} & 0 & 0 & 0 & 0 & 0 & 0 & 0 & 0 & 0 & 0\\
		H_{\{1,3\}}   & 0 & 0 & 0 & 0 & 0 & 0 & 0 & 0 & 0 & 0 & 0 & 0\\
		H_{\{1,4\}}   & 0 & 0 & 0 & \mathbf{I}_{\{1,4\}} & 0 & 0 & 0 & 0 & 0 & 0 & 0 & 0\\
		H_{\{2,3\}}   & 0 & 0 & 0 & 0 & 0 & 0 & 0 & 0 & 0 & 0 & 0 & 0\\
		H_{\{2,4\}}   & 0 & 0 & 0 & 0 & 0 & \mathbf{I}_{\{2,4\}} & 0 & 0 & 0 & 0 & 0 & 0\\
		H_{\{3,4\}}   & 0 & 0 & 0 & 0 & 0 & 0 & 0 & 0 & 0 & 0 & 0 & 0\\
		H_{\{1,2,3\}} & 0 & 0 & 0 & 0 & 0 & 0 & 0 & 0 & 0 & 0 & 0 & 0\\
		H_{\{1,2,4\}} & 0 & 0 & 0 & 0 & 0 & 0 & 0 & 0 & \mathbf{I}_{\{1,2,4\}} & 0 & 0 & 0\\
		H_{\{1,3,4\}} & 0 & 0 & 0 & 0 & 0 & 0 & 0 & 0 & 0 & 0 & 0 & 0\\
		H_{\{2,3,4\}} & 0 & 0 & 0 & 0 & 0 & 0 & 0 & 0 & 0 & 0 & 0 & 0\\
		H_{\set{4}}   & 0 & 0 & 0 & 0 & 0 & 0 & 0 & 0 & 0 & 0 & 0 & 0\\
	}
	\]
\end{example}

Proposition \ref{prop:projector} says at the same time that the marginal operator in the MRA representation boils down to a simple filter, and that all the marginal operators are ``block-diagonalized'' in the MRA representation. These properties mean that the MRA representation is best fitted to solve linear systems that involve the marginal operator. This is formalized in the following theorem. For any collection $\mathcal{S}\subset{\Subsets{\n}}$, we set 
\[
\SubsetsWE{\mathcal{S}} := \bigcup_{A\in\mathcal{S}}\SubsetsWE{A}.
\]

\begin{theorem}[Solutions to linear systems]
\label{th:inverse-linear-system}
Let $A\in\Subsets{\n}$ and $F_{0}\in\Space{\Rank{A}}$.
\begin{itemize}
	\item For $A'\in\Subsets{A}$, the solutions to the problem
	\begin{equation}
	\label{eq:linear-system-one-subset}
	\text{Find } F\in\Space{\Rank{A}} \text{ such that } M_{A'}F = M_{A'}F_{0}
	\end{equation}
	are all of the form 
	\[
	\phi_{A}\sum_{B\in\SubsetsWE{A'}}\Psi_{B}F_{0} \quad +\quad \phi_{A}\mathbf{X},
	\]
	with $\mathbf{X}\in\mathbb{H}(\SubsetsWE{A}\setminus\Subsets{A'})$. In particular the space of solutions has dimension $\dim \mathbb{H}(\SubsetsWE{A}\setminus\Subsets{A'}) = \vert A\vert ! - \vert A'\vert !$.
	\item More generally for $\mathcal{S}\subset\Subsets{A}$, the solutions to the problem
	\begin{equation}
	\label{eq:linear-system-several-subsets}
	\text{Find } F\in\Space{\Rank{A}} \text{ such that } M_{A'}F = M_{A'}F_{0} \text{ for all } A'\in\mathcal{S}
	\end{equation}
	are all of the form 
	\[
	\phi_{A}\sum_{B\in\SubsetsWE{\mathcal{S}}}\Psi_{B}F_{0} \quad +\quad \phi_{A}\mathbf{X}
	\]
	with $\mathbf{X}\in\mathbb{H}(\SubsetsWE{A}\setminus\Subsets{\mathcal{S}})$. In particular the space of solutions has dimension $\dim \mathbb{H}(\SubsetsWE{A}\setminus\Subsets{\mathcal{S}}) = \vert A\vert ! - \sum_{A'\in\mathcal{S}}d_{\vert A'\vert}$.
\end{itemize}
\end{theorem}

\begin{proof}
It is sufficient to prove the theorem for problem \eqref{eq:linear-system-several-subsets}. Let $F\in\Space{\Rank{A}}$. For any $A'\in\mathcal{S}$,
\[
\begin{array}{ccll}
M_{A'}F = M_{A'}F_{0} 	& \Leftrightarrow & \Psi M_{A'}F = \Psi M_{A'}F_{0} & \text{by Theorem \ref{th:MRA-general}}\\
						& \Leftrightarrow & \mathcal{M}_{A'}\Psi F = \mathcal{M}_{A'}\Psi F_{0} & \text{by Proposition \ref{prop:marginal-operator-feature-space}}\\
						& \Leftrightarrow & \Psi_{B} F = \Psi_{B}F_{0} \text{ for all } B\in\SubsetsWE{A'} & \text{by Proposition \ref{prop:projector}}.
\end{array}
\]
Thus $M_{A'}F = M_{A'}F_{0}$ for all $A'\in\mathcal{S}$ if and only if $\Psi_{B}F = \Psi_{B}F_{0}$ for all $B\in\SubsetsWE{\mathcal{S}}$. Applying Theorem \ref{th:MRA-general} one last time concludes the proof.
\end{proof}

\begin{example}
We illustrate Theorem \ref{th:inverse-linear-system} for the ranking model $p$ over $\Sym{4}$ constructed from the real dataset already considered in Example \ref{ex:illustration-MRA}. Let us assume that one does not know the ranking model $p$, but knows exactly some of its marginals $M_{A}p$ for subsets $A$ in the observation design $\mathcal{A} = \{ \{1,3\}, \{2,4\}, \{3,4\}, \{1,2,3\}, \{1,3,4\} \}$. One then has $\Subsets{\set{4}}\setminus\SubsetsWE{\mathcal{A}} = \{ \{1,2,4\}, \{2,3,4\}, \{1,2,3,4\} \}$. Theorem \ref{th:inverse-linear-system} thus tells us that the functions on $\Sym{4}$ that have the same marginal as $p$ for all subsets $A\in\mathcal{A}$ are of the form
\[
F = \phi_{\set{4}}\sum_{B\in\SubsetsWE{\mathcal{A}}}\Psi_{B}p + \phi_{\set{4}}\left[X_{\{1,2,4\}} + X_{\{2,3,4\}} + X_{\{1,2,3,4\}}\right] \qquad\text{with } X_{B}\in H_{B},
\]
where the $X_{B}$'s can be arbitrary. The set composed of such functions is therefore a linear space of dimension $d_{4} + 2d_{3} = 13$. Examples of such functions with their marginals are represented in Figure \ref{fig:approximations}. The graphs on the left represent the function with $X_{B} = 0$ and the graphs on the right represent a function obtained with $X_{B}$'s sampled randomly.

\begin{figure}[h!]
\centering
\begin{tabular}{cccccccccccc}
\multicolumn{12}{@{}c@{}}{\tiny $F$}\\
\multicolumn{12}{@{}c@{}}{\includegraphics[scale=0.4]{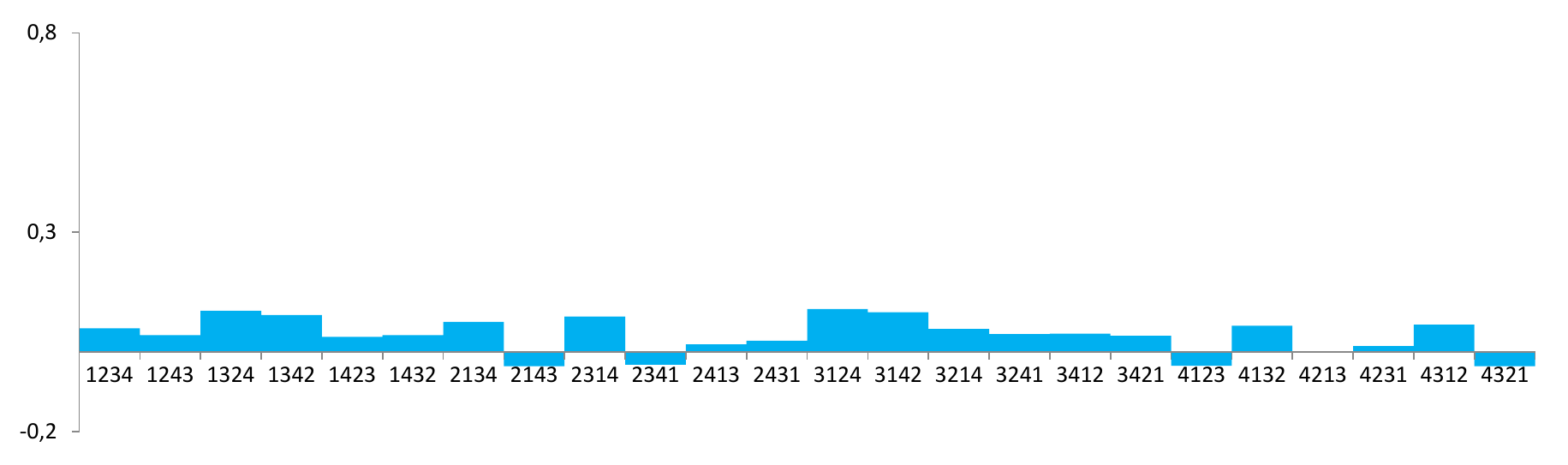}}\\
\multicolumn{3}{p{1.2cm}}{\tiny $M_{\{1,2,3\}}F$} & \multicolumn{3}{p{1.2cm}}{\tiny $M_{\{1,2,4\}}F$} & \multicolumn{3}{p{1.2cm}}{\tiny $M_{\{1,3,4\}}F$} & \multicolumn{3}{p{1.2cm}}{\tiny $M_{\{2,3,4\}}F$}\\
\multicolumn{3}{p{1.2cm}}{\includegraphics[scale=0.4]{P123-German.pdf}} & \multicolumn{3}{p{1.2cm}}{\includegraphics[scale=0.4]{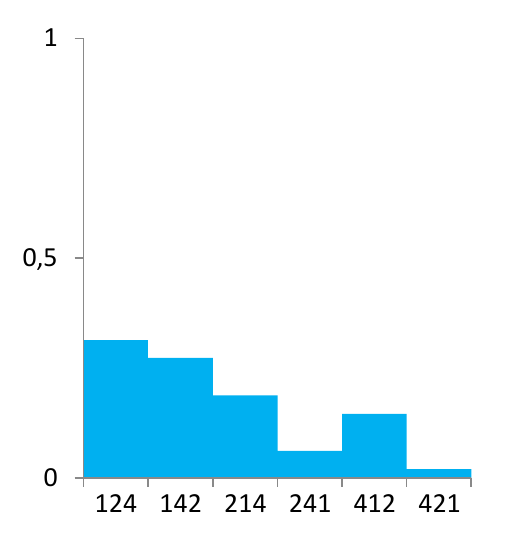}} & \multicolumn{3}{p{1.2cm}}{\includegraphics[scale=0.4]{P134-German.pdf}} & \multicolumn{3}{p{1.2cm}}{\includegraphics[scale=0.4]{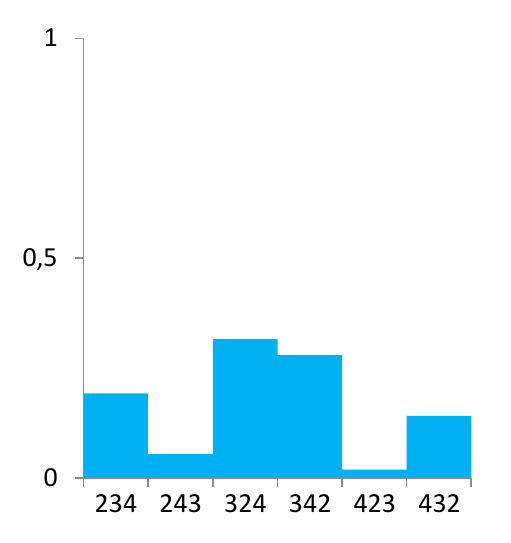}}\\
\multicolumn{2}{p{0.8cm}}{\tiny $M_{\{1,2\}}F$} & \multicolumn{2}{p{0.8cm}}{\tiny $M_{\{1,3\}}F$} & \multicolumn{2}{p{0.8cm}}{\tiny $M_{\{1,4\}}F$} & \multicolumn{2}{p{0.8cm}}{\tiny $M_{\{2,3\}}F$} & \multicolumn{2}{p{0.8cm}}{\tiny $M_{\{2,4\}}F$} & \multicolumn{2}{p{0.8cm}}{\tiny $M_{\{3,4\}}F$}\\
\multicolumn{2}{p{0.8cm}}{\includegraphics[scale=0.4]{P12-German.pdf}} & \multicolumn{2}{p{0.8cm}}{\includegraphics[scale=0.4]{P13-German.pdf}} & \multicolumn{2}{p{0.8cm}}{\includegraphics[scale=0.4]{P14-German.pdf}} & \multicolumn{2}{p{0.8cm}}{\includegraphics[scale=0.4]{P23-German.pdf}} & \multicolumn{2}{p{0.8cm}}{\includegraphics[scale=0.4]{P24-German.pdf}} & \multicolumn{2}{p{0.8cm}}{\includegraphics[scale=0.4]{P34-German.pdf}}
\end{tabular}
\
\begin{tabular}{cccccccccccc}
\multicolumn{12}{@{}c@{}}{\tiny $F$}\\
\multicolumn{12}{@{}c@{}}{\includegraphics[scale=0.4]{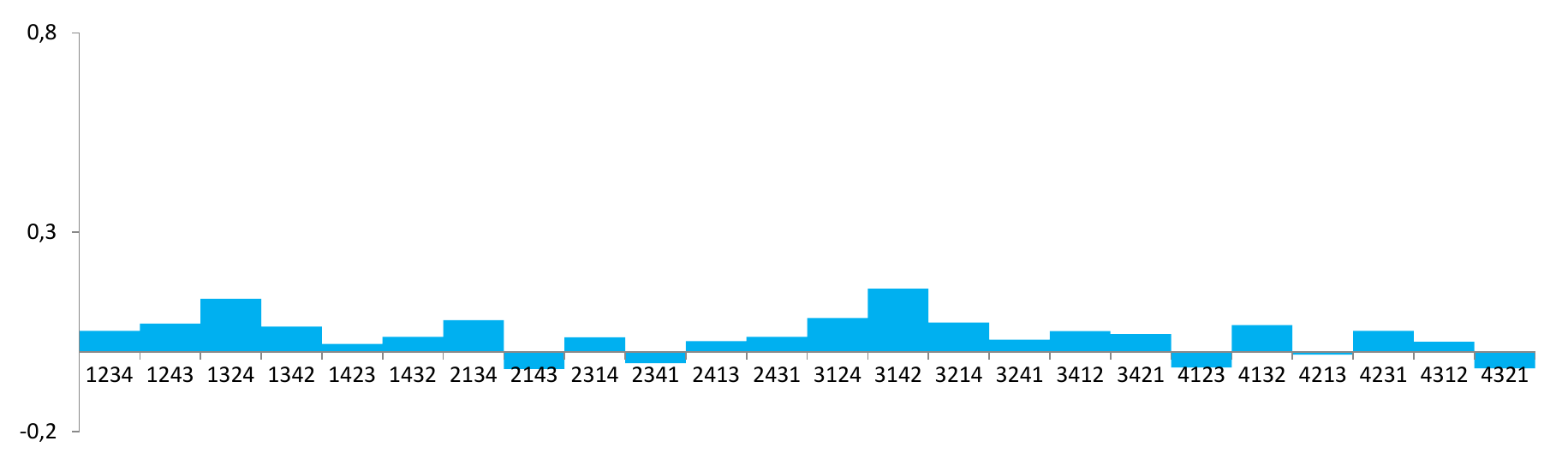}}\\
\multicolumn{3}{p{1.2cm}}{\tiny $M_{\{1,2,3\}}F$} & \multicolumn{3}{p{1.2cm}}{\tiny $M_{\{1,2,4\}}F$} & \multicolumn{3}{p{1.2cm}}{\tiny $M_{\{1,3,4\}}F$} & \multicolumn{3}{p{1.2cm}}{\tiny $M_{\{2,3,4\}}F$}\\
\multicolumn{3}{p{1.2cm}}{\includegraphics[scale=0.4]{P123-German.pdf}} & \multicolumn{3}{p{1.2cm}}{\includegraphics[scale=0.4]{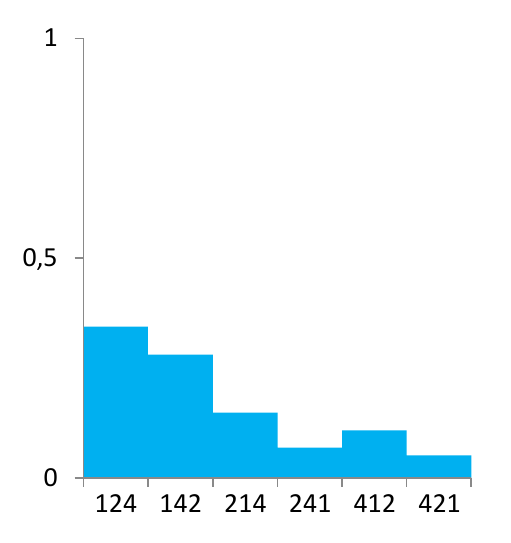}} & \multicolumn{3}{p{1.2cm}}{\includegraphics[scale=0.4]{P134-German.pdf}} & \multicolumn{3}{p{1.2cm}}{\includegraphics[scale=0.4]{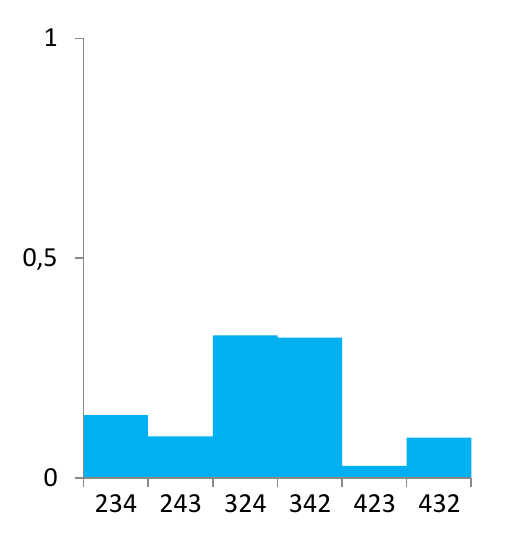}}\\
\multicolumn{2}{p{0.8cm}}{\tiny $M_{\{1,2\}}F$} & \multicolumn{2}{p{0.8cm}}{\tiny $M_{\{1,3\}}F$} & \multicolumn{2}{p{0.8cm}}{\tiny $M_{\{1,4\}}F$} & \multicolumn{2}{p{0.8cm}}{\tiny $M_{\{2,3\}}F$} & \multicolumn{2}{p{0.8cm}}{\tiny $M_{\{2,4\}}F$} & \multicolumn{2}{p{0.8cm}}{\tiny $M_{\{3,4\}}F$}\\
\multicolumn{2}{p{0.8cm}}{\includegraphics[scale=0.4]{P12-German.pdf}} & \multicolumn{2}{p{0.8cm}}{\includegraphics[scale=0.4]{P13-German.pdf}} & \multicolumn{2}{p{0.8cm}}{\includegraphics[scale=0.4]{P14-German.pdf}} & \multicolumn{2}{p{0.8cm}}{\includegraphics[scale=0.4]{P23-German.pdf}} & \multicolumn{2}{p{0.8cm}}{\includegraphics[scale=0.4]{P24-German.pdf}} & \multicolumn{2}{p{0.8cm}}{\includegraphics[scale=0.4]{P34-German.pdf}}
\end{tabular}
\caption{Function $F$ and its marginals, with $X_{B} = 0$ on the left and $X_{B}$ drawn randomly on the right.}
\label{fig:approximations}
\end{figure}
\end{example}

\subsection{Fast Wavelet Transform}

The MRA representation would be of little interest without efficient procedures to compute the wavelet transform of a function $F\in\Space{\Gn}$ an the synthesis of an element $\mathbf{X}\in\Hn$. Fortunately, such procedures exist and we now describe them in details. They are directly inspired by the \textit{Fast Wavelet Transform} (FWT) introduced in \citet{Mallat1989}. We first recall some background about it.

\ \\
\noindent
{\bf Background on FWT in classic wavelet theory.} In classic multiresolution analysis on $l^{2}(\mathbb{Z})$\footnote{$l^{2}(\mathbb{Z}) = \{ f :\mathbb{Z}\rightarrow\mathbb{R} \;\vert\; \sum_{m\in\mathbb{Z}}f(m)^{2} < \infty\}$.}, one is given a scaling basis $(\phi_{j,k})_{j,k\in\mathbb{Z}}$ and a wavelet basis $(\phi_{j,k})_{j,k\in\mathbb{Z}}$, so that any function $f\in l^{2}(\mathbb{Z})$ decomposes as
\[
f = \sum_{k\in\mathbb{Z}}\left\langle f,\phi_{j_{0},k}\right\rangle\phi_{j_{0},k} + \sum_{j=j_{0}}^{+\infty}\sum_{k\in\mathbb{Z}}\left\langle f,\psi_{j,k}\right\rangle\psi_{j,k}
\]
for any $j_{0}\in\mathbb{Z}$ \citep[see][for the details]{Mallat}. The scalars $d_{j}[k] := \left\langle f,\psi_{j,k}\right\rangle$ are the wavelet coefficients and the scalars $a_{j}[k] := \left\langle f,\phi_{j,k}\right\rangle$ are called the approximation coefficients. The fast wavelet transform computes efficiently the wavelet coefficients by exploiting the two following properties of wavelet bases.
\begin{itemize}
	\item All the wavelet coefficients at scale $j$ can be computed from the approximation coefficients at scale $j$ via a linear operator $h$:
	\begin{equation}
	\label{eq:classic-high-pass-filter}
	d_{j}[k] = (h a_{j})[k]
	\end{equation}
	\item All the approximation coefficients at scale $j$ can be computed from the approximation coefficients at scale $j+1$ via a linear operator $g$:
	\begin{equation}
	\label{eq:classic-low-pass-filter}
	a_{j}[k] = (g a_{j+1})[k]
	\end{equation}
\end{itemize}
The operator $g$ computes local averages of the signal and is therefore called a low-pass filter. The operator $h$ computes local differences of the signal and is therefore called a high-pass filter. The FWT then consists in applying recursively these filters in two steps:
\begin{enumerate}
	\item Apply the high-pass filter $h$ to $a_{j}$ to obtain the wavelet coefficients $d_{j}$
	\item Apply the low-pass filter $g$ to $a_{j}$ to obtain $a_{j-1}$
\end{enumerate}
This procedure is illustrated by Figure \ref{fig:FWT-classic} (the wavelet coefficients are highlighted in blue). It is called ``fast'' because it computes all the coefficients of a same scale at the same time.

\begin{figure}[h!]
\centering
\begin{tikzpicture}
\node (signal) at (0,0) {$f = a_{J}$};
\node[draw] (h1) at (2.5,0) {$h$};
\node[draw] (g1) at (2.5,-0.75) {$g$};
\draw[->] (signal) -- (h1);
\draw[->] (signal) -- ++(1.5,0) |- (g1);
\node (dJ) at (4,0) {$\textcolor{blue}{d_{J}}$};
\node (aJ-1) at (4,-0.75) {$a_{J-1}$};
\draw[->] (h1) -- (dJ);
\draw[->] (g1) -- (aJ-1);
\node[draw] (h2) at (6,-0.75) {$h$};
\node[draw] (g2) at (6,-1.5) {$g$};
\draw[->] (aJ-1) -- (h2);
\draw[->] (aJ-1) -- ++(1,0) |- (g2);
\node (dJ-1) at (7.5,-0.75) {$\textcolor{blue}{d_{J-1}}$};
\node (aJ-2) at (7.5,-1.5) {$a_{J-2}$};
\draw[->] (h2) -- (dJ-1);
\draw[->] (g2) -- (aJ-2);
\node[draw] (h3) at (9.5,-1.5) {$h$};
\node[draw] (g3) at (9.5,-2.25) {$g$};
\draw[->] (aJ-2) -- (h3);
\draw[->] (aJ-2) -- ++(1,0) |- (g3);
\node (dJ-2) at (11,-1.5) {$\textcolor{blue}{d_{J-2}}$};
\node (aJ-3) at (11,-2.25) {$a_{J-3}$};
\draw[->] (h3) -- (dJ-2);
\draw[->] (g3) -- (aJ-3);
\draw[->,dashed] (aJ-3) -- (13,-2.25); 
\end{tikzpicture}
\caption{Fast Wavelet transform with filter banks}
\label{fig:FWT-classic}
\end{figure}
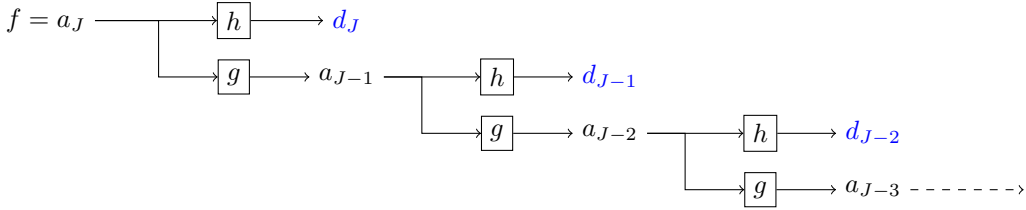

In practice for a function $f\in l^{2}(\mathbb{Z})$ with finite support, the number of wavelet and approximation coefficients decreases with the scale. The application of the filters $g$ and $h$ at scale $j$ then only involve the operations with the finite vector $a_{j}$. The implementation of the FWT therefore uses families of filters $(g_{j})_{j}$ and $(h_{j})_{j}$ where $g_{j}$ and $h_{j}$ are the operators applied effectively on $a_{j}$.

\begin{example}[FWT for the Haar wavelets]
\label{ex:FWT-Haar}
The following diagram illustrates the fast Haar wavelet transform of a signal $f = (f_{1},\dots,f_{8})\in\mathbb{R}^{8}$. 
\begin{center}
\begin{tikzpicture}
\node (signal) at (0,0) {$\left[\begin{array}{c}
	f_{1}\\
	f_{2}\\
	f_{3}\\
	f_{4}\\
	f_{5}\\
	f_{6}\\
	f_{7}\\
	f_{8}
	\end{array}\right]$};
\node[draw] (h3) at (1.5,0) {$h_{3}$};
\node[draw] (g3) at (1.5,-2) {$g_{3}$};
\draw[->] (signal) -- (h3);
\draw[->] (signal) -- ++(0.75,0) |- (g3);
\node (d3) at (3.75,0) {$\textcolor{blue}{d_{3} = \left[\begin{array}{c}
		f_{1}-f_{2}\\
		f_{3}-f_{4}\\
		f_{5}-f_{6}\\
		f_{7}-f_{8}\\
		\end{array}\right]}$};
\node (a2) at (3.75,-2) {$a_{2} = \left[\begin{array}{c}
		f_{1}+f_{2}\\
		f_{3}+f_{4}\\
		f_{5}+f_{6}\\
		f_{7}+f_{8}\\
	\end{array}\right]$};
\draw[->] (h3) -- (d3);
\draw[->] (g3) -- (a2);
\node[draw] (h2) at (6.25,-2) {$h_{2}$};
\node[draw] (g2) at (6.25,-3.25) {$g_{2}$};
\draw[->] (a2) -- (h2);
\draw[->] (a2) -- ++(1.75,0) |- (g2);
\node (d2) at (8.75,-2) {$\textcolor{blue}{d_{2} = \left[\begin{array}{c}
		a_{2}[1]-a_{2}[2]\\
		a_{2}[3]-a_{2}[4]\\
		\end{array}\right]}$};
\node (a1) at (8.75,-3.25) {$a_{1} = \left[\begin{array}{c}
		a_{2}[1]+a_{2}[2]\\
		a_{2}[3]+a_{2}[4]\\
	\end{array}\right]$};
\draw[->] (h2) -- (d2);
\draw[->] (g2) -- (a1);
\node[draw] (h1) at (11.5,-3.25) {$h_{3}$};
\node[draw] (g1) at (11.5,-4) {$g_{3}$};
\draw[->] (a1) -- (h1);
\draw[->] (a1) -- ++(2,0) |- (g1);
\node (d1) at (13.75,-3.25) {$\textcolor{blue}{d_{1} = [a_{1}[1] - a_{1}[2]]}$};
\node (a0) at (13.75,-4) {$\textcolor{blue}{a_{0} = [a_{1}[1] + a_{1}[2]]}$};
\draw[->] (h1) -- (d1);
\draw[->] (g1) -- (a0);
\end{tikzpicture}
\end{center}
\end{example}

\noindent
{\bf The FWT for the MRA representation.} We now define the FWT for the MRA representation. We first consider the wavelet transform of a function $F\in\Space{\Rank{A}}$ with $A\in\Subsets{\n}$. For $k\in\{2,\dots,\vert A\vert\}$ we denote by $\Gamma^{k}_{A} := \bigsqcup_{B\subset A,\, \vert B\vert = k}\Rank{B}$ the set of all incomplete rankings of $k$ items of $A$.
\begin{itemize}
	\item The analogues of the approximation coefficients of $F$ at scale $j\in\{2,\dots,\vert A\vert\}$ are the marginals $M_{B}F$ for $B\subset A$ with $\vert B\vert = j$. The vector of approximation coefficients of $F$ at scale $j$ is defined by
	\begin{equation}
	\label{eq:def-MRA-approximation-coefficients}
	M^{j}F  = (M_{B}F(\pi))_{\pi\in\Rank{B},\vert B\vert = j} = (M_{c(\pi)}F(\pi))_{\pi\in\Gamma_{A}^{j}} \in\mathbb{R}^{\vert A\vert!/(\vert A\vert - j)!}.
	\end{equation}
	\item The wavelet coefficients of $F$ at scale $j\in\{2,\dots,\vert A\vert\}$ are the wavelet projections $\Psi_{B}F$ for $B\subset A$ with $\vert B\vert = j$. The vector of wavelet coefficients of $F$ at scale $j$ is defined by
	\begin{equation}
	\label{eq:def-MRA-wavelet-coefficients}
	\Psi^{j}F = (\Psi_{B}F(\pi))_{\pi\in\Rank{B},\vert B\vert = j} = (\Psi_{c(\pi)}F(\pi))_{\pi\in\Gamma_{A}^{j}} \in\mathbb{R}^{\vert A\vert !/(\vert A\vert-j)!}.
	\end{equation}
\end{itemize}
Same as in classic wavelet theory, the FWT for the MRA representation also relies on two major relations between the wavelet and approximation coefficients, analogous to Formulas \eqref{eq:classic-high-pass-filter} and \eqref{eq:classic-low-pass-filter}. The analogue of Formula \eqref{eq:classic-low-pass-filter} stems from the properties of the marginal operators. For $\pi = \pi_{1}\dots\pi_{j}\in\Gamma_{n}$ with $c(\pi)\varsubsetneq A$, one has 
\begin{equation}
\label{eq:spanning-tree}
M_{c(\pi)}F(\pi) = M_{c(\pi)\cup\{b\}}F(b\pi_{1}\dots\pi_{j}) + M_{c(\pi)\cup\{b\}}F(\pi_{1}b\dots\pi_{j}) + \dots + M_{c(\pi)\cup\{b\}}F(\pi_{1}\dots\pi_{j}b)
\end{equation}
for any $b\in A\setminus c(\pi)$. In addition one has $M_{\emptyset}F(\bar{0}) = M_{\{a,b\}}F(ab) + M_{\{a,b\}}F(ba)$ for any $a,b\in A$ with $a\neq b$. We therefore define the low-pass filters as follows.

\begin{definition}[Low-pass filters]
\label{def:MRA-low-pass-filters}
We define the order $2$ low-pass filter $g_{A}^{2}:\Space{\Gamma^{2}_{A}}\rightarrow\mathbb{R}\bar{0}$ on $A\in\Subsets{\n}$ by
\[
g_{A}^{2}F(\bar{0}) = F_{\{a,b\}}(ab) + F_{\{a,b\}}(ba) \qquad\text{for any } F\in\Space{\Gamma^{2}_{A}},
\]
where $a$ and $b$ are distinct items in $A$ (we take the two smallest by convention). For $j\in\{3,\dots,\vert A\vert\}$ we define the order $j$ low-pass filter $g_{A}^{j}:\Space{\Gamma^{j}_{A}}\rightarrow\Space{\Gamma^{j-1}_{A}}$ on $A$ by 
\[
g_{A}^{j}F(\pi_{1}\dots\pi_{j-1}) = F(b_{\pi}\pi_{1}\dots\pi_{j-1}) + F(\pi_{1}b_{\pi}\dots\pi_{j-1}) + \dots + F(\pi_{1}\dots\pi_{j-1}b_{\pi})
\]
for any $F\in\Space{\Gamma^{j}_{A}}$ and $\pi = \pi_{1}\dots\pi_{j-1}\in\Gamma^{j-1}_{A}$, where $b_{\pi}$ is any item in $A\setminus c(\pi)$ (we take the smallest by convention).
\end{definition}

To define the high-pass filters, first observe that by Property \eqref{eq:MRA-general-2} of Theorem \ref{th:MRA-general}, one has $\Psi_{B}F = \Psi_{B}M_{B}F$ for any $B\in\SubsetsWE{\n}$ and $F\in\Space{\Gn}$. This justifies the following definition.

\begin{definition}[Alpha coefficients]
\label{def:alpha-coefficients}
For $B\in\SubsetsWE{\n}$ and $\pi,\pi'\in\Rank{B}$, we define the alpha coefficient 
\[
\alpha_{B}(\pi,\pi') = \Psi_{B}\delta_{\pi'}(\pi) \qquad\text{so that}\qquad \Psi_{B}F(\pi) = \sum_{\pi'\in\Rank{B}}\alpha_{B}(\pi,\pi')M_{B}F(\pi') \qquad\text{for any } F\in\Space{\Gn}.
\]
\end{definition}

The high-pass filters are constructed with the alpha coefficients from Definition \ref{def:alpha-coefficients}.

\begin{definition}[High-pass filters]
\label{def:MRA-high-pass-filters}
Let $A\in\Subsets{\n}$. For $k\in\{2,\dots,\vert A\vert\}$, the high-pass filter on $A$ at scale $j$ is the operator $h_{A}^{j}:\Space{\Gamma^{j}_{A}}\rightarrow\Space{\Gamma^{j}_{A}}$ defined by
\[
h_{A}^{j}F(\pi) = \sum_{\pi'\in\Rank{c(\pi)}}\alpha_{c(\pi)}(\pi,\pi')F(\pi') \qquad\text{for any } F\in\Space{\Gamma^{j}_{A}} \text{ and } \pi\in\Gamma^{j}_{A}.
\]
\end{definition}

The analogues of Formulas \eqref{eq:classic-low-pass-filter} and \eqref{eq:classic-high-pass-filter} are then given by the following proposition. As it is a direct consequence of Definitions \ref{def:MRA-low-pass-filters} and \ref{def:MRA-high-pass-filters}, its proof is left to the reader.

\begin{proposition}
\label{prop:structural-formulas-FWT}
Let $A\in\Subsets{\n}$ and $F\in\Space{\Rank{A}}$. 
\begin{itemize}
	\item The wavelet coefficients $\Psi^{j}F$ of $F$ at scale $j\in\{2,\dots,\vert A\vert\}$ can all be computed from the approximation coefficients $M^{j}F$ at scale $j$ through the high-pass filter $h_{A}^{j}$:
	\begin{equation}
	\label{eq:MRA-high-pass-filter}
	\Psi^{j}F = h_{A}^{j}M^{j}F.
	\end{equation}
	\item The approximation coefficients $M^{j}F$ of $F$ at scale $j\in\{2,\dots,\vert A\vert - 1\}$ can all be computed from the approximation coefficients $M^{j+1}F$ at scale $j+1$ through the low-pass filter $g_{A}^{j+1}$:
	\begin{equation}
	\label{eq:MRA-low-pass-filter}
	M^{j}F = g_{A}^{j+1}M^{j+1}F.
	\end{equation}
\end{itemize}
\end{proposition}

Formulas \eqref{eq:MRA-high-pass-filter} and \eqref{eq:MRA-low-pass-filter} are the respective analogues of Formulas \eqref{eq:classic-high-pass-filter} and \eqref{eq:classic-low-pass-filter} in classic wavelet analysis. The FWT for the MRA representation can then be formulated as the FWT in classic wavelet theory: starting from the highest scale, apply recursively the high-pass filter on the approximation coefficients to obtain the wavelet coefficients and the low-pass filter to obtain the approximation coefficients of lower scale. The procedure is formalized in Algorithm \ref{alg:FWT-MRA}.

\begin{algorithm}
\caption{FWT for a function $F\in\Space{\Rank{A}}$ with $A\in\Subsets{\n}$}
\label{alg:FWT-MRA}
\begin{algorithmic}
\Require $F\in\Space{\Rank{A}}$ with $A\in\Subsets{\n}$
\State $M^{\vert A\vert}F = F$
\For{$j$ from $\vert A\vert$ to $2$}
	\State $\Psi^{j}F = h_{A}^{j}M^{j}F$
	\State $M^{j-1}F = g_{A}^{j}M^{j}F$
\EndFor
\State \Return $\Psi F = \{M^{1}F\}\cup(\Psi^{j}F)_{2\leq j\leq \vert A\vert}$
\end{algorithmic}
\end{algorithm}

\begin{example}[FWT for the MRA representation]
The following diagram illustrates the FWT for a function $F\in\Space{\Rank{\set{3}}}$. For any $F'\in\Space{\Gn}$ and $\pi\in\Gn$, the value $F'(\pi)$ is denoted by $F'_{\pi}$.
\begin{center}
	\begin{tikzpicture}
	\node (signal) at (0,0) {$\left[\begin{array}{c}
		F_{123}\\
		F_{132}\\
		F_{213}\\
		F_{231}\\
		F_{312}\\
		F_{321}\\
		\end{array}\right]$};
	\node[draw] (h3) at (1.6,0) {$h_{\set{3}}^{3}$};
	\node[draw] (g3) at (1.6,-3) {$g_{\set{3}}^{3}$};
	\draw[->] (signal) -- (h3);
	\draw[->] (signal) -- ++(1,0) |- (g3);
	\node (d3) at (5.5,0) {$\textcolor{blue}{\Psi_{\set{3}}^{3}F = \left[\begin{array}{c}
			\sum_{\pi\in\Rank{\set{3}}}\alpha_{\set{3}}(123,\pi)F_{\pi}\\
			\sum_{\pi\in\Rank{\set{3}}}\alpha_{\set{3}}(132,\pi)F_{\pi}\\
			\sum_{\pi\in\Rank{\set{3}}}\alpha_{\set{3}}(213,\pi)F_{\pi}\\
			\sum_{\pi\in\Rank{\set{3}}}\alpha_{\set{3}}(231,\pi)F_{\pi}\\
			\sum_{\pi\in\Rank{\set{3}}}\alpha_{\set{3}}(312,\pi)F_{\pi}\\
			\sum_{\pi\in\Rank{\set{3}}}\alpha_{\set{3}}(321,\pi)F_{\pi}\\
			\end{array}\right]}$};
	\node (a2) at (5.5,-3) {$M_{\set{3}}^{2}F = \left[\begin{array}{c}
		F_{123} + F_{132} + F_{312}\\
		F_{213} + F_{231} + F_{321}\\
		F_{132} + F_{123} + F_{213}\\
		F_{312} + F_{321} + F_{231}\\
		F_{231} + F_{213} + F_{123}\\
		F_{321} + F_{312} + F_{231}\\
		\end{array}\right]$};
	\draw[->] (h3) -- (d3);
	\draw[->] (g3) -- (a2);
	\node[draw] (h2) at (9.25,-3) {$h_{2}$};
	\node[draw] (g2) at (9.25,-5) {$g_{2}$};
	\draw[->] (a2) -- (h2);
	\draw[->] (a2) -- ++(3,0) |- (g2);
	\node (d2) at (12.5,-3) {$\textcolor{blue}{\Psi_{\set{3}}^{2}F = \left[\begin{array}{c}
		M^{2}F_{12}-M^{2}F_{21}\\
		M^{2}F_{21}-M^{2}F_{12}\\
		M^{2}F_{13}-M^{2}F_{31}\\
		M^{2}F_{31}-M^{2}F_{13}\\
		M^{2}F_{23}-M^{2}F_{32}\\
		M^{2}F_{32}-M^{2}F_{23}\\
		\end{array}\right]}$};
	\node (a0) at (12.5,-5) {$\textcolor{blue}{\Psi_{\emptyset}F = [M^{12}F_{12} + M^{2}F_{21}]}$};
	\draw[->] (h2) -- (d2);
	\draw[->] (g2) -- (a0);
	\end{tikzpicture}
\end{center}
\end{example}

Same as the FWT in classic wavelet theory, we call Algorithm \ref{alg:FWT-MRA} a ``fast'' wavelet transform because it computes all the coefficients of a same scale at the same time. Several differences are worth being pointed out though. We refer the reader to \cite{Mallat} for background on classic wavelet theory.
\begin{itemize}
	\item {\bf Forest structure instead of tree structure.} The classic FWT involves a recursive partitioning of the signal space: at each scale $j$, the vector of approximation coefficients $a_{j}$ is partitioned into sub-vectors and each sub-vector is averaged to output the approximation coefficients at scale $j-1$. This structure is encoded in the definition of the low-pass filter, Example \ref{ex:FWT-Haar} provides an illustration. The recursive partitioning can be represented by a tree, as shown by Figure \ref{fig:tree-structure-classic-FWT}. By contrast, the FWT for the MRA representation follows more a ``forest structure'', namely the multi-scale structure of the marginals represented by Figure \ref{fig:multiscale-structure}. At scale $j$, each approximation coefficient can be computed as the average of several subsets of approximation coefficients of scale $j+1$, as shown by Equation \eqref{eq:spanning-tree}. As a consequence, the low-pass filters from Definition \ref{def:MRA-low-pass-filters} are defined up to a convention. They correspond to a certain choice of a spanning tree for the forest structure of the marginals, as illustrated by Figure \ref{fig:forest-structure-FWT-MRA}.
	
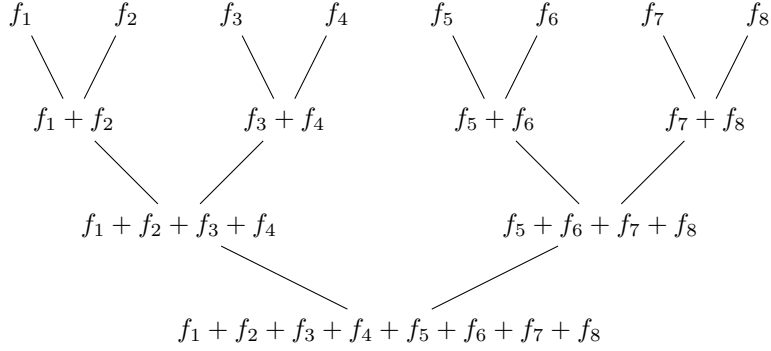
\begin{figure}
\centering
\begin{tikzpicture}[scale = 0.7]
\node (f1) at (-7,0) {$f_{1}$};
\node (f2) at (-5,0) {$f_{2}$};
\node (f3) at (-3,0) {$f_{3}$};
\node (f4) at (-1,0) {$f_{4}$};
\node (f5) at (1,0) {$f_{5}$};
\node (f6) at (3,0) {$f_{6}$};
\node (f7) at (5,0) {$f_{7}$};
\node (f8) at (7,0) {$f_{8}$};
\node (f12) at (-6,-2) {$f_{1} + f_{2}$};
\node (f34) at (-2,-2) {$f_{3} + f_{4}$};
\node (f56) at (2,-2) {$f_{5} + f_{6}$};
\node (f78) at (6,-2) {$f_{7} + f_{8}$};
\node (f1234) at (-4,-4) {$f_{1} + f_{2} + f_{3} + f_{4}$};
\node (f5678) at (4,-4) {$f_{5} + f_{6} + f_{7} + f_{8}$};
\node (f12345678) at (0,-6) {$f_{1} + f_{2} + f_{3} + f_{4} + f_{5} + f_{6} + f_{7} + f_{8}$};
\draw 
	(f1) -- (f12)
	(f2) -- (f12)
	(f3) -- (f34)
	(f4) -- (f34)
	(f5) -- (f56)
	(f6) -- (f56)
	(f7) -- (f78)
	(f8) -- (f78)
	(f12) -- (f1234)
	(f34) -- (f1234)
	(f56) -- (f5678)
	(f78) -- (f5678)
	(f1234) -- (f12345678)
	(f5678) -- (f12345678)
	;
\end{tikzpicture}
\caption{Tree structure of the FWT in classic wavelet theory}
\label{fig:tree-structure-classic-FWT}
\end{figure}

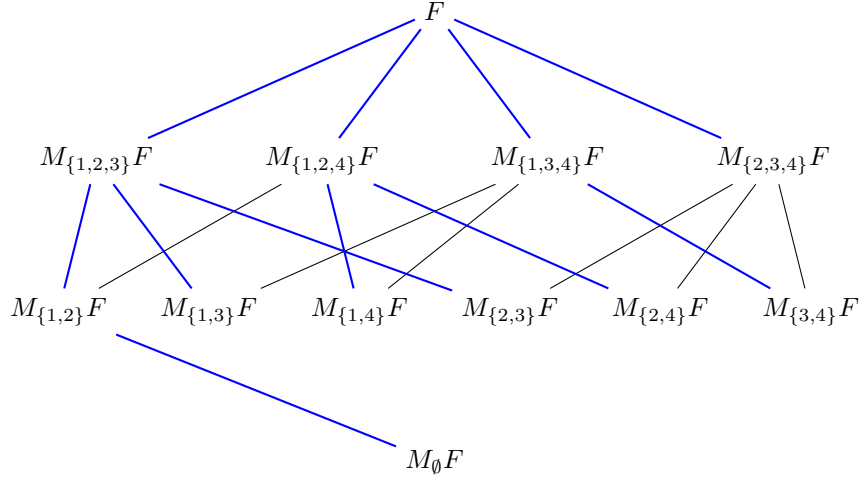
\begin{figure}
\centering
\begin{tikzpicture}
	\node (1234) at (0,2) {$F$};
	\node (123) at (-4.5,0) {$M_{\{1,2,3\}}F$};
	\node (124) at (-1.5,0) {$M_{\{1,2,4\}}F$};
	\node (134) at (1.5,0) {$M_{\{1,3,4\}}F$};
	\node (234) at (4.5,0) {$M_{\{2,3,4\}}F$};
	\node (12) at (-5,-2) {$M_{\{1,2\}}F$};
	\node (13) at (-3,-2) {$M_{\{1,3\}}F$};
	\node (14) at (-1,-2) {$M_{\{1,4\}}F$};
	\node (23) at (1,-2) {$M_{\{2,3\}}F$};
	\node (24) at (3,-2) {$M_{\{2,4\}}F$};
	\node (34) at (5,-2) {$M_{\{3,4\}}F$};
	\node (0) at (0, -4) {$M_{\emptyset}F$};
	\draw 
	(124) -- (12)
	(134) -- (13)
	(134) -- (14)
	(234) -- (23)
	(234) -- (24)
	(234) -- (34)
	;
	\draw[thick, blue]
	(12) -- (0)
	(123) -- (12)
	(123) -- (13)
	(123) -- (23)
	(124) -- (14)
	(124) -- (24)
	(134) -- (34)
	(1234) -- (123)
	(1234) -- (124)
	(1234) -- (134)
	(1234) -- (234)
	;
\end{tikzpicture}
\caption{Forest structure of the FWT for the MRA representation for $A = \set{4}$. The spanning tree highlighted in blue is the one obtained for $b_{\pi} = \min A\setminus c(\pi)$ in the Definition \ref{def:MRA-low-pass-filters} of the low-pass filters.}
\label{fig:forest-structure-FWT-MRA}
\end{figure}

	\item {\bf Downsampling.} The FWT in classic wavelet theory more specifically relies on a binary tree structure. At each step, the low-pass filter therefore divides the number of approximation (and thus also wavelet) coefficients by $2$. Example \ref{ex:FWT-Haar} provides an illustration. In the MRA representation, the number of approximation and wavelet coefficients of a function $F\in\Space{\Rank{A}}$ with $A\in\Subsets{\n}$ at scale $j\in\{2,\dots,\vert A\vert\}$ is equal to $\vert A\vert ! / (\vert A\vert - j)!$, as shown by Equations \ref{eq:def-MRA-approximation-coefficients} and \eqref{eq:def-MRA-wavelet-coefficients}. Hence at scale $j$, the FWT divides the number of coefficients by $(\vert A\vert - j)$.
	\item {\bf Support of the high-pass filters.} In classic wavelet theory, each wavelet coefficient at scale $j$ is computed from a specific subset of approximation coefficients at scale $j$. Equivalently, each approximation coefficient is involved in the computation of only one wavelet coefficient. As a consequence, the computation of all the wavelet coefficients at scale $j$ can be done in one convolution of the vector $a_{j}$. The structure of the high-pass filter is a little more complicated in the MRA representation: for a function $F\in\Space{\Rank{A}}$ with $A\in\Subsets{\n}$ and a subset $B\in\Subsets{A}$, the computation of each of the wavelet coefficients $\Psi_{B}F(\pi)$ for $\pi\in\Rank{B}$ involves all the approximation coefficients $M_{B}F(\pi')$ for $\pi'\in\Rank{B}$, by Definition \ref{def:MRA-high-pass-filters} of the high-pass filters. This means that for $j\in\{2,\dots,\vert A\vert\}$, the application of the high-pass filter $h_{A}^{j}$ requires $j!$ convolutions of the vector $M^{j}F$.
\end{itemize}

\begin{remark}[Further Optimization of the FWT]
We point out that the FWT could be further optimized. Indeed for $B\in\Subsets{\n}$, the space $H_{B}$ has dimension $d_{\vert B\vert}$, whereas the wavelet $\Psi_{B}F$ projection of a function $F\in\Space{\Gn}$ on $H_{B}$ is a vector of size $\vert B\vert !$. A fully optimized procedure would therefore compute only $d_{\vert B\vert}$ scalar coefficients and not $\vert B\vert !$. This could be done for instance with the use of a wavelet basis (see Section \ref{sec:discussion} for more details). This direction is left for future work.
\end{remark}

The aforementioned differences between the FWT in classic wavelet theory and the FWT for the MRA representation are due to the specific combinatorial structure of the latter. They also stem from the differences between the notions of information localization. In classic multiresolution analysis, the wavelet coefficients are localized in ``space'' and ``scale'', where ``space'' is the very object the signal is defined on. In other words, the metric in this space corresponds to the difference between the indexes of the coordinates: for a signal $f = (f_{1},\dots,f_{m})\in\mathbb{R}^{m}$, $f_{i}$ and $f_{i'}$ corresponds to the values of the function $f$ at points that are separated by a distance of $\vert i'-i\vert$. Then at each scale, the coordinates are partitioned recursively into subsets of adjacent coordinates (see Figure \ref{fig:tree-structure-classic-FWT}), defining a metric for the scale that is coarser but consistent with the metric of the higher scales. Each wavelet coefficient is thus localized in space and scale because its computation only involves a small number of approximation coefficients that are close with respect the scale.

The notion of information localization in the MRA representation is fundamentally different. The signal is defined on rankings but the wavelet coefficients are localized in ``items'' and ``scale''. They thus do not localize components of the signal in the ``space of rankings''. In other words each wavelet coefficient is not computed from a subset of the signal's coordinates that are ``close''. Instead, they are computed from subsets of coordinates that lead to the localization properties through the marginal operators that we described at length in the previous subsections. 

Algorithm \ref{alg:FWT-MRA} computes the wavelet transform of a function $F\in\Space{\Rank{A}}$ with $A\in\Subsets{\n}$. To extend it for any function $F\in\Space{\Gn}$, recall that $\Psi F = \sum_{A\in\Supp(F)}\Psi F_{A}$, where $\Supp(F) = \{A\in\SubsetsWE{\n} \;\vert\; F_{A} \neq 0 \}$ is the global support of $F$ (see Subsection \ref{subsec:definitions}). We naively extend the FWT by applying Algorithm \ref{alg:FWT-MRA} to each $F_{A}$ and summing all the wavelet transforms $\Psi F_{A}$. This procedure is formalized by Algorithm \ref{alg:FWT-MRA-naive}.

\begin{algorithm}
\caption{FWT for a function $F\in\Space{\Gn}$}
\label{alg:FWT-MRA-naive}
\begin{algorithmic}
\Require $F\in\Space{\Gn}$
\For{$A\in\Supp(F)$}
	\State Compute $\Psi F_{A}$ with Algorithm \ref{alg:FWT-MRA}
\EndFor
\State \Return $\Psi F = \sum_{A\in\Supp(F)}\Psi F_{A}$
\end{algorithmic}
\end{algorithm}

Algorithm \ref{alg:FWT-MRA-naive} is of course not optimal to compute the wavelet transform of any function $F\in\Gn$. Indeed, if there exists $B\in\Subsets{\n}$ included in at least two subsets of items in $\Supp(F)$, then the computation of the wavelet coefficients $\Psi_{B}F(\pi)$ for $\pi\in\Rank{B}$ will involve redundant applications of the high-pass filters of scale $\vert B\vert$ whereas it requires only one. The definition of an optimal FWT for any function $F\in\Space{\Gn}$ necessitates however to introduce new definitions and notations. For clarity's sake, we leave it to the reader. In addition, we assert that the optimal FWT would still have a complexity of same order of magnitude as the one of Algorithm \ref{alg:FWT-MRA-naive} (see below).

\ \\
{\bf Algorithmic complexity.} We now turn to the analysis of the complexity of the FWT and related computations. First, the high-pass filters $h_{A}^{j}$ are constructed from the alpha coefficients given by Definition \ref{def:alpha-coefficients}. The latter does not however provide an explicit formula for them. Fortunately, they can be precomputed once and for all with an efficient procedure provided in Section \ref{sec:MRA-construction}. The following proposition then gives an upper bound for its complexity.

\begin{proposition}[Complexity of the computation of alpha coefficients]
	\label{prop:complexity-alpha-coefficients}
	For $k\in\{2,\dots,n\}$, the computation of all coefficients $\alpha_{B}(\pi,\pi')$ for $\pi,\pi'\in\Rank{B}$ and $B\in\SubsetsWE{\n}$ with $\vert B\vert \leq k$ has complexity bounded by $(1/2)k^{2}k!$.
\end{proposition}

The proof of Proposition \ref{prop:complexity-alpha-coefficients} finely exploits the combinatorial structure of the operators of the MRA representation. It is postponed to the Appendix. Once the alpha coefficients and therefore the high-pass filters are precomputed, one can apply the FWT, the complexity of which is bounded by the following proposition. We recall that the support of a function $F\in\Space{\Gn}$ is defined by $\supp(F) = \{\pi\in\Gn \;\vert\; F(\pi) \neq 0\}$ whereas its global support is defined by $\Supp(F) = \{A\in\Subsets{\n} \;\vert\; F_{A}\neq 0\}$.

\begin{proposition}[Complexity of the FWT for the MRA representation]
\label{prop:complexity-wavelet-transform}
Let $F\in\Space{\Gn}$ and $k = \max\{\vert A\vert \;\vert\; A\in\Supp(F)\}$. The complexity of Algorithm \ref{alg:FWT-MRA-naive} applied to $F$ is bounded by
\[
\sum_{A\in\Supp(F)}[e\,\vert A\vert! + \vert A\vert(2^{\vert A\vert-1}-1)]\vert\supp(F_{A})\vert \quad\leq\quad [e\,k! + k(2^{k-1}-1)]\vert\supp(F)\vert.
\]
\end{proposition}

\begin{proof}
We first prove the proposition for a function $F\in\Space{\Rank{A}}$ with $A\in\Subsets{\n}$. Let $k = \vert A\vert$ and $j\in\{2,\dots,k\}$. At scale $j$, Algorithm \ref{alg:FWT-MRA} involves
\begin{itemize}
	\item the application of the high-pass filter $h_{A}^{j}$ on $M^{j}F$, with complexity equal to
	\[
	\sum_{B\subset A, \vert B\vert = j}\sum_{\pi\in\Rank{B}}\vert\supp(M_{B}F)\vert = j!\sum_{B\subset A,\vert B\vert = j}\vert\supp(M_{B}F)\vert;
	\]
	\item the application of the low-pass filter $g_{A}^{j}$ on $M^{j}F$, with complexity bounded by
	\[
	\sum_{\pi\in\Gamma_{A}^{j}}\mathbb{I}\{\pi\in\supp(M^{j}F)\}j = j\vert\supp(M^{j}F)\vert.
	\]
	Indeed, each coefficient $M^{j}F(\pi)$ for $\pi\in\Gamma_{A}^{j}$ is involved in the computation of at most $j$ approximation coefficients of scale $j-1$, namely the approximation coefficients $M^{j-1}F(\pi')$ for $\pi'\subset\pi$ with $\vert\pi'\vert = j-1$.
\end{itemize}
Now, it is easy to see that for any $B\in\Subsets{\n}$, $\vert\supp(M_{B}F)\vert \leq \vert\supp(F)\vert$. One therefore has $\vert\supp(M^{j}F)\vert = \sum_{B\subset A, \vert B\vert = j}\vert\supp(M_{B}F)\vert \leq \binom{k}{j}\vert\supp(F)\vert$ and the complexity of Algorithm \ref{alg:FWT-MRA} is bounded by
\[
\vert\supp(F)\vert\sum_{j=2}^{k}\binom{k}{j}\left(j! + j\right).
\]
Classic combinatorial calculations then give
\[
\sum_{j=2}^{k}\binom{k}{j}j! = \sum_{j=2}^{k}\frac{k!}{(k-j)!}\leq k!\sum_{j=0}^{+\infty}\frac{1}{j!} = e\, k! \qquad\text{and}\qquad \sum_{j=2}^{k}\binom{k}{j}j = k(2^{k-1}-1).
\]
For a function $F\in\Space{\Gn}$, the complexity of Algorithm \ref{alg:FWT-MRA-naive} is then clearly bounded by
\[
\sum_{A\in\Supp(F)}[e\,\vert A\vert! + \vert A\vert(2^{\vert A\vert-1}-1)]\vert\supp(F_{A})\vert \quad\leq\quad [e\,k! + k(2^{k-1}-1)]\vert\supp(F)\vert.
\]
\end{proof}

We finish this subsection with the analysis of the wavelet synthesis. In classic multiresolution analysis, the inverse wavelet transform can be computed with a ``dual'' procedure of the FWT. In the present context, it happens that the synthesis operator $\phi_{A}$ involves computations that are not similar to the ones involved in the wavelet transform (refer to Section \ref{sec:MRA-construction} for the definition). A simple procedure leads however to the following complexity bounds.

\begin{proposition}[Complexity of the wavelet synthesis]
\label{prop:complexity-synthesis-operator}
Let $A\in\Subsets{\n}$ and $\mathbf{X}\in\Hn$. The computation of $\phi_{A}\mathbf{X}(\pi)$ can be done with complexity bounded by $\binom{\vert A\vert}{2}$ for any $\pi\in\Rank{A}$, and the computation of $\phi_{A}\mathbf{X}$ with complexity bounded by $\vert A\vert ! \binom{\vert A\vert}{2}$.
\end{proposition}

Refer to the Appendix for the proof of Proposition \ref{prop:complexity-synthesis-operator}. The complexity bounds of Propositions \ref{prop:complexity-alpha-coefficients} and \ref{prop:complexity-wavelet-transform} can appear a little high at first glance, as they involve powers and factorials. We however point out that the value of the exponent or under the factorial is the size of the subset of items considered. This size is actually small in practical applications typically around $10$, and the complexity thus does not explode.

\begin{remark}[Connection with the Fourier transform]
In classic multiresolution analysis, the wavelet transform is connected to the Fourier transform. As we shall see in Section \ref{sec:connections}, it happens that some connections exist too in the present context between the MRA representation and $\Sn$-based harmonic analysis. The algorithms we introduced in this section do not however use the Fourier transform on $\Sn$ at all. The design of such procedures would certainly be an interesting direction for future work.
\end{remark}

\section{The MRA framework for the statistical analysis of incomplete rankings}
\label{sec:MRA-framework}

We now describe a general framework to apply the MRA representation to the statistical analysis of incomplete rankings, in the setting defined in Section \ref{sec:setting}.

\subsection{Identifiability issues}

In each of the statistical application mentioned in Subsection \ref{subsec:problems}, the goal is to recover a certain target part of $p$. In the context of full ranking analysis, one observes drawings of permutations $\Sigma_{1}, \dots, \Sigma_{N}$ that provide a direct access to global information about $p$. The task is then to best approximate the target part of $p$ from global information about $p$. In the context of incomplete ranking analysis, the target part of $p$ must be recovered from the observation of a dataset $\mathcal{D}_{N} = ((\mathbf{A}_{1},\Pi^{(1)}),\dots, (\mathbf{A}_{N},\Pi^{(N)}))$ where the $(\mathbf{A}_{i},\Pi^{(i)})$'s are drawn IID from the process \eqref{eq:scheme}. This brings an additional difficulty as information about $p$ in $\mathcal{D}_{N}$ is censored by the probability distribution $\nu$. One must therefore deal with two types of uncertainty: 
\begin{enumerate}
	\item Remove the noise from the observation process \eqref{eq:scheme} to access to information about $p$.
	\item Recover the target part of $p$ from the accessible part of information about $p$.
\end{enumerate}

By the law of large numbers, it is obvious that the (asymptotically) accessible part of information about $p$ (as $N$ grows to infinity) are the marginals $P_{A}$ for observable subsets of items $A$, that is to say subsets of items in the observation design $\mathcal{A}$. The second problem then boils down to recover the target part of $p$ from the knowledge of the marginals $(P_{A})_{A\in\mathcal{A}}$. 

Depending on the target part and the observation design $\mathcal{A}$, this task can require a structural assumption on $p$. Suppose for instance that one seeks to recover the full ranking model $p$ from the observation of pairwise comparisons only. In other words, with an observation design $\mathcal{A}$ included in the set of pairs of $\n$. Each pairwise marginal $P_{\{a,b\}}$ for $\{a,b\}\subset\n$ being a probability distribution on a set with two elements, it is characterized by one parameter. The number of accessible parameters is therefore at most $\binom{n}{2}$, whereas characterizing the full ranking model $p$ requires $n!-1$ parameters. This task thus requires to stipulate an additional structural assumption on $p$, so that $p$ becomes identifiable from the knowledge of its pairwise marginals only.

In a general context, we consider the following question: without any structural assumption, what part of $p$ can be recovered from the knowledge of the marginals $(P_{A})_{A\in\mathcal{A}}$? The following theorem provides the answer. It is already proved in \citet{SCJ2015} and is a direct consequence of Theorem \ref{th:inverse-linear-system}. Its proof is thus left to the reader.

\begin{theorem}[Identifiable parameters]
\label{th:accessible-information}
The knowledge of $(P_{A})_{A\in\mathcal{A}}$ characterizes the component 
\[
\left(\Psi_{B}p\right)_{B\in\SubsetsWE{\mathcal{A}}}\in\mathbb{H}(\SubsetsWE{\mathcal{A}})
\]
of the ranking model $p$. In particular, it has a number of degrees of freedom equal to $\dim\mathbb{H}(\SubsetsWE{\mathcal{A}}) = \sum_{B\in\SubsetsWE{\mathcal{A}}}d_{\vert B\vert}$.
\end{theorem}

Through Theorem \ref{th:accessible-information}, the MRA representation allows to quantify the part of $p$ that is identifiable without any structural assumption in the statistical setting introduced in Section \ref{sec:setting}. This justifies the general method we introduce for the statistical analysis of incomplete rankings.

\subsection{General method for the statistical analysis of incomplete rankings}
\label{subsec:general-method}

The MRA framework we now introduce is performed in two steps, one to perform each of the two tasks mentioned in the previous Subsection.

\begin{definition}[MRA framework]
The MRA framework for the statistical analysis of incomplete rankings is described by the following general procedure.
\begin{enumerate}
	\item Construct from the dataset $\mathcal{D}_{N}$ the \textit{wavelet empirical estimator} $\mathbf{\widehat{X}}\in\mathbb{H}(\SubsetsWE{\mathcal{A}})$ defined for each $B\in\SubsetsWE{A}$ as the simple average of the wavelet projections of the $\delta_{\Pi^{(i)}}$:
	\begin{equation}
	\label{eq:wavelet-estimator}
	\widehat{X}_{B} = \frac{1}{\vert\{1\leq i\leq N \;\vert\; B\subset\mathbf{A}_{i}\}\vert}\sum_{i=1}^{N}\Psi_{B}\delta_{\Pi^{(i)}}
	\end{equation}
	(we recall that $\Psi_{B}\delta_{\pi} = 0$ if $B\not\subset c(\pi)$ by construction). By convention, $\widehat{X}_{B} = 0$ if $\vert\{1\leq i\leq N \;\vert\; B\subset\mathbf{A}_{i}\}\vert = 0$. As shown in Subsection \ref{subsec:statistical-challenge}, $\mathbf{\widehat{X}}$ is an unbiased estimator of the accessible component $\left(\Psi_{B}p\right)_{B\in\SubsetsWE{\mathcal{A}}}$ of $p$. 
	\item Perform the task related to the considered application in the feature space $\Hn$ using $\mathbf{\widehat{X}}$ as empirical distribution.
\end{enumerate}
\end{definition}

\begin{remark}
The wavelet empirical estimator $\mathbf{\widehat{X}}$ is equal to the weighted least square estimator considered in \citet{SCJ2015} and denoted by $\widehat{X}^{WLS}$. We use a different notation here for simplicity's sake.
\end{remark}

Beyond this decomposition in two steps, the major novelty of the MRA framework is to offer the possibility to perform the analysis of the data in the feature space $\Hn$. This is a radical change from existing approaches that all rely on the construction of a ranking model $\widehat{p}_{N}$ over $\Sn$ (see Subsection \ref{subsec:existing-approaches}). Subsections \ref{subsec:statistical-challenge} and \ref{subsec:computational-challenge} respectively show how this method allows to overcome the statistical and computational challenges. Before that, we illustrate the application of the MRA framework on different applications.

\ \\
\noindent
{\bf Estimation.} Two estimation problems naturally arise when observing incomplete rankings: the estimation of the full ranking model $p$ or the estimation of the accessible marginals $(P_{A})_{A\in\mathcal{A}}$ only. In the latter, the target part of $p$ is accessible. The MRA framework can therefore be applied in its simplest form: construct the wavelet empirical estimator $\mathbf{\widehat{X}}$ and use it directly to estimate the marginals, taking $\phi_{A}\mathbf{\widehat{X}}$ as estimator of $P_{A}$ for each $A\in\mathcal{A}$. This approach is used in \citet{SCJ2015} and shown to have strong theoretical guarantees as well as a good performance in numerical applications. Depending on the dataset, it can nonetheless be useful to add a regularization procedure. Estimator $\mathbf{\widehat{X}}$ is indeed characterized by $\sum_{B\in\SubsetsWE{\mathcal{A}}}d_{\vert B\vert}$ independent parameters. If for instance $\mathcal{A} = \{A \subset\n \;\vert\; 2\leq \vert A\vert \leq K\}$ for some $K \in\{2,\dots,n\}$, then this quantity is of order $O(n^{K})$. It may thus require a huge number of observations $N$ available to attain a good accuracy on a large dataset. Regularization procedures can then help to obtain a more robust estimator. Many approaches are possible, we provide here two examples for illustration purpose\footnote{Examples to define the proposed mathematical objects are provided in Section \ref{sec:discussion}.}.
\begin{itemize}
	\item Kernel-based estimation: Given a distance $D$\ on the set $\SubsetsWE{\n}$ of subsets of items, one can define a kernel $K_{h}:\Hn\rightarrow\Hn$ that maps an element $\mathbf{X}\in\Hn$ to a smoother element $K_{h}\mathbf{X}\in\Hn$ with $h$ as a window parameter on the distance $D$, and consider the kernel-based wavelet estimator $\mathbf{\widehat{X}}^{Ker}$ defined for $B\in\SubsetsWE{\mathcal{A}}$ by
	\[
	\widehat{X}_{B}^{Ker} = \frac{1}{Z_{B,N}}\sum_{i=1}^{N}\sum_{B'\in\SubsetsWE{\mathcal{A}}}K_{h}(\Psi_{B'}\delta_{\Pi^{(i)}}),
	\]
	where $Z_{B,N}$ is a normalizing constant. 
	\item Penalty minimization: Given a distance $\Delta$ on $\Hn$, one can construct a regularized estimator as the solution $\mathbf{\widehat{X}}^{Pen}$ of a minimization problem of the form
	\[
	\min_{\mathbf{X'}\in\Hn} \Delta(\mathbf{X'},\mathbf{\widehat{X}}) + \lambda_{N}\Omega(\mathbf{X'}),
	\]
	where $\Omega:\Hn\rightarrow\mathbb{R}$ is a penalty function and $\lambda_{N}>0$ is a regularization parameter. 
\end{itemize}
Regularization procedures are discussed in more details in Section \ref{sec:discussion}. They are also required when one seeks to recover the full ranking model $p$. In this case, it may not be necessary to reduce the variance of the estimator $\mathbf{\widehat{X}}$ but the goal is to recover information that is not accessible in absence of any structural model assumption. It can be expressed as an inverse problem of the form: knowing an estimation $\mathbf{\widehat{X}}$ of $(\Psi_{B}p)_{B\in\SubsetsWE{\mathcal{A}}}$, recover $p$. This task of course requires a structural assumption on $p$ and can typically be tackled by minimizing a penalty function that quantifies it. In both cases, the MRA framework is applied the following way:
\begin{enumerate*}[label=\arabic*\upshape.]
	\item construct the wavelet empirical estimator $\mathbf{\widehat{X}}$;
	\item apply a regularization procedure to obtain a final estimator $\mathbf{\widehat{X}}^{\ast}$.
\end{enumerate*}

\ \\
\noindent
{\bf Clustering.} Here we consider a clustering problem. We assume that the observations from the dataset $\mathcal{D}_{N}$ come from a set of $m\geq 1$ users. For $j\in\{1,\dots,m\}$, user $j$ provides a dataset of $N_{j}\geq 1$ incomplete rankings that we denote by $\mathcal{D}_{N}^{j} = ((\mathbf{A}_{j,1},\Pi^{(j,1)}),\dots, (\mathbf{A}_{j,N_{j}},\Pi^{(j,N_{j})}))$, so that $\mathcal{D}_{N} = \mathcal{D}_{N}^{1}\sqcup\dots\sqcup\mathcal{D}_{N}^{m}$. We assume that each user $j$ is modeled by a ranking model $p_{j}$ over $\Sn$ and the goal is regroup the ranking models $p_{1},\dots,p_{m}$ into $k$ clusters, with $k\in\{1,\dots,m\}$ known in advance for simplicity. The difficulty is of course that none of the $p_{j}$'s is known, and not even accessible from the observations, since each dataset $\mathcal{D}_{N}^{j}$ is composed of censored thus incomplete rankings. The MRA framework can be applied as follows:
\begin{enumerate}
	\item For each user $j\in\{1,\dots,m\}$, compute the wavelet empirical estimator $\mathbf{\widehat{X}}^{j}$ defined for each $B\in\SubsetsWE{\mathcal{A}}$ by
	\[
	\widehat{X}_{B}^{j} = \frac{1}{\vert\{1\leq i\leq N_{j} \;\vert\; B\subset\mathbf{A}_{j,i}\}\vert}\sum_{i=1}^{N_{j}}\Psi_{B}\delta_{\Pi^{(j,i)}}.
	\]
	\item Apply a clustering algorithm to the data points $\mathbf{\widehat{X}}^{1},\dots,\mathbf{\widehat{X}}^{N}$ in the feature space $\Hn$. It can be for instance the $k$-means algorithm or a spectral clustering method based on a similarity measure on $\Hn$.
\end{enumerate}

\ \\
\noindent{\bf Ranking aggregation.} Ranking aggregation is certainly one the most considered applications in the ranking literature. Broadly speaking, it consists in ``summarizing'' a population of rankings into one single ranking. In the most classic setting, the population of rankings is a finite collection $(\sigma^{(1)},\dots,\sigma^{(N)})$ of full rankings and the goal is to summarize it into one full ranking $\sigma$ whose performance is measured by the following cost function
\begin{equation}
\label{eq:empirical-cost-function}
\sum_{i=1}^{N}d\left(\sigma,\sigma^{(i)}\right),
\end{equation}
where $d$ is a distance on $\Sn$. Minimizers of \eqref{eq:empirical-cost-function} are called consensus rankings of the collection $(\sigma^{(1)},\dots,\sigma^{(N)})$ for the distance $d$. Though this problem has mainly been considered in a deterministic setting in most of the dedicated literature, it can be naturally extended to a statistical setting where the population is a collection of $N$ random permutations $(\Sigma^{(1)},\dots,\Sigma^{(N)})$ drawn IID from a ranking model $p$. The aggregation performance of a full ranking $\sigma$ is then measured by the expected cost function
\begin{equation}
\label{eq:expected-cost-function}
\mathbb{E}_{\Sigma\sim p}\left[d(\sigma,\Sigma)\right] = \sum_{\sigma'\in\Sn}d(\sigma,\sigma')p(\sigma').
\end{equation}
Consensuses for \eqref{eq:expected-cost-function} are for instance considered in \citet{Sibony2014} or \citet{PPR2015}. A possible aggregation procedure in this context is for instance to take a minimizer of the empirical cost function $\sum_{i=1}^{N}d(\sigma,\Sigma^{(i)})$. In the context of the statistical analysis of incomplete rankings, one does not have access to drawings of $p$ but to a dataset $\mathcal{D}_{N}$ only. It is still natural however to consider the problem of aggregating the statistical population of rankings into a full ranking $\sigma$ and measure its performance by the same function \eqref{eq:expected-cost-function}. This setting is used for instance in \citet{RGLA15}, with the Kendall's tau distance, in the context of pairwise comparisons. In a general setting, the MRA framework applies as follows:
\begin{enumerate}
	\item Compute the wavelet empirical estimator $\mathbf{\widehat{X}}$ defined by \eqref{eq:wavelet-estimator}.
	\item Take the minimizer of the cost function
	\[
	\Delta_{d}\left(\Psi\delta_{\sigma},\mathbf{X}\right),
	\]
	where $\Delta_{d}$ is a distance on $\Hn$ that can be defined from $d$\footnote{How to define a distance $\Delta_{d}$ on $\Hn$ that would lead to efficient procedures requires however some deeper analysis and is left to future work.}.
\end{enumerate}

\subsection{Overcoming the statistical challenge}
\label{subsec:statistical-challenge}

We now describe the advantages of the MRA framework for the statistical analysis of incomplete rankings. First, $\mathbf{\widehat{X}}$ is an unbiased estimator of $(\Psi_{B}p)_{B\in\SubsetsWE{A}}$.

\begin{proposition}[Expectation of the wavelet empirical estimator]
\label{prop:unbiased}
For all $B\in\SubsetsWE{\mathcal{A}}$,
\[
\mathbb{E}\left[\widehat{X}_{B}\right] = \Psi_{B}p.
\]
\end{proposition}

\begin{proof}
Let $B\in\SubsetsWE{\mathcal{A}}$. Denoting by $\mathcal{B}_{N}^{\nu}$ the $\sigma$-algebra generated by the collection of random variables $(\mathbf{A}_{1},\dots,\mathbf{A}_{N})$, one has by definition
\begin{align*}
\mathbb{E}\left[\widehat{X}_{B}\right] 
&= \mathbb{E}\left[\mathbb{E}\left[\frac{1}{\vert\{1\leq i\leq N \;\vert\; B\subset\mathbf{A}_{i}\}\vert}\sum_{i=1}^{N}\mathbb{I}\{B\subset\mathbf{A}_{i}\}\Psi_{B}\delta_{\Pi^{(i)}}\Bigg\vert\mathcal{B}_{N}^{\nu}\right]\right]\\
&= \mathbb{E}\left[\frac{1}{\vert\{1\leq i\leq N \;\vert\; B\subset\mathbf{A}_{i}\}\vert}\sum_{i=1}^{N}\mathbb{I}\{B\subset\mathbf{A}_{i}\}\mathbb{E}\left[\Psi_{B}\delta_{\Pi^{(i)}}\Big\vert\mathcal{B}_{N}^{\nu}\right]\right].
\end{align*}
Now, reformulation \eqref{eq:censoring-process} of the statistical process \eqref{eq:scheme} ensures that for each $i\in\{1,\dots,N\}$, $\Pi^{(i)}$ has the same law as $\Sigma^{(i)}_{\vert \mathbf{A}_{i}}$, where $\Sigma^{(1)}, \dots,\Sigma^{(N)}$ are random permutations drawn IID from $p$. We recall in addition that for any permutation $\sigma\in\Sn$ and any subset $A\in\Subsets{\n}$ with $B\subset A$, one has $\Psi_{B}\delta_{\sigma_{\vert A}} = \Psi_{B}\delta_{\sigma}$ by Property \eqref{eq:MRA-general-2} of Theorem \ref{th:MRA-general}. One therefore has
\[
\mathbb{E}\left[\Psi_{B}\delta_{\Pi^{(i)}}\Big\vert\mathcal{B}_{N}^{\nu}\right] = \mathbb{E}\left[\Psi_{B}\delta_{\Sigma^{(i)}_{\vert \mathbf{A}_{i}}}\Big\vert\mathcal{B}_{N}^{\nu}\right] = \mathbb{E}\left[\Psi_{B}\delta_{\Sigma^{(i)}}\Big\vert\mathcal{B}_{N}^{\nu}\right] = \mathbb{E}\left[\Psi_{B}\delta_{\Sigma^{(i)}}\right] = \Psi_{B}\mathbb{E}\left[\delta_{\Sigma}\right] = \Psi_{B}p.
\]
This concludes the proof.
\end{proof}

Proposition \ref{prop:unbiased} ensures that $\mathbf{\widehat{X}}$ is a good representative of the accessible part $(\Psi_{B}p)_{B\in\SubsetsWE{\mathcal{A}}}$ of the ranking model $p$, whatever it is. This advantage is to be compared to existing methods:
\begin{itemize}
	\item Methods based on parametric models are necessarily biased when the ranking model does not satisfy the structural assumption.
	\item Methods that identify an incomplete ranking with the set of its linear extensions are fundamentally biased by the censoring process $\nu$, as shown in Subsection \ref{subsec:existing-approaches}.
\end{itemize} 
In a sense, one can say that the MRA framework allows to remove the noise due to the censoring process $\nu$ whatever the ranking model $p$.

The other statistical advantage of the MRA framework is that it allows to fully exploit the consistency assumption \eqref{eq:consistency-assumption}. As explained in Subsection \ref{subsec:challenges}, the consistency assumption induces two rules to transfer information between subsets of items $A,B\in\Subsets{\n}$ with $B\subset A$: information is transferred from $A$ to $B$ through the marginal operator $M_{B}$, and information is transferred from $B$ to $A$ as the constraint that $P_{A}$ must satisfy $M_{B}P_{A} = P_{B}$. By Theorem \ref{th:inverse-linear-system}, this constraint is equivalent to $\Psi_{B'}P_{A} = \Psi_{B'}P_{B}$ for all $B'\in\SubsetsWE{B}$. The second rule can thus be reformulated as: information is transferred from $B$ to $A$ through the operators $(\Psi_{B'})_{B'\in\SubsetsWE{B}}$. In other words, the MRA representation allows to quantifies the amount of information in the constraints imposed by the consistency assumption. The wavelet empirical estimator $\mathbf{\widehat{X}}$ therefore naturally exploits more information than other empirical estimators, as illustrated by the following comparison.
\begin{itemize}
	\item {\bf Naive empirical estimator.} For an observed subset $A$ ($\vert\{1\leq i\leq N \;\vert\; A = \mathbf{A}_{i}\}\vert > 0$), we recall that the naive empirical estimator is defined in \eqref{eq:empirical-marginal} by
	\[
	\widehat{P_{A}} = \frac{1}{\vert\{1\leq i\leq N \;\vert\; A = \mathbf{A}_{i}\}\vert}\sum_{i=1}^{N}\mathbb{I}\{A = \mathbf{A}_{i} \}\delta_{\Pi^{(i)}}.
	\]
	The $\widehat{P_{A}}$'s are two-by-two independent. Each $\widehat{P_{A}}$ consolidates information on $A$ but no information is transferred between subsets. In other words, the naive empirical estimator does not exploit the consistency assumption at all. For instance if rankings are observed on $\{1,2\}$ and $\{1,2,3\}$, neither information is transferred from $\{1,2,3\}$ to $\{1,2\}$ nor in the other way round.
	\item {\bf Marginal-based empirical estimator.} For a subset $B\in\SubsetsWE{\n}$ included in at least one observed subset ($\vert\{1\leq i\leq N \;\vert\; B\subset\mathbf{A}_{i}\}\vert > 0$), we define the marginal-based empirical estimator by
	\[
	\widehat{Q_{B}} = \frac{1}{\vert\{1\leq i\leq N \;\vert\; B\subset\mathbf{A}_{i}\}\vert}\sum_{i=1}^{N}\mathbb{I}\{B\subset\mathbf{A}_{i}\}M_{B}\delta_{\Pi^{(i)}}.	
	\]
	The marginal-based empirical estimator exploits the consistency assumption but only in one sense, from a subset of item $A$ to its subsets $B\in\SubsetsWE{\n}$. For instance if rankings are observed on $\{1,2\}$ and $\{1,2,3\}$, information is transferred from $\{1,2,3\}$ to $\{1,2\}$ but not in the other way round.
	\item {\bf Wavelet empirical estimator.} For a subset $B\in\SubsetsWE{\n}$ included in at least one observed subset ($\vert\{1\leq i\leq N \;\vert\; B\subset\mathbf{A}_{i}\}\vert > 0$), we recall that the wavelet empirical estimator is defined by 
	\[
	\widehat{X}_{B} = \frac{1}{\vert\{1\leq i\leq N \;\vert\; B\subset\mathbf{A}_{i}\}\vert}\sum_{i=1}^{N}\mathbb{I}\{B\subset\mathbf{A}_{i}\}\Psi_{B}\delta_{\Pi^{(i)}}.
	\]
	Thanks to the wavelet transform, the wavelet empirical estimator fully exploits the consistency assumption. For instance if rankings are observed on $\{1,2\}$ and $\{1,2,3\}$, information is transferred from $\{1,2,3\}$ to $\{1,2\}$ and in the other way round.
\end{itemize}

\subsection{Overcoming the computational challenge}
\label{subsec:computational-challenge}

The following proposition gives a theoretical bound on the complexity of the computation of the wavelet empirical estimator $\mathbf{\widehat{X}}$.

\begin{proposition}[Complexity of the computation of the wavelet empirical estimator]
\label{prop:complexity-wavelet-estimator}
Let $K = \max_{A\in\mathcal{A}}\vert A\vert$. The complexity of the computation of $\mathbf{\widehat{X}}$ is bounded by
\[
[e\, K! + (K+4)2^{K-1}]\min\left(N,\sum_{A\in\mathcal{A}}\vert A\vert !\right).
\]
\end{proposition}

\begin{proof}
Defining the function $F_{N} = \sum_{i=1}^{N}\delta_{\Pi^{(i)}}$ and the scalars $Z_{N,B} = \vert\{1\leq i\leq N \;\vert\; B\subset\mathbf{A}_{i}\}\vert$, one has for any $B\in\SubsetsWE{\mathcal{A}}$,
\[
\widehat{X}_{B} = \frac{1}{Z_{N,B}}\Psi_{B}F_{N}.
\]
The computation of $\mathbf{\widehat{X}}$ can thus be decomposed into three steps:
\begin{enumerate}
	\item Computation of $F_{N}$ and $(Z_{N,B})_{B\in\SubsetsWE{\mathcal{A}}}$: this is performed in one loop over the dataset with complexity bounded by
	\[
	\sum_{\pi\in\supp(F_{N})}\vert\SubsetsWE{c(\pi)}\vert \leq 2^{K}\vert\supp(F_{N})\vert.
	\]
	\item Computation of $\Psi F_{N}$: this is performed using Algorithm \ref{alg:FWT-MRA-naive}. By Proposition \ref{prop:complexity-wavelet-transform}, its complexity is bounded by 
	\[
	[e\, K! + K2^{K-1}]\vert\supp(F_{N})\vert.
	\]
	\item Division of $\Psi_{B}F_{N}$ by $Z_{N,B}$ for each $B\in\SubsetsWE{\mathcal{A}}$ such that $Z_{N,B}\neq 0$: this is performed in one loop over the subsets $B$ with $Z_{N,B}> 0$ with complexity bounded by
	\[
	\vert\SubsetsWE{\Supp(F_{N})}\vert \leq 2^{K}\vert\supp(F_{N})\vert.
	\]
\end{enumerate}
To conclude the proof, notice that $\vert\supp(F_{N})\vert$ is exactly the number of parameters required to store the dataset $\mathcal{D}_{N}$. Lemma \ref{lem:dataset-storage} therefore ensures that it is bounded by $\min (N,\sum_{A\in\mathcal{A}}\vert A\vert !)$.
\end{proof}

Although the bound in Proposition \ref{prop:complexity-wavelet-estimator} is not small, it is sufficient to ensure that the computation of the wavelet empirical estimator is tractable in common situations. In practical applications indeed, the number of items $n$ can be large, say around $10^{4}$, but the parameter $K$, which represents the maximal size of an observed ranking, is fairly small, typically less than $10$. The factor $[e\, K! + K2^{K-1}]$ then does not represent too much of an issue. On the other hand, the term $\min\left(N,\sum_{A\in\mathcal{A}}\vert A\vert !\right)$ is smaller than the number $N$ of observations, which is always tractable. For instance if one has a dataset of one billion rankings that each involve less than $5$ items then the number of required operations is bounded by $5\times 10^{11}$, which is still tractable. We also point out that the wavelet empirical estimator can be easily computed in an on-line mode or parallelized in a map/reduce framework (refer to \citet{SCJ2014} for more details).

From a theoretical point of view, the interesting aspect of the bound in Proposition \ref{prop:complexity-wavelet-estimator} is that it does not depend directly on the number of items $n$. Only the term $\sum_{A\in\mathcal{A}}\vert A\vert !$ can indeed depend on $n$ through the observation design $\mathcal{A}$, as explained in Subsection \ref{subsec:impact-of-nu}. More particularly, this term is exactly the bound on the number of parameters required to store the dataset $\mathcal{D}_{N}$ from Lemma \ref{lem:dataset-storage}. We can therefore say in a sense that the computation of the wavelet empirical estimator deals with the complexity of the data itself.

More generally, this can be considered as the great achievement of the MRA framework. As explained in Subsection \ref{subsec:challenges}, the analysis of incomplete rankings necessarily involves at some point the computation of the marginal $M_{A}q$ of a ranking model over $\Sn$ on a subset of items $A\in\Subsets{\n}$. If $q$ is represented as the vector of its values $(q(\sigma))_{\sigma\in\Sn}$, the computation of $M_{A}q(\pi)$ for $\pi\in\Rank{A}$ using Formula \eqref{eq:consistency-assumption} requires $n!/\vert A\vert!$ operations. Now, if $q$ is represented by its wavelet transform $\Psi q$, Theorem \ref{th:MRA-general} tells us that $M_{A}q(\pi) = \phi_{A}\Psi q(\pi)$. The computation then has complexity bounded by $\binom{\vert A\vert}{2}$, by Proposition \ref{prop:complexity-synthesis-operator}. This bound shows that the dependency in $n$ is an artifact of the theoretical framework of ranking models over $\Sn$: when the ranking model is not represented as a function on $\Sn$ but by its wavelet transform, this dependency vanishes.

\ \\
This section has shown that the MRA framework for the statistical analysis of incomplete rankings offers at the same time a great flexibility to define new approaches for a wide variety of applications and great advantages to face the inherent statistical and computational challenges. All of this is due to the strong properties of the MRA representation. As shall be explained in the following Section, its construction relies on recent results from algebraic topology in order to exploit accurately the multi-scale structure of incomplete rankings.

\section{The construction of the MRA representation}
\label{sec:MRA-construction}

We now define rigorously the objects of the MRA representation and establish its properties. Here and throughout the article, the null space of any operator $T$ is denoted by $\ker T$.

\subsection{The multiresolution decomposition}

The construction of the MRA representation starts with the definition of the spaces $H_{B}$ and the wavelet synthesis operators $\phi_{A}$ for $A,B\in\SubsetsWE{\n}$.

\begin{definition}[Spaces $H_{B}$]
\label{def:space-H}
We set $H_{\emptyset} = \mathbb{R}\bar{0} = \Space{\Rank{\bar{0}}}$ and define for $B\in\Subsets{\n}$ the linear space
\[
H_{B} = \{ F\in L(\Rank{B}) \;\vert\; M_{B'}F = 0 \text{ for all } B'\subsetneq B \} = \Space{\Rank{B}}\cap\bigcap_{B'\subsetneq B}\ker M_{B'}.
\]
\end{definition}

We recall that the feature space is then equal to $\Hn = \bigoplus_{B\in\SubsetsWE{\n}}H_{B}$. The definition of the spaces $H_{B}$ for $B\in\SubsetsWE{\n}$ is rather natural to obtain the properties of the MRA representation. Indeed two functions $F$ and $G$ in $L(\Rank{B})$ have the same marginals on all strict subsets of $B$ if and only if $F-G\in H_{B}$. Thus the projection of $F$ onto $H_{B}$ (in parallel to any space supplementary to $H_{B}$) contains information about $F$ that is specific to $B$. Equivalently, it localizes the piece of information of scale $\vert B\vert$ of $F$ on $B$. It is then natural to expect that for any $A\in\Subsets{\n}$, the structure of the space $\Space{\Rank{A}}$ is somehow equivalent to that of the sum of spaces $\bigoplus_{B\in\SubsetsWE{A}}H_{B}$. 

By contrast, the definition of the wavelet synthesis operators is not intuitive. It relies on the following concept: word $\pi'\in\Gamma_{n}$ is a \textit{contiguous subword} of word $\pi\in\Gamma_{n}$ if there exists $i\in\{1,\dots,\vert \pi\vert-\vert\pi'\vert+1\}$ such that $\pi' = \pi_{i}\pi_{i+1}\dots\pi_{i+\vert\pi'\vert-1}$. This is denoted by $\pi' \sqsubset \pi$. 

\begin{definition}[Operators $\phi_{A}$]
\label{def:wavelet-synthesis-operator}
For $A\in\SubsetsWE{\n}$, we define the linear operator $\phi_{A} : \Space{\Gn} \rightarrow \Space{\Rank{A}}$ on the Dirac function of a ranking $\pi\in\Gn$ by
\[
\phi_{A}\delta_{\bar{0}} = \frac{1}{\vert A\vert !}\1{\Rank{A}} \qquad\text{and}\qquad \phi_{A}\delta_{\pi} = \frac{1}{(\vert A\vert - \vert\pi\vert + 1)!}\1{\{\sigma\in\Rank{A} \;\vert\; \pi\,\sqsubset\,\sigma\}} \quad\text{if}\quad \pi\neq\bar{0}.
\]
\end{definition}

Notice that we define the operators $\phi_{A}$ globally on $\Space{\Gn}$ and not just on the feature space $\Hn$. This will allow us to highlight the key ingredients in the construction of the MRA representation. By Definition \ref{def:wavelet-synthesis-operator}, $\phi_{A}F_{B} = 0$ for $F_{B}\in\Space{B}$ with $B\not\subset A$. The operator $\phi_{A}$ can therefore be seen as an embedding operator from $\bigoplus_{B\in\SubsetsWE{A}}\Space{\Rank{B}}$ to $\Space{\Rank{A}}$. The following theorem exploits the properties of both spaces $H_{B}$ and operators $\phi_{A}$. It is the basis of the entire MRA representation.

\begin{theorem}[Multiresolution decomposition]
\label{th:MRA-decomposition}
For any $A\in\SubsetsWE{\n}$, one has the decomposition
\[
\Space{\Rank{A}} = \bigoplus_{B\in\SubsetsWE{A}}\phi_{A}\left(H_{B}\right).
\]
In addition, for $B\in\SubsetsWE{A}$,
\begin{enumerate}
	\item $\phi_{A}$ is injective on $H_{B}$: $\ker\phi_{A}\cap H_{B} = \{0\}$,
	\item for all $F\in H_{B}$ and $A'\in\SubsetsWE{A}$, $M_{A'}\phi_{A}F = \phi_{A'}F$,
	\item $\dim H_{B} = d_{\vert B\vert}$, where for $k\in\{2,\dots,n\}$, $d_{k}$ is the number of fixed-point free permutations (also called \textit{derangements}) on a set with $k$ elements.
\end{enumerate}
\end{theorem}

The proof of Theorem \ref{th:MRA-decomposition} relies on two key properties, one about the spaces $H_{B}$ and the other about the operators $\phi_{A}$. We start with the latter, given by the following lemma. For $A\subset\n$ with $\vert A\vert = 1$ we set by convention $\Space{\Rank{A}} = H_{\emptyset}$ and $M_{A} = M_{\emptyset}$. 

\begin{lemma}[Commutation between marginal and wavelet synthesis operators]
	\label{lem:commutativity}
	Let $A,B\in\Subsets{\n}$, $F\in L(\Rank{A})$ and $C\in\Subsets{\n}$ such that $A\cup B\subset C$. Then $M_{B}\phi_{C}F = \phi_{B}M_{A\cap B}F$. In other words, the following diagram is commutative.
	\begin{center}
		\begin{tikzpicture}
		\node (a) at (-3,0) {$L(\Rank{A})$};
		\node (b) at (0,1.5) {$L(\Rank{C})$};
		\node (c) at (0,-1.5) {$L(\Rank{A\cap B})$};
		\node (d) at (3,0) {$L(\Rank{B})$};
		\path[->] (a) edge node[auto] {$\phi_{C}$} (b);
		\path[->] (b) edge node[auto] {$M_{B}$} (d);
		\path[->] (a) edge node[auto] {$M_{A\cap B}$} (c);
		\path[->] (c) edge node[auto] {$\phi_{B}$} (d);
		\end{tikzpicture}
	\end{center}
	The diagram actually represents the restrictions of the operators to the involved spaces but we do not notify them for clarity's sake.
\end{lemma}

Lemma \ref{lem:commutativity} says in a way that the embedding operators $\phi_{A}$ commute with the marginal operators $M_{B}$. Notice in particular that if $\vert A\cap B\vert \leq 1$, $M_{B}\phi_{C}F = \phi_{B}M_{\emptyset}F$ is the constant function on $\Space{\Rank{B}}$ equal to $\sum_{\pi\in\Rank{A}}F(\pi)$. The proof of Lemma \ref{lem:commutativity} is purely technical and left to the Appendix. We however provide an illustrating example.
\begin{example}
\label{ex:lemma-illustration}
Let $A = \{1,2,3\}$, $B = \{1,2,4\}$ and $C = \{1,2,3,4\}$. Then for $\pi = 123$ for instance,
\begin{align*}
	M_{B}\phi_{C}\delta_{\pi} &= M_{\{1,2,4\}}\phi_{\{1,2,3,4\}}\delta_{123} = \frac{1}{2}M_{\{1,2,4\}}\left[ \delta_{4123} + \delta_{1234} \right] = \frac{1}{2}\left[ \delta_{412} + \delta_{124}\right]\\
	\text{and}\qquad \phi_{B}M_{A\cap B}\delta_{\pi} &= \phi_{\{1,2,4\}}M_{\{1,2\}}\delta_{123} = \phi_{\{1,2,4\}}\delta_{12} = \frac{1}{2}\left[ \delta_{412} + \delta_{124}\right].
\end{align*}
\end{example}

Lemma \ref{lem:commutativity} allows to prove, for $A\in\SubsetsWE{\n}$, the three following properties.
\begin{enumerate}
	\item For $B\in\SubsetsWE{A}$, $\phi_{A}$ is injective on $H_{B}$, \textit{i.e.} $\ker\phi_{A}\cap H_{B} = \{0\}$.
	\item For $B\in\SubsetsWE{A}$, $F\in H_{B}$ and $A'\in\SubsetsWE{A}$, $M_{A'}\phi_{A}F = \phi_{A'}F$.
	\item The sum of spaces $(\phi_{A}(H_{B}))_{B\in\SubsetsWE{A}}$ is direct.
\end{enumerate}

\begin{proof}
We prove each property separately.
\begin{enumerate}
	\item Let $F\in\ker\phi_{A}\cap H_{B}$. Applying Lemma \ref{lem:commutativity} to $A,B := B$ and $C := A$ gives
	\[
	\phi_{B}M_{B}F = M_{B}\phi_{A}F \qquad\textit{i.e.}\qquad F = 0\qquad\text{because } F\in\ker\phi_{A},
	\]
	which concludes the proof.
	\item Applying Lemma \ref{lem:commutativity} to $A := B$, $B := A'$ and $C := A$ gives
	\[
	M_{A'}\phi_{A}F = \phi_{A'}M_{B\cap A'}F.
	\]
	If $B\subset A'$ then $B\cap A' = A'$ and one obtains $M_{A'}\phi_{A}F = \phi_{A'}F$. If $B\not\subset A'$ then $B\cap A' \varsubsetneq B$ and $M_{B\cap A'}F = 0$ because $F\in H_{B}$. Hence  $M_{A'}\phi_{A}F = 0 = \phi_{A'}F$.
	\item Let $(F_{B})_{B\in\SubsetsWE{A}}\in\bigoplus_{B\in\SubsetsWE{A}}H_{B}$ such that 
	\begin{equation}
	\label{eq:local-equation}
	\sum_{B\in\SubsetsWE{A}}\phi_{A}F_{B} = 0.
	\end{equation}
	We need to show that $F_{B} = 0$ for each $B\in\SubsetsWE{A}$. We do it recursively on $\vert B\vert$ by applying property $2.$ to \eqref{eq:local-equation} for different subsets $A'$. First, applying $M_{\emptyset}$ cancels all the terms $\phi_{A}F_{B}$ for $B\in\Subsets{A}$, leading to $F_{\emptyset} = 0$. Then for any $A'\subset A$ with $\vert A'\vert = 2$, applying $M_{A'}$ cancels all the terms $\phi_{A}F_{B}$ for $B\in\Subsets{A}\setminus\{A'\}$, leading to $F_{A'} = 0$. The proof is concluded by induction.
\end{enumerate}
\end{proof}

The second key ingredient of the proof of Theorem \ref{th:MRA-decomposition} is the following theorem. We recall that for $k\in\{2,\dots,n\}$, $d_{k}$ is the number of derangements on a set of $k$ elements.

\begin{theorem}[Dimension of the space $H_{\set{k}}$]
	\label{th:topology}
	For $k\in\{2,\dots,n\}$, $\dim H_{\set{k}} = d_{k}$.
\end{theorem}

Theorem \ref{th:topology} is proved in \citet{Reiner13}, where $H_{\set{k}}$ is denoted by $\ker\pi_{\set{k}}$ (see proposition 6.8 and corollary 6.15). As simple as it may seem, this result is far from being trivial. It is actually shown in \citet{Reiner13} that $H_{\set{k}}$ is isomorphic to the top homology space of the complex of injective words on $\set{k}$. The calculation of the dimension of the latter relies on the Hopf trace formula for virtual characters and the topological properties of the partial order of subword inclusion, proved in several contributions of the algebraic topology literature \citep[see][]{Farmer78, BW83, Reiner04}.

Theorem \ref{th:topology} allows to conclude the proof of Theorem \ref{th:MRA-decomposition} with a dimensional argument. First observe that for $k\in\{2,\dots,n\}$, all the spaces $H_{B}$ for $B\subset\n$ with $\vert B\vert = k$ are isomorphic to $H_{\set{k}}$. Thus $\dim H_{B} = d_{\vert B\vert}$ for all $B\in\SubsetsWE{\n}$. Combining this result with properties $1.$ and $3.$, one obtains for any $A\in\Subsets{\n}$,
\begin{equation}
\label{eq:dimensional-argument}
\vert A\vert ! = \dim L(\Rank{A}) \geq \dim \bigoplus_{B\in\SubsetsWE{A}}\phi_{A}\left(H_{B}\right) \geq \sum_{B\in\SubsetsWE{A}}d_{\vert B\vert} = \sum_{k=0}^{\vert A\vert}\binom{\vert A\vert}{k}d_{k} = \vert A\vert !,
\end{equation}
where the last equality is a classic result in elementary combinatorics. All the inequalities in \eqref{eq:dimensional-argument} are therefore equalities, and the proof of Theorem \ref{th:MRA-decomposition} is finished.

\subsection{The construction of the wavelet transform}

Theorem \ref{th:MRA-decomposition} allows to construct implicitly the wavelet transform as follows: for any $A\in\SubsetsWE{\n}$ and $F\in\Rank{A}$, it shows the existence of a unique element $(\Psi_{B}^{A}F)_{B\in\SubsetsWE{A}}$ in $\bigoplus_{B\in\SubsetsWE{A}}H_{B}$ such that 
\[
F = \sum_{B\in\SubsetsWE{A}}\phi_{A}\Psi^{A}_{B}F.
\]
This naturally defines for any $B\in\SubsetsWE{\n}$ the linear operator $\Psi_{B}:\Space{\Gn}\rightarrow H_{B}$ on each subspace $L(\Rank{A})$ for $A\in\Subsets{\n}$ as the mapping
\begin{equation}
\label{eq:wavelet-projection}
\Psi_{B} : F \mapsto \Psi_{B}^{A}F\ \text{ if } B\subset A\text{ and } 0\text{ otherwise}.
\end{equation}

\begin{definition}[Wavelet transform]
\label{def:wavelet-transform}
The wavelet transform is the operator $\Psi:\Space{\Gn}\rightarrow\Hn$ constructed from the operators $\Psi_{B}$ defined in \eqref{eq:wavelet-projection} as
\[
\Psi:F\mapsto\left(\Psi_{B}F\right)_{B\in\SubsetsWE{\n}}.
\]
\end{definition}

All the objects of the MRA representation being defined, we now prove Theorem \ref{th:MRA-general}.

\begin{proof}
Property \eqref{eq:MRA-general-1} and the first part of Property \eqref{eq:MRA-general-2} are direct consequences of Theorem \ref{th:MRA-decomposition} and Definition \ref{def:wavelet-transform}. To prove the second part of Property \eqref{eq:MRA-general-2}, observe that the first part applied to $F$ gives $M_{A'}F = \sum_{B\in\SubsetsWE{A'}}\phi_{A'}\Psi_{B}F$ and \eqref{eq:MRA-general-1} applied to $M_{A'}F$ gives $M_{A'}F = \sum_{B\in\SubsetsWE{A'}}\phi_{A'}\Psi_{B}M_{A'}F$. The uniqueness of the decomposition concludes the proof.
\end{proof}

Definition \ref{def:wavelet-transform} relies on an implicit construction. We now provide an explicit construction of the wavelet transform. First, observe that Property \eqref{eq:MRA-general-1} of Theorem \ref{th:MRA-general} applied to $A=\emptyset$ implies that for any $F\in\Space{\Rank{\bar{0}}}$, $\Psi_{\emptyset}F = F$. Applying Property \eqref{eq:MRA-general-2}, one obtains for any $F\in\Space{\Gn}$,
\begin{equation}
\label{eq:initialization}
\Psi_{\emptyset}F = \Psi_{\emptyset}M_{\emptyset}F = M_{\emptyset}F = \left(\sum_{\pi\in\Gn}F(\pi)\right)\delta_{\bar{0}}.
\end{equation}
On the other hand, one has $\phi_{A}F = F$ for any $A\in\SubsetsWE{\n}$ and $F\in\Space{\Rank{A}}$, so that by Theorem \ref{th:MRA-general},
\begin{equation}
\label{eq:induction-formula}
\Psi_{A}F \quad=\quad F\ - \sum_{B\in\SubsetsWE{A}\setminus\{A\}}\phi_{A}\Psi_{B}F.
\end{equation}
We use Eq. \eqref{eq:initialization} and \eqref{eq:induction-formula} to construct the wavelet projections $\Psi_{B}$ by induction. We actually construct by induction the coefficients $\alpha_{B}(\pi,\pi')$ introduced in Definition \ref{def:alpha-coefficients}. The calculation first relies on the following lemma. For a ranking $\pi = \pi_{1}\dots\pi_{k}\in\Gamma_{n}$ and two indexes $1\leq i < j \leq k$, we denote by $\pi_{\set{i,j}}$ the contiguous subword $\pi_{i}\dots\pi_{j}$ of $\pi$.

\begin{lemma}
\label{lem:useful-lemma}
Let $A\in\SubsetsWE{\n}$ with $\vert A\vert = k$ and $\mathbf{X} = (X_{B})_{B\in\SubsetsWE{A}}\in\Hn$. Then for all $\pi\in\Rank{A}$,
\[
\sum_{B\in\SubsetsWE{A}}\phi_{A}X_{B}(\pi) = \frac{1}{k!}X_{\emptyset}(\bar{0}) + \sum_{1\leq i < j \leq k}\frac{1}{(k-j+i)!}X_{c(\pi_{\set{i, j}})}\left(\pi_{\set{i, j}}\right).
\]
\end{lemma}

\begin{proof}
First, one clearly has $\sum_{B\in\SubsetsWE{A}}\phi_{A}X_{B}(\pi) = \frac{1}{k!}X_{\emptyset}(\bar{0}) + \sum_{B\in\Subsets{A}}\phi_{A}X_{B}(\pi)$. Now by definition of operator $\phi_{A}$, one has for any $B\in\Subsets{A}$
\[
\phi_{A}X_{B}(\pi) 	= \sum_{\pi^{\prime}\in\Rank{B}}X_{B}(\pi')\frac{\mathbb{I}\{ \pi'\sqsubset\pi\}}{(k-\vert\pi'\vert+1)!} = X_{B}(\pi_{\vert B})\frac{\mathbb{I}\{ \pi_{\vert B}\sqsubset\pi\}}{(k-\vert B\vert+1)!}.
\]
Thus only the terms $\phi_{A}X_{B}(\pi)$ where $B$ is such that $\pi_{\vert B}$ is a contiguous subword of $\pi$ are potentially not null in the sum $\sum_{B\in\Subsets{A}}\phi_{A}X_{B}(\pi)$. As the contiguous subwords of $\pi$ are all of the form $\pi_{\set{i,j}}$ with $1\leq i < j \leq k$, this concludes the proof.
\end{proof}

The recursive formula for the coefficients $\alpha_{B}(\pi,\pi')$ for $\pi,\pi'\in\Rank{B}$ and $B\in\SubsetsWE{\n}$ is then given by the following theorem.

\begin{theorem}[Recursive formula for the alpha coefficients]
\label{th:alpha-recursive-formula}
The coefficients $(\alpha_{B}(\pi,\pi'))_{\pi,\pi'\in\Rank{B},\, B\in\Subsets{\n}}$ are given by the following recursive formula:
\begin{itemize}
	\item $\alpha_{\emptyset}(\bar{0},\bar{0}) = 1$
	\item for all $B\in\Subsets{\n}$ and $\pi,\pi'\in\Rank{B}$,
	\[
	\alpha_{B}(\pi,\pi^{\prime}) = \mathbb{I}\{\pi = \pi^{\prime}\} - \frac{1}{\vert B\vert !} - \sum_{\substack{1\leq i < j \leq \vert B\vert \\ j-i < \vert B\vert -1}} \frac{1}{(\vert B\vert-j+i)!} \alpha_{c(\pi_{\set{i,j}})} \left(\pi_{\set{i,j}},\pi'_{\vert c(\pi_{\set{i,j}})}\right).
	\]
\end{itemize}
\end{theorem}

\begin{proof}
Eq. \eqref{eq:initialization} directly implies that $\alpha_{\emptyset}(\bar{0},\bar{0}) = 1$. Now,  Eq. \eqref{eq:induction-formula} gives for $B\in\Subsets{\n}$ and $\pi,\pi'\in\Rank{B}$
\[
\Psi_{B}\delta_{\pi'}(\pi) = \delta_{\pi'}(\pi)\quad - \sum_{B'\in\SubsetsWE{B}\setminus\{B\}}\phi_{B}\Psi_{B'}\delta_{\pi'}(\pi).
\]
Combined with Lemma \ref{lem:useful-lemma}, this leads to the desired result.
\end{proof}

\begin{example}
As an example, we provide the matrix $(\alpha_{B}(\pi,\pi'))_{\pi,\pi'\in\Rank{B}}$ for $B = \{1,2\}$ and $B = \{1,2,3\}$:
\[
\left[\alpha_{\{1,2\}}(\pi,\pi')\right]_{(\pi,\pi')} = 
\kbordermatrix{
	& 12 & 21\\
	12 & 1/2 & -1/2\\
	21 & -1/2 & 1/2
}
\]
and
\[
\left[\alpha_{\{1,2,3\}}(\pi,\pi')\right]_{(\pi,\pi')} = 
\kbordermatrix{
	& 123 & 132 & 213 & 231 & 312 & 321\\
	123 & 1/3 & -1/6 & -1/6 & -1/6 & -1/6 & 1/3\\
	132 & -1/6 & 1/3 & -1/6 & 1/3 & -1/6 & -1/6\\
	213 & -1/6 & -1/6 & 1/3 & -1/6 & 1/3 & -1/6\\
	231 & -1/6 & 1/3 & -1/6 & 1/3 & -1/6 & -1/6\\
	312 & -1/6 & -1/6 & 1/3 & -1/6 & 1/3 & -1/6\\
	321 & 1/3 & -1/6 & -1/6 & -1/6 & -1/6 & 1/3\\
}.
\]
\end{example}

\subsection{Interpretation of the wavelet synthesis operators $\phi_{A}$}
\label{subsec:embedding-operator}

Here we provide some more insights about the wavelet synthesis operators $\phi_{A}$. Their definition is indeed not intuitive as mentioned previously. For $A,B\in\Subsets{\n}$ with $B\subset A$, the most natural way to embed a Dirac function $\delta_{\pi}$ with $\pi\in\Rank{B}$ into $\Space{\Rank{A}}$ would rather be to map it to the uniform distribution over all the rankings on $A$ that extend $\pi$, that is to use the following operator
\begin{equation}
\label{eq:alternative-embedding-operator}
\phi'_{A}\quad:\quad\delta_{\pi}\quad\mapsto\quad\frac{\vert A\vert !}{\vert B\vert !}\sum_{\sigma\in\pi,\ \pi\subset\sigma}\delta_{\sigma}.
\end{equation}

\begin{example}
For $\pi = 42$ and $A = \set{4}$:
\begin{align*}
\phi_{A}\delta_{\pi} &= \frac{1}{6}\left[\delta_{1342} + \delta_{3142} + \delta_{1423} + \delta_{3421} + \delta_{4213} + \delta_{4231} \right]\\
\phi'_{A}\delta_{\pi} &= \frac{1}{12}\left[\delta_{1342} + \delta_{3142} + \delta_{1423} + \delta_{3421} + \delta_{4213} + \delta_{4231} + \delta_{1432} + \delta_{3412} + \delta_{4132} + \delta_{4312} + \delta_{4123} + \delta_{4321} \right]
\end{align*}
\end{example}

The mapping $\phi'_{A}$ is used implicitly in shuffling interpretations of rankings \citep[see][]{Diaconis1988, HG2012}. It corresponds to sending a ranking $\pi$ to the uniform distribution over all the possible shuffles between $\pi$ and a random ranking on $\Rank{A\setminus B}$. Notice also that for $A = \n$, $\phi'_{A}:\pi\mapsto(\vert A\vert !/n!)\1{\Sn(\pi)}$. In other words, $\phi'_{\n}$ maps an incomplete ranking to the uniform distribution on the set of its linear extensions. It is thus also involved implicitly in the approaches introduced in \citet{YLA02}, \citet{Kondor2010} and \citet{Sun2012} described in Subsection \ref{subsec:existing-approaches}. For these two reasons, $\phi'_{A}$ can be considered as the most intuitive embedding operator. It does not lead however to the localization properties of Theorem \ref{th:MRA-decomposition}. This is because it does not satisfy the key Lemma \ref{lem:commutativity}, whereas the embedding operator $\phi_{A}$ does. 

\begin{example}
Coming back to Example \ref{ex:lemma-illustration} with $A = \{1,2,3\}$, $B = \{1,2,4\}$ and $C = \{1,2,3,4\}$, we recall that, for $\pi = 123$ for instance,
\begin{align*}
M_{B}\phi_{C}\delta_{\pi} &= \frac{1}{2}M_{\{1,2,4\}}\left[ \delta_{4123} + \delta_{1234} \right] = \frac{1}{2}\left[ \delta_{412} + \delta_{124}\right],\qquad\qquad\qquad\qquad\qquad\\
\phi_{B}M_{A\cap B}\delta_{\pi} &= \phi_{\{1,2,4\}}\delta_{12} = \frac{1}{2}\left[ \delta_{412} + \delta_{124}\right].
\end{align*}
By contrast,
\begin{align*}
M_{B}\phi'_{C}\delta_{\pi} &= \frac{1}{4}M_{\{1,2,4\}}\left[ \delta_{4123} + \delta_{1423} + \delta_{1243} + \delta_{1234} \right] = \frac{1}{4}\left[ \delta_{412} + \delta_{142} + 2\delta_{124}\right]\\
\phi'_{B}M_{A\cap B}\delta_{\pi} &= \phi_{\{1,2,4\}}\delta_{12} = \frac{1}{3}\left[ \delta_{412} + \delta_{142} + \delta_{124}\right].
\end{align*}
Even if the operators $\phi'_{A}$ were normalized differently, they would still not satisfy Lemma \ref{lem:commutativity}. In the example, the difference comes from the fact that the element $\delta_{1243}$ leads to an additional term $\delta_{124}$ in the end.
\end{example}

We now develop a more intuitive interpretation of the localization properties induced by the operators $\phi_{A}$. Let's consider for instance $\pi = 12\in\Rank{\{1,2\}}$ and $C = \set{5}$, and let $\sigma\in\Rank{\set{5}}$ be a ranking that extends $\pi$. It induces rankings on all subsets $B\in\Subsets{\set{5}}$ with in particular $\sigma_{\vert \{1,2\}} = \pi$. Now consider a perturbation that changes $\sigma$ to $\sigma'$ such that $\sigma'_{\vert \{1,2\}} = 21$. It necessarily changes the relative positions of items $1$ and $2$ in $\sigma$ and more generally in all the subwords of $\sigma$ that contain $1$ and $2$. The question is then: how does it affect the other induced rankings $\sigma'_{\vert B}$ for $B\in\Subsets{\set{5}}$ such that $\{1,2\}\not\subset B$? If $B\cap\{1,2\} = \emptyset$, $\sigma'_{\vert B}$ is different from $\sigma_{\vert B}$ if and only if the perturbation also modifies the relative order of some items in $B$. This is independent from the action on $1$ and $2$. Now, for $B\in\Subsets{\set{5}}$ such that $\vert B\cap \{1,2\}\vert = 1$, the key observation is that it depends on the items that are placed \textit{between} $1$ and $2$ in $\sigma$. For instance if $\sigma = 41523$, any perturbation that changes the relative positions of $1$ and $2$ will necessarily impact the relative position of at least $1$ and $5$ or $2$ and $5$. By contrast, if $\sigma = 45123$ for instance, swapping items $1$ and $2$ will not have any impact on $\sigma_{\vert B}$ for all $B$ such that $\vert B\cap\{1,2\}\vert = 1$. Therefore among the rankings that extend $12$, only the ones in which $1$ and $2$ are adjacent can be perturbed such that only the ranking induced on $\{1,2\}$ is affected and not the ones on the subsets $B$ with $\vert B\cap\{1,2\}\vert\leq 1$. A similar interpretation holds for subsets of items of any size. Developing a general theory of perturbations for rankings would certainly be an interesting future research direction.

\section{Connection with $\Sn$-based harmonic analysis and other mathematical constructions}
\label{sec:connections}

Though its construction only relies on results from combinatorics and algebraic topology, it happens that the MRA representation is connected with $\Sn$-based harmonic analysis and other mathematical constructions. In this section we recall some background about $\Sn$-based harmonic analysis and detail these connections. The main results are:
\begin{itemize}
	\item Theorem \ref{th:decomposition-H-k}, which draws the connection between the MRA representation and $\Sn$-based harmonic analysis;
	\item Theorem \ref{th:connection-other-MRA}, which establishes a decomposition for the alternative embedding operator $\phi'_{\n}$ (considered in Subsection \ref{subsec:interpretation}) with subspaces isomorphic to subspaces constructed with the MRA decomposition;
	\item Theorem \ref{th:eigenspaces}, which provides additional insights on the interpretation of scales.
\end{itemize} 

\subsection{Background on $\Sn$-based harmonic analysis}
\label{subsec:background-harmonic-analysis}

Harmonic analysis on a finite set $\mathcal{X}$ consists in analyzing functions $f\in\Space{\mathcal{X}}$ by representing them as sums of projections onto subspaces that are invariant under the action of translations of a canonic group $G$ \citep[see][]{Diaconis89}. Let us introduce some definitions to be more specific \citep[we refer the reader to][for background on group theory]{FH91}. A transitive action $(g,x)\mapsto g\cdot x$ of $G$ on $\mathcal{X}$ naturally defines a family of translation operators $T_{g}$ on $\Space{\mathcal{X}}$ by $T_{g}\delta_{x} = \delta_{g\cdot x}$ or equivalently by $T_{g}f(x) = f(g^{-1}\cdot x)$ for any $f\in\Space{\mathcal{X}}$ and $x\in\mathcal{X}$. The mapping $g\mapsto T_{g}$ is a representation of $G$ on $\Space{\mathcal{X}}$ and a classic result from group representation theory says that the latter is isomorphic to the direct sum $\bigoplus_{\rho}m_{\rho}V_{\rho}$, where each $\rho$ is an irreducible representation of $G$, $V_{\rho}$ its associated linear space and $m_{\rho}$ a nonnegative integer. In other words, given an isomorphism $\Phi:\bigoplus_{\rho}m_{\rho}V_{\rho}\rightarrow\Space{\mathcal{X}}$, any function $f\in\Space{\mathcal{X}}$ admits a decomposition
\[
f = \Phi \sum_{\rho}m_{\rho}\mathcal{F}_{\rho}f,
\]
where $\mathcal{F}_{\rho}f$ is a projection of $f$ onto $V_{\rho}$ for each $\rho$, which thus localizes a certain part of information about $f$ that is invariant under translations. \citep[The projection $\mathcal{F}_{\rho}f$ is usually constructed via the Fourier transform of the spherical function associated to $f$, see][for more details]{ST09}. Harmonic analysis then consists in analyzing the function $f$ through its representation $(\mathcal{F}_{\rho}f)_{\rho}$. In the particular case where $\mathcal{X} = G$, a classic result says that for each $\rho$, the multiplicity $m_{\rho}$ of $V_{\lambda}$ in the decomposition of $\Space{G}$ is equal to its dimension: $m_{\rho} = \dim V_{\rho}$. 

In a discrete setting, the symmetric group usually appears as the canonic group that operates on $\mathcal{X}$. For instance, $\Sn$ naturally operates on $\n$ via $\sigma\cdot i = \sigma(i)$, on $\{A\subset\n\}$ via $\sigma\cdot A = \sigma(A) := \{\sigma(a) \;\vert\; a\in A\}$ or even on $\Sn$ via $\sigma\cdot\tau = \sigma\tau$ or via $\sigma\cdot\tau = \tau\sigma^{-1}$. Representations of the symmetric group have been thoroughly studied in the literature \citep[see for instance][]{JK81, CSST10, Sagan2013}. Each irreducible representation of $\Sn$ is indexed by a partition of $n$, namely a tuple $\lambda = (\lambda_{1},\dots,\lambda_{r})$ of positive integers such that $\lambda_{1}\geq\dots\geq\lambda_{r}$ and $\sum_{i=1}^{r}\lambda_{i} = n$. The fact that $\lambda$ is a partition of $n$ is denoted by $\lambda\vdash n$. The spaces of the irreducible representations are called the Specht modules. They are denoted by $S^{\lambda}$ and their dimensions by $d_{\lambda}$ for $\lambda\vdash n$. One thus has in particular the isomorphism of representations (here and throughout, we use the sign $\cong$ to denote that two spaces are isomorphic as $\Sn$-representations).
\begin{equation}
\label{eq:Fourier-decomposition}
\Space{\Sn}\cong\bigoplus_{\lambda\vdash n}d_{\lambda}S^{\lambda}.
\end{equation}
In the decomposition of Eq. \eqref{eq:Fourier-decomposition}, each irreducible representation $S^{\lambda}$ appears with multiplicity $d_{\lambda}$ for $\lambda\vdash n$. The copies of each irreducible representation admits a finer canonical differentiation, based on \textit{standard Young tableaux}. A Young diagram (or a Ferrer's diagram) of size $n$ is a collection of boxes of the form
\begin{center}
\begin{tikzpicture}[scale = 0.6]
\node at (-2,3.5) {$\lambda_{1}$};
\node at (-2,2.5) {$\lambda_{2}$};
\node at (-2,1.5) {$\vdots$};
\node at (-2,0.5) {$\lambda_{r}$};
\draw (0,4) -- (5,4);
\draw (0,3) -- (5,3);
\draw (0,2) -- (4,2);
\draw[dotted] (1,1) -- (3,1);
\draw (0,1) -- (1,1);
\draw (0,0) -- (1,0);
\draw (0,4) -- (0,2);
\draw (1,4) -- (1,2);
\draw (2,4) -- (2,2);
\draw (3,4) -- (3,2);
\draw (4,4) -- (4,2);
\draw (5,4) -- (5,3);
\draw[dotted] (0,2) -- (0,1);
\draw[dotted] (1,2) -- (1,1);
\draw[dotted] (2,2) -- (2,1);
\draw[dotted] (3,2) -- (3,1);
\draw (0,1) -- (0,0);
\draw (1,1) -- (1,0);
\end{tikzpicture}
\end{center}
where if $\lambda_{i}$ denotes the number of boxes in row $i$, then $\lambda = (\lambda_{1}, \dots, \lambda_{r})$, called the shape of the Young diagram, must be a partition of $n$. The total number of boxes of a Young diagram is therefore equal to $n$, and each row contains at most as many boxes as the row above it. A Young tableau is a Young diagram filled with all the integers $1, \dots, n$, one in each boxes. The shape of a Young tableau $Q$, denoted by $\shape (Q)$, is the shape of the associated Young Diagram, it is thus a partition of $n$. There are clearly $n!$ Young tableaux of a given shape $\lambda \vdash n$. A Young tableau is said to be \textit{standard} if the numbers increase along the rows and down the columns.

\begin{example}
In the following figure, the first tableau is standard whereas the second is not.
\begin{center}
\begin{tikzpicture}[scale=0.6]
\node at (0.5,2.5) {1};
\node at (1.5,2.5) {2};
\node at (2.5,2.5) {3};
\node at (0.5,1.5) {4};
\node at (1.5,1.5) {5};
\node at (0.5,0.5) {6};
\draw (0,3) -- (3,3);
\draw (0,2) -- (3,2);
\draw (0,1) -- (2,1);
\draw (0,0) -- (1,0);
\draw (0,3) -- (0,0);
\draw (1,3) -- (1,0);
\draw (2,3) -- (2,1);
\draw (3,3) -- (3,2);
\node at (6.5,2.5) {1};
\node at (7.5,2.5) {3};
\node at (8.5,2.5) {5};
\node at (6.5,1.5) {4};
\node at (7.5,1.5) {2};
\node at (6.5,0.5) {6};
\draw (6,3) -- (9,3);
\draw (6,2) -- (9,2);
\draw (6,1) -- (8,1);
\draw (6,0) -- (7,0);
\draw (6,3) -- (6,0);
\draw (7,3) -- (7,0);
\draw (8,3) -- (8,1);
\draw (9,3) -- (9,2);
\end{tikzpicture}
\end{center}
\end{example}

Notice that a standard Young tableau always have $1$ in its top-left box, and that the box that contains $n$ is necessarily at the end of a row and a column. We denote by $\SYT$ the set of all standard Young tableaux of size $n$ and by $\SYT (\lambda) = \{Q\in \SYT \;\vert\; \shape (Q) = \lambda\}$ the set of standard Young tableaux of shape $\lambda$, for $\lambda\vdash n$. By construction, $\SYT = \bigsqcup_{\lambda\,\vdash\, n}\SYT (\lambda)$.  Now, a classic result in the representation theory of the symmetric group states that  $d_{\lambda} = \vert\SYT (\lambda)\vert$ for each $\lambda\vdash n$. The decomposition of Equation \eqref{eq:Fourier-decomposition} is then refined into:
\begin{equation}
\label{eq:SYT-decomposition}
L(\Sn) \cong \bigoplus_{\lambda\vdash n}\bigoplus_{Q\in\SYT (\lambda)}S^{\,\shape (Q)} \cong \bigoplus_{Q\in\SYT}S^{\,\shape (Q)}.
\end{equation}
Figure \ref{fig:tableaux-decomposition} represents all the standard Young tableaux of size $n = 4$, gathered by shape.

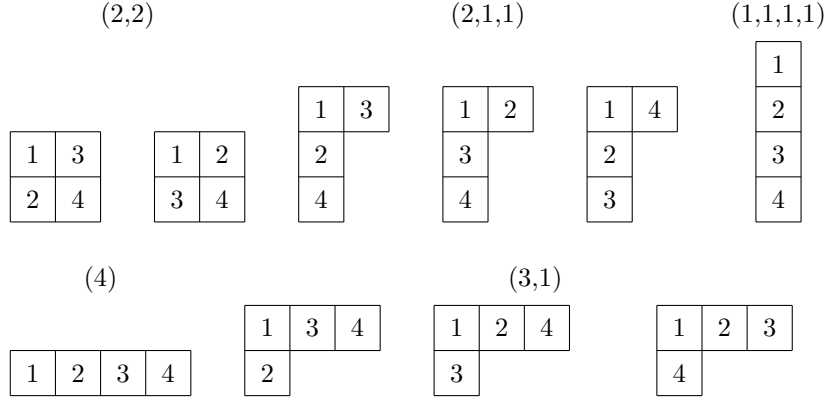
\begin{figure}
\[
\arraycolsep=10pt\def\arraystretch{2}
\begin{array}{ccccccccccccccccccccccccccccccccccccccccccccccccccccccccccccccccccccccccccccccccccccccccccccccccccccccccccccccccccccccccccccccccccccccccccccccc}
\multicolumn{52}{c}{$(2,2)$} & \multicolumn{78}{c}{$(2,1,1)$} & \multicolumn{13}{c}{$(1,1,1,1)$}\\
\multicolumn{26}{c}{
	\begin{tikzpicture}[scale=0.6]
	\draw (0,0) -- (2,0);
	\draw (0,1) -- (2,1);
	\draw (0,2) -- (2,2);
	\draw (0,0) -- (0,2);
	\draw (1,0) -- (1,2);
	\draw (2,0) -- (2,2);
	\node at (0.5,0.5) {2};
	\node at (0.5,1.5) {1};
	\node at (1.5,1.5) {3};
	\node at (1.5,0.5) {4};
	\end{tikzpicture}
}
& \multicolumn{26}{c}{
	\begin{tikzpicture}[scale=0.6]
	\draw (0,0) -- (2,0);
	\draw (0,1) -- (2,1);
	\draw (0,2) -- (2,2);
	\draw (0,0) -- (0,2);
	\draw (1,0) -- (1,2);
	\draw (2,0) -- (2,2);
	\node at (0.5,0.5) {3};
	\node at (0.5,1.5) {1};
	\node at (1.5,1.5) {2};
	\node at (1.5,0.5) {4};
	\end{tikzpicture}
}
& \multicolumn{26}{c}{
	\begin{tikzpicture}[scale=0.6]
	\draw (0,0) -- (1,0);
	\draw (0,1) -- (1,1);
	\draw (0,2) -- (2,2);
	\draw (0,3) -- (2,3);
	\draw (0,0) -- (0,3);
	\draw (1,0) -- (1,3);
	\draw (2,2) -- (2,3);
	\node at (0.5,0.5) {4};
	\node at (0.5,1.5) {2};
	\node at (0.5,2.5) {1};
	\node at (1.5,2.5) {3};
	\end{tikzpicture}
}
& \multicolumn{26}{c}{
	\begin{tikzpicture}[scale=0.6]
	\draw (0,0) -- (1,0);
	\draw (0,1) -- (1,1);
	\draw (0,2) -- (2,2);
	\draw (0,3) -- (2,3);
	\draw (0,0) -- (0,3);
	\draw (1,0) -- (1,3);
	\draw (2,2) -- (2,3);
	\node at (0.5,0.5) {4};
	\node at (0.5,1.5) {3};
	\node at (0.5,2.5) {1};
	\node at (1.5,2.5) {2};
	\end{tikzpicture}
}
& \multicolumn{26}{c}{
	\begin{tikzpicture}[scale=0.6]
	\draw (0,0) -- (1,0);
	\draw (0,1) -- (1,1);
	\draw (0,2) -- (2,2);
	\draw (0,3) -- (2,3);
	\draw (0,0) -- (0,3);
	\draw (1,0) -- (1,3);
	\draw (2,2) -- (2,3);
	\node at (0.5,0.5) {3};
	\node at (0.5,1.5) {2};
	\node at (0.5,2.5) {1};
	\node at (1.5,2.5) {4};
	\end{tikzpicture}
}
& \multicolumn{13}{c}{
	\begin{tikzpicture}[scale=0.6]
	\draw (0,0) -- (1,0);
	\draw (0,1) -- (1,1);
	\draw (0,2) -- (1,2);
	\draw (0,3) -- (1,3);
	\draw (0,4) -- (1,4);
	\draw (0,0) -- (0,4);
	\draw (1,0) -- (1,4);
	\node at (0.5,0.5) {4};
	\node at (0.5,1.5) {3};
	\node at (0.5,2.5) {2};
	\node at (0.5,3.5) {1};
	\end{tikzpicture}
}\\
\multicolumn{44}{c}{$(4)$} & \multicolumn{99}{c}{$(3,1)$}\\
\multicolumn{44}{c}{
 	\begin{tikzpicture}[scale=0.6]
 	\draw (0,0) -- (4,0);
 	\draw (0,1) -- (4,1);
 	\draw (0,0) -- (0,1);
 	\draw (1,0) -- (1,1);
 	\draw (2,0) -- (2,1);
 	\draw (3,0) -- (3,1);
 	\draw (4,0) -- (4,1);
 	\node at (0.5,0.5) {1};
 	\node at (1.5,0.5) {2};
 	\node at (2.5,0.5) {3};
 	\node at (3.5,0.5) {4};
 	\end{tikzpicture}
}
& \multicolumn{33}{c}{
	\begin{tikzpicture}[scale=0.6]
	\draw (0,0) -- (1,0);
	\draw (0,1) -- (3,1);
	\draw (0,2) -- (3,2);
	\draw (0,0) -- (0,2);
	\draw (1,0) -- (1,2);
	\draw (2,1) -- (2,2);
	\draw (3,1) -- (3,2);
	\node at (0.5,0.5) {2};
	\node at (0.5,1.5) {1};
	\node at (1.5,1.5) {3};
	\node at (2.5,1.5) {4};
	\end{tikzpicture}
}
& \multicolumn{33}{c}{
	\begin{tikzpicture}[scale=0.6]
	\draw (0,0) -- (1,0);
	\draw (0,1) -- (3,1);
	\draw (0,2) -- (3,2);
	\draw (0,0) -- (0,2);
	\draw (1,0) -- (1,2);
	\draw (2,1) -- (2,2);
	\draw (3,1) -- (3,2);
	\node at (0.5,0.5) {3};
	\node at (0.5,1.5) {1};
	\node at (1.5,1.5) {2};
	\node at (2.5,1.5) {4};
	\end{tikzpicture}
}
& \multicolumn{33}{c}{
	\begin{tikzpicture}[scale=0.6]
	\draw (0,0) -- (1,0);
	\draw (0,1) -- (3,1);
	\draw (0,2) -- (3,2);
	\draw (0,0) -- (0,2);
	\draw (1,0) -- (1,2);
	\draw (2,1) -- (2,2);
	\draw (3,1) -- (3,2);
	\node at (0.5,0.5) {4};
	\node at (0.5,1.5) {1};
	\node at (1.5,1.5) {2};
	\node at (2.5,1.5) {3};
	\end{tikzpicture}
}
\end{array}
\]
\caption{Standard Young tableaux of size $n = 4$}
\label{fig:tableaux-decomposition}
\end{figure}

\subsection{$\Sn$-based harmonic analysis localizes absolute rank information}
\label{subsec:absolute-rank-information}

By construction, for any finite set $\mathcal{X}$ on which $\Sn$ acts transitively, the projection of a function $f\in\Space{\mathcal{X}}$ on a Specht module $S^{\lambda}$ localizes a certain part of information about $f$ that is invariant under $\Sn$-translations. It happens that this part of information has a concrete interpretation from a ranking point of view: it is the part of information specific to the $\lambda$-marginals \citep[see for instance][]{Diaconis1988, HGG09}. Let us introduce some more definitions to be more specific. For $\lambda = (\lambda_{1},\dots,\lambda_{r})\vdash n$, we define the set 
\[
\Part_{\lambda}(\n) = \{\mathcal{B} = (B_{1},\dots, B_{r}) \;\vert\; B_{1}\sqcup\dots\sqcup B_{r} = \n \text{ and } \vert B_{i}\vert = \lambda_{i} \text{ for each } i = 1,\dots,r \}.
\]
The space $\Space{\Part_{\lambda}(\n)}$ is usually denoted by $M^{\lambda}$ in the literature and called a Young module. For any function $f\in\Space{\Sn}$, we define its $\lambda$-marginal on the partition $\mathcal{B}\in\Part_{\lambda}(\n)$ as the function $f^{\lambda}_{\mathcal{B}}\in M^{\lambda}$ given by
\[
f^{\lambda}_{\mathcal{B}}(\mathcal{B}') = \sum_{\substack{\sigma\in\Sn\\ \sigma(B_{1}) = B'_{1}, \dots, \sigma(B_{r}) = B'_{r}}}f(\sigma) \qquad\text{for all }\mathcal{B}'\in\Part_{\lambda}(\n).
\]
If $p$ is a ranking model over $\Sn$ and $\Sigma$ a random permutation drawn from $p$, the marginal $p^{\lambda}_{\mathcal{B}}$ is simply the law of the random variable $(\Sigma(B_{1}),\dots,\Sigma(B_{r}))$:
\[
p^{\lambda}_{\mathcal{B}}(\mathcal{B}') = \mathbb{P}\left[\Sigma(B_{1}) = B_{1}',\dots,\Sigma(B_{r}) = B_{r}'\right].
\]
The collection of $\lambda$-marginals $(p^{\lambda}_{\mathcal{B}})_{\mathcal{B}\in\Part_{\lambda}(\n)}$ then has a simple interpretation. Let us consider first the simple case $\lambda = (n-1,1)$. Elements of $\Part_{(n-1,1)}(\n)$ are necessarily of the form $(\n\setminus\{i\}, \{i\})$, with $i\in\n$. Then for $(i,j)\in\n^{2}$, we have the simplification
\[
\mathbb{P}\left[\Sigma(\n\setminus\{i\}) = \n\setminus\{j\},\ \Sigma(\{i\}) = \{j\}\right] = \mathbb{P}\left[\Sigma(i) = j\right].
\]
The marginal of $p$ associated to $(\n\setminus\{i\}, \{i\})$ is thus the probability distribution $(\mathbb{P}[\Sigma(i) = j])_{j\in\n}$ on $\n$. From a ranking point of view, this is the law of the rank of item $i$. The $n\times n$ matrix $T_{(n-1,1)}(p)$ that gathers all the $(n-1,1)$-marginals of $p$ is then equal to
\[
T_{(n-1,1)}(p) = \left(
\begin{array}{ccc}
\mathbb{P}[\Sigma(1) = 1] & \cdots & \mathbb{P}[\Sigma(n) = 1]\\
\vdots & \ddots & \vdots \\
\mathbb{P}[\Sigma(1) = n] & \cdots & \mathbb{P}[\Sigma(n) = n]\\
\end{array}\right).
\]
It thus contains the laws of all the random variables $\Sigma(i)$ or $\Sigma^{-1}(j)$ for $i,j\in\n$. All these distributions capture information about an ``absolute rank'', in the sense that it is the rank of an item inside a ranking implying all the $n$ items. Such information is considered to be of order $1$, because it concerns only one item. 

There are two types of marginals of order $2$: the $(n-2,2)$-marginals and the $(n-2,1,1)$-marginals. They correspond respectively to the probability distributions 
\[
\Big(\mathbb{P}[\Sigma(\{i,i'\})=\{j,j'\}]\Big)_{\substack{1\leq i < i'\leq n \\ 1\leq j < j'\leq n}} \qquad\text{and}\qquad \Big(\mathbb{P}[\Sigma(i) = j, \Sigma(i') = j']\Big)_{\substack{1\leq i \neq i'\leq n \\ 1\leq j \neq j'\leq n}}.
\]
In both cases, the matrices that gathers all the marginals $T_{(n-2,2)}(p)$ and $T_{(n-2,1,1)}(p)$ capture information about the absolute ranks of two items, either as a pair or as a couple. More generally for $k\in\{1,\dots,n-1\}$ and $\lambda\vdash n$ such that $\lambda_{1} = k$, the $\lambda$-marginals capture information about the absolute ranks of $k$ items of type $\lambda$.

Now, the absolute rank information localized by $\lambda$-marginals can be decomposed into components that are invariant under translations. Indeed, for any $\lambda\vdash n$, the mapping $(\sigma,\mathcal{B})\mapsto \sigma(\mathcal{B})$ is actually a transitive action of $\Sn$ on $\Part_{\lambda}(\n)$ so that $M^{\lambda}$ is a representation of $\Sn$ and is isomorphic to a decomposition involving the $S^{\mu}$'s for $\mu\vdash n$. This decomposition is given by Young's rule \citep[see for instance][]{Diaconis1988}:
\begin{equation}
\label{eq:Young-rule}
M^{\lambda} \cong S^{\lambda}\oplus\bigoplus_{\mu \rhd \lambda}K_{\mu,\lambda}S^{\mu},
\end{equation}
where $\rhd$ is the strict partial order associated to the dominance order on partitions of $n$, defined for $\lambda = (\lambda_{1}, \dots, \lambda_{r})$ and $\mu = (\mu_{1}, \dots, \mu_{s})$ by $\lambda \unrhd \mu$ if for all $j\in \{1, \dots, r\}$, $\sum_{i=1}^{j}\lambda_{i} \geq \sum_{i=1}^{j}\mu_{i}$, and  the $K_{\mu,\lambda}$'s are positive integers called the Kotska's numbers for $\mu\rhd\lambda$. Applying Young's rule \eqref{eq:Young-rule} recursively leads to
\begin{equation}
\label{eq:recursive-Young-rule}
\begin{aligned}
M^{(n)} &\cong S^{(n)}\\
M^{(n-1)} &\cong S^{(n-1,1)}\oplus M^{(n)}\\
M^{(n-2,2)} &\cong S^{(n-2,2)}\oplus M^{(n-1,1)}\\
M^{(n-2,1,1)} &\cong S^{(n-2,1,1)}\oplus M^{(n-2,2)}\oplus S^{(n-1,1)}.
\end{aligned}
\end{equation}
Equation \eqref{eq:recursive-Young-rule} means first that $S^{(n)}$ contains the part of information of level $0$. Then $S^{(n-1,1)}$ contains the additional part of information to get from $M^{(n)}$ to $M^{(n-1,1)}$, or in other words the part of information specific to level $1$. Then $S^{(n-2,2)}$ contains the additional part of information of $M^{(n-2,2)}$ to get from $M^{(n-1,1)}$ to $M^{(n-2,2)}$, or in other words the part of information specific to $(n-2,2)$-marginals. And finally $S^{(n-2,1,1)}$ contains the additional part of information to get from $M^{(n-2,2)}$ to $M^{(n-2,1,1)}$, or in other words the part of information specific to $(n-2,1,1)$-marginals, because the information of $S^{(n-1,1)}$ is already contained in $M^{(n-2,2)}$. More generally for any given $\lambda \vdash n$, the Specht module $S^{\lambda}$ localizes the information of $M^{\lambda}$ that is not contained in the $M^{\mu}$'s for $\mu \rhd \lambda$. In this sense, $S^{\lambda}$ localizes the part of absolute rank information that is specific to $\lambda$-marginals.

\subsection{The MRA representation and $\Sn$-based harmonic analysis provide ``orthogonal'' decompositions of rank information}
\label{subsec:two-decompositions}

If the $S^{\lambda}$'s localize parts of absolute rank information, we recall by contrast that for $B\in\SubsetsWE{\n}$, the space $H_{B}$ localizes the part of information specific to the marginal on $B$. It thus localizes ``relative'' rank information as it concerns the ranks of the items of $B$ inside rankings that involve only the items of $B$. To stress on the difference, we assert that such information is by nature not invariant under translation. To be more specific, we consider the natural action of $\Sn$ on $\Gn$ defined for $\sigma\in\Sn$ and $\pi = \pi_{1}\dots\pi_{k}\in\Gn$ by $\sigma\cdot\pi = \sigma(\pi_{1})\dots\sigma(\pi_{k})$ (by convention $\sigma(\bar{0}) = \bar{0}$). Denoting by $T_{\sigma}$ the associated translation operators on $\Space{\Gn}$, one has the following proposition.

\begin{proposition}[Action of translations on spaces $H_{B}$]
\label{prop:action-on-H}
For all $\sigma\in\Sn$ and $B\in\SubsetsWE{\n}$,
\[
T_{\sigma} (H_{B}) = H_{\sigma(B)}.
\]
\end{proposition}

\begin{proof}
Since $\vert\sigma(B)\vert = \vert B\vert$, $\dim H_{\sigma(B)} = \dim H_{B}$. It is thus sufficient to prove that $T_{\sigma} (H_{B}) \subset H_{\sigma(B)}$. For $F\in\Space{\Rank{B}}$, it is clear that $T_{\sigma}F = \sum_{\pi\in\Rank{B}}F(\pi)\delta_{\sigma\cdot\pi}\in\Space{\Rank{\sigma(B)}}$. We just need to show that $M_{C}T_{\sigma}F = 0$ for any $C\in\SubsetsWE{\sigma(B)}\setminus\{\sigma(B)\}$ or equivalently $M_{\sigma(B')}T_{\sigma}F = 0$ for any $B'\in\SubsetsWE{B}\setminus\{B\}$. This is proven by noticing that for any $\pi\in\Rank{B}$, $(\sigma\cdot\pi)_{\vert \sigma(B')} = \sigma\cdot(\pi_{\vert B'})$.
\end{proof}

Proposition \ref{prop:action-on-H} implies that for all $\sigma$ and $B$ such that $\sigma(B)\neq B$, one has $\sigma\cdot H_{B}\neq H_{B}$. The space $H_{B}$ is thus not invariant invariant under all $\Sn$-based translations. We now show however that there is a mathematical connection between the MRA and the harmonic analysis decompositions. For $k\in\{0,\dots,n\}\setminus\{1\}$, we define
\begin{equation}
H^{k} = \bigoplus_{B\subset\n,\ \vert B\vert = k}H_{B}, \qquad\text{so that}\qquad \Hn = \bigoplus_{\substack{k=0 \\ k\neq 1}}^{n}H^{k}.
\end{equation}
Space $H^{k}$ localizes all relative rank information of scale $k$. In addition, as $\vert \sigma(B)\vert = \vert B\vert$ for any $B\in\SubsetsWE{\n}$, Proposition \ref{prop:action-on-H} implies that $\sigma\cdot H^{k}$ for all $\sigma\in\Sn$ or in other words that $H^{k}$ is invariant under $\Sn$-based translations. It is thus also the case of the feature space $\Hn$ and both can be decomposed as a sum of irreducible representations $S^{\lambda}$:
\begin{equation}
\label{eq:H-k-decomposition}
H^{k} \cong \bigoplus_{\lambda \vdash n}\kappa_{\lambda}^{k}S^{\lambda} \qquad\text{and}\qquad \Hn \cong \bigoplus_{\substack{k=0 \\ k\neq 1}}^{n}\bigoplus_{\lambda \vdash n}\kappa_{\lambda}^{k}S^{\lambda},
\end{equation}
where the $\kappa_{\lambda}^{k}$'s are nonnegative integers. Eq. \eqref{eq:H-k-decomposition} means that the space $H^{k}$ also localize some absolute rank information, quantified through the multiplicities $\kappa_{\lambda}^{k}$ of the $S^{\lambda}$'s. The connection with the harmonic decomposition of $\Space{\Sn}$ is provided in the following proposition. Its proof is mainly formal and left in Appendix.

\begin{proposition}[Representation isomorphism]
\label{prop:connection-Fourier-MRA}
The spaces $\Space{\Sn}$ and $\Hn$ are isomorphic as representations of $\Sn$: $\Space{\Sn} \cong \Hn$. In particular one has 
\[
\sum_{\substack{k = 0 \\ k\neq 1}}^{n}\kappa_{\lambda}^{k} = d_{\lambda} \qquad\text{for all}\qquad \lambda\vdash n.
\]
\end{proposition}

The multiplicity $\kappa_{\lambda}^{k}$ of each irreducible in $H^{k}$ can actually be calculated through a combinatorial formula. This is one of the major results established in \citet{Reiner13}. Its statement requires an additional definition. Notice that any standard Young tableau $Q$ contains a unique maximal subtableau of the form
\begin{center}
\begin{tikzpicture}[scale=0.6, every node/.style={scale=0.6}]
\draw (0,0) -- (0,1);
\draw (0,0) -- (1,0);
\draw (1,0) -- (1,1);
\draw (0,1) -- (1,1);
\node at (0.5,0.5) {$l+m$};
\draw (0,2) -- (0,4);
\draw (1,2) -- (1,4);
\draw (2,3) -- (2,4);
\draw (0,2) -- (1,2);
\draw (0,3) -- (2,3);
\draw (0,4) -- (2,4);
\node at (0.5,2.5) {$l+1$};
\node at (0.5,3.5) {$1$};
\node at (1.5,3.5) {$2$};
\draw (3,3) -- (3,4);
\draw (3,3) -- (4,3);
\draw (3,4) -- (4,4);
\draw (4,4) -- (4,3);
\node at (3.5,3.5) {$l$};
\draw[dashed] (0,1) -- (0,2);
\draw[dashed] (1,1) -- (1,2);
\draw[dashed] (2,3) -- (3,3);
\draw[dashed] (2,4) -- (3,4);
\end{tikzpicture}
\end{center}
with $1 \leq l \leq n$ and $0 \leq m \leq n-l$. Then the authors of \citet{Reiner13} define (in the proof of Proposition 6.23) the following quantity:
\begin{equation}
\label{eq:eig-definition}
\eig (Q) = \left\{
\begin{aligned}
l &\quad\text{if } m \text{ is even},\\
l-1 &\quad\text{if } m \text{ is odd}.\\
\end{aligned}\right.
\end{equation}

\begin{theorem}[Fourier decomposition of the spaces $H^{k}$]
\label{th:decomposition-H-k}
For $k\in\{0,\dots,n\}\setminus\{1\}$ and $\lambda\vdash n$, the multiplicity of $S^{\lambda}$ in $H^{k}$ is given by $\kappa_{\lambda}^{k} = \vert\{Q\in\SYT \;\vert\; \eig (Q) = n-k\}\vert$. In other words, the following decomposition holds
\[
H^{k} \cong \bigoplus_{\substack{Q\in\SYT \\ \eig (Q) = n-k}}S^{\,\shape (Q)}.
\]
\end{theorem}

In the notations of \citet{Reiner13}, $H_{B} = \ker \pi_{B}$, so that Theorem \ref{th:decomposition-H-k} is a reformulation of their theorem 6.26. It provides a new decomposition of rank information. For $\lambda\vdash n$ we denote by $U^{\lambda}$ the component $d_{\lambda}S^{\lambda}$ in the decomposition \eqref{eq:Fourier-decomposition} of $\Space{\Sn}$ (it is usually called an isotypic component). Then gathering Equation \eqref{eq:Fourier-decomposition} with Theorem \ref{th:decomposition-H-k} gives
\begin{equation}
\label{eq:all-decompositions}
\Space{\Sn} \cong \bigoplus_{Q\in\SYT}S^{\shape (Q)} \cong \bigoplus_{\lambda\vdash n}U^{\lambda} \cong \bigoplus_{\substack{k=0 \\ k\neq 1}}^{n}H^{k}.
\end{equation}
The first decomposition in Equation \eqref{eq:all-decompositions} is the full decomposition of $\Space{\Sn}$ into irreducible representations, each localizing an ``elementary'' part of absolute rank information. The second decomposition, into components $U^{\lambda}$, corresponds to the harmonic analysis decomposition where for each $\lambda\vdash n$, $U^{\lambda}$ localizes the part of absolute rank information specific to $\lambda$-marginals. The last decomposition, into spaces $H^{k}$, corresponds to the MRA decomposition where for each $k\in\{0,\dots,n\}\setminus\{1\}$, $H^{k}$ localizes the part of absolute information specific to scale $k$. These different decompositions are illustrated for $n = 4$ in Figure \ref{fig:decompositions}.

\begin{figure}
\newcommand{\ucong}{\mathbin{\rotatebox[origin=c]{-90}{$\cong$}}}
\[
\def\arraystretch{1.5}
\begin{array}{ccccccccccc}

\Space{\Sym{4}} & \cong & U^{(4)} & \oplus & U^{(3,1)} & \oplus & U^{(2,2)} & \oplus & U^{(2,1,1)} & \oplus & U^{(1,1,1,1)}\\
\ucong				&		& \ucong  & 	   & \ucong    &        & \ucong    &        & \ucong      &        & \ucong       \\

H^{4} & \cong & & & S^{
	\begin{tikzpicture}[scale=0.2, every node/.style={scale=0.4}]
	\draw (0,0) -- (1,0);
	\draw (0,1) -- (3,1);
	\draw (0,2) -- (3,2);
	\draw (0,0) -- (0,2);
	\draw (1,0) -- (1,2);
	\draw (2,1) -- (2,2);
	\draw (3,1) -- (3,2);
	\node at (0.5,0.5) {2};
	\node at (0.5,1.5) {1};
	\node at (1.5,1.5) {3};
	\node at (2.5,1.5) {4};
	\end{tikzpicture}
} & \oplus & 
S^{
	\begin{tikzpicture}[scale=0.2, every node/.style={scale=0.4}]
	\draw (0,0) -- (2,0);
	\draw (0,1) -- (2,1);
	\draw (0,2) -- (2,2);
	\draw (0,0) -- (0,2);
	\draw (1,0) -- (1,2);
	\draw (2,0) -- (2,2);
	\node at (0.5,0.5) {2};
	\node at (0.5,1.5) {1};
	\node at (1.5,1.5) {3};
	\node at (1.5,0.5) {4};
	\end{tikzpicture}
} & \oplus & 
S^{
	\begin{tikzpicture}[scale=0.2, every node/.style={scale=0.4}]
	\draw (0,0) -- (1,0);
	\draw (0,1) -- (1,1);
	\draw (0,2) -- (2,2);
	\draw (0,3) -- (2,3);
	\draw (0,0) -- (0,3);
	\draw (1,0) -- (1,3);
	\draw (2,2) -- (2,3);
	\node at (0.5,0.5) {4};
	\node at (0.5,1.5) {2};
	\node at (0.5,2.5) {1};
	\node at (1.5,2.5) {3};
	\end{tikzpicture}
} & \oplus & 
S^{
	\begin{tikzpicture}[scale=0.2, every node/.style={scale=0.4}]
	\draw (0,0) -- (1,0);
	\draw (0,1) -- (1,1);
	\draw (0,2) -- (1,2);
	\draw (0,3) -- (1,3);
	\draw (0,4) -- (1,4);
	\draw (0,0) -- (0,4);
	\draw (1,0) -- (1,4);
	\node at (0.5,0.5) {4};
	\node at (0.5,1.5) {3};
	\node at (0.5,2.5) {2};
	\node at (0.5,3.5) {1};
	\end{tikzpicture}
}\\
\oplus	            &       &         &        & \oplus    &        & \oplus    &        & \oplus      &               \\
H^{3} & \cong & & & S^{
 	\begin{tikzpicture}[scale=0.2, every node/.style={scale=0.4}]
 	\draw (0,0) -- (1,0);
 	\draw (0,1) -- (3,1);
 	\draw (0,2) -- (3,2);
 	\draw (0,0) -- (0,2);
 	\draw (1,0) -- (1,2);
 	\draw (2,1) -- (2,2);
 	\draw (3,1) -- (3,2);
 	\node at (0.5,0.5) {3};
 	\node at (0.5,1.5) {1};
 	\node at (1.5,1.5) {2};
 	\node at (2.5,1.5) {4};
 	\end{tikzpicture}
 } & \oplus & 
S^{
	\begin{tikzpicture}[scale=0.2, every node/.style={scale=0.4}]
	\draw (0,0) -- (2,0);
	\draw (0,1) -- (2,1);
	\draw (0,2) -- (2,2);
	\draw (0,0) -- (0,2);
	\draw (1,0) -- (1,2);
	\draw (2,0) -- (2,2);
	\node at (0.5,0.5) {3};
	\node at (0.5,1.5) {1};
	\node at (1.5,1.5) {2};
	\node at (1.5,0.5) {4};
	\end{tikzpicture}
} & \oplus & 
S^{
	\begin{tikzpicture}[scale=0.2, every node/.style={scale=0.4}]
	\draw (0,0) -- (1,0);
	\draw (0,1) -- (1,1);
	\draw (0,2) -- (2,2);
	\draw (0,3) -- (2,3);
	\draw (0,0) -- (0,3);
	\draw (1,0) -- (1,3);
	\draw (2,2) -- (2,3);
	\node at (0.5,0.5) {3};
	\node at (0.5,1.5) {2};
	\node at (0.5,2.5) {1};
	\node at (1.5,2.5) {4};
	\end{tikzpicture}
} & & \\
\oplus				&       &         &        & \oplus    &        &           &        & \oplus      &          \\
H^{2} & \cong & & & S^{
 	\begin{tikzpicture}[scale=0.2, every node/.style={scale=0.4}]
 	\draw (0,0) -- (1,0);
 	\draw (0,1) -- (3,1);
 	\draw (0,2) -- (3,2);
 	\draw (0,0) -- (0,2);
 	\draw (1,0) -- (1,2);
 	\draw (2,1) -- (2,2);
 	\draw (3,1) -- (3,2);
 	\node at (0.5,0.5) {4};
 	\node at (0.5,1.5) {1};
 	\node at (1.5,1.5) {2};
 	\node at (2.5,1.5) {3};
 	\end{tikzpicture}
 } & & \oplus & & 
S^{
	\begin{tikzpicture}[scale=0.2, every node/.style={scale=0.4}]
	\draw (0,0) -- (1,0);
	\draw (0,1) -- (1,1);
	\draw (0,2) -- (2,2);
	\draw (0,3) -- (2,3);
	\draw (0,0) -- (0,3);
	\draw (1,0) -- (1,3);
	\draw (2,2) -- (2,3);
	\node at (0.5,0.5) {4};
	\node at (0.5,1.5) {3};
	\node at (0.5,2.5) {1};
	\node at (1.5,2.5) {2};
	\end{tikzpicture}
} & & \\
\oplus				&       &         &        &           &        &           &        &             &               \\
H^{0} & \cong & S^{
 	\begin{tikzpicture}[scale=0.2, every node/.style={scale=0.4}]
 	\draw (0,0) -- (4,0);
 	\draw (0,1) -- (4,1);
 	\draw (0,0) -- (0,1);
 	\draw (1,0) -- (1,1);
 	\draw (2,0) -- (2,1);
 	\draw (3,0) -- (3,1);
 	\draw (4,0) -- (4,1);
 	\node at (0.5,0.5) {1};
 	\node at (1.5,0.5) {2};
 	\node at (2.5,0.5) {3};
 	\node at (3.5,0.5) {4};
 	\end{tikzpicture}
 }
 & & & & & & & &
\end{array}
\]
\caption{Harmonic analysis and MRA decompositions of $\Space{\Sym{4}}$}
\label{fig:decompositions}
\end{figure}

Using the combinatorial formula of Theorem \ref{th:decomposition-H-k} to calculate the multiplicities $\kappa_{\lambda}^{k}$, one can obtain some further properties. They are given in the following proposition.

\begin{proposition}[Properties of the multiplicities $\kappa_{\lambda}^{k}$]
\label{prop:properties-multiplicities}
Let $k\in\{0,\dots,n\}\setminus\{1\}$. One has the following properties:
\begin{enumerate}
	\item The part of absolute rank information of scale $k$ (in terms of MRA) is included in the part of absolute rank information of order $k$ (in terms of harmonic analysis): for any $\lambda\vdash n$ such that $\lambda_{1} < n-k$, $\kappa_{\lambda}^{k} = 0$.
	\item There is exactly one copy of the Specht module $S^{(n-1,1)}$ in each of the decompositions of the spaces $H^{k}$ for $k\in\{2,\dots,n\}$.
\end{enumerate}
\end{proposition}

\begin{proof}
To show Property 1., notice that for $Q\in\SYT (\lambda)$, one necessarily has $\eig (Q) \leq \lambda_{\lambda_{1}}$ by definition \eqref{eq:eig-definition}. Thus if $\lambda_{1} < n-k$ then $\vert\{Q\in\SYT \;\vert\; \eig (Q) = n-k\}\vert =0$ and therefore $\kappa_{\lambda}^{k} = 0$ by Theorem \ref{th:decomposition-H-k}. Property 2. is given by proposition 6.34 from \citet{Reiner13}.
\end{proof}

Notice that the decompositions illustrated by Figure \ref{fig:decompositions} satisfy all properties from Propositions \ref{prop:connection-Fourier-MRA} and \ref{prop:properties-multiplicities}.

\subsection{Alternative embedding of the MRA decomposition into $\Space{\Sn}$, connection with card shuffling and generalized Kendall's tau distances}

In this subsection we provide some further insights about the use of the alternative embedding $\phi'_{\n}$ considered in Subsection \ref{subsec:embedding-operator}, especially about its connection with $\Sn$-based harmonic analysis and card shuffling. We recall that the operator $\phi'_{\n}$ is defined in Eq. \eqref{eq:alternative-embedding-operator} by
\[
\phi'_{\n}:\qquad \Space{\Gn}\rightarrow\Space{\Sn},\qquad f\mapsto\sum_{\pi\in\Gn}\frac{\vert\pi\vert !}{n!}f(\pi)\1{\Sn(\pi)}.
\]
We also recall that $\Sn(\pi)$ is the set of linear extensions of $\pi\in\Gn$, which can be seen as the set of full rankings that induce $\pi$ on $c(\pi)$ or as the set of all the possible configurations obtained by shuffling $\pi$ with any ranking $\pi'\in\Rank{\n\setminus c(\pi)}$. The former interpretation is behind the approaches introduced in \citet{YLA02}, \citet{Kondor2010} and \citet{Sun2012} and more specifically the empirical ranking model $\widehat{p}_{N}$ defined in Equation \eqref{eq:biased-empirical-estimator} is actually equal to $\widehat{p}_{N} = \frac{1}{N}\sum_{i=1}^{N}\phi'_{\n}\left(\delta_{\Pi^{(i)}}\right)$. \citet{HGG09} also follow this interpretation and define probabilistic models on $\Sn$ as linear combinations of elements of the form $\alpha\1{\Sn(ij)} + (1-\alpha)\1{\Sn(ji)}$ with $1\leq i < j\leq n$ and $0\leq\alpha\leq 1$. We show that the part of information contained in these models can be decomposed into components that localize the same part of information as the spaces $H^{k}$.

For $k\in\{2,\dots,n\}$, we denote by $\Gamma_{n}^{k} := \Gamma_{\n}^{k}$ the set of all incomplete rankings of size $k$. Set $V^{0} = \mathbb{R}\1{\Sn}$ the space of constant functions on $\Sn$ and define for $k\in\{2,\dots,n\}$ the space $V^{k} = \phi'_{\n}(\Space{\Gamma^{k}}) = \Span\{\1{\Sn(\pi)} \;\vert\; \pi\in\Gamma^{k}\}$. One has the following nested sequence of spaces
\[
V^{0}\subset V^{2}\subset\dots\subset V^{n} = \Space{\Sn}.
\]
Indeed, $\1{\Sn} = \1{\Sn(ab)} + \1{\Sn(ba)}$ for any distinct $a,b\in\n$, and for $k\in\{2,\dots,n-1\}$, $\pi = \pi_{1}\dots\pi_{k}$ and $a\not\in c(\pi)$, one clearly has $\1{\Sn(\pi)} = \1{\Sn(a\pi_{1}\dots\pi_{k})} + \1{\Sn(\pi_{1}a\dots\pi_{k})} + \dots \1{\Sn(\pi_{1}\dots\pi_{k}a)}$. We then define the space $W^{2}$ as the orthogonal supplementary of $V^{0}$ in $V^{2}$ and for $k\in\{3,\dots,n\}$ the space $W^{k}$ as the orthogonal supplementary of $V^{k-1}$ in $V^{k}$. One thus has $V^{0}\overset{\perp}{\oplus}W^{2} = V^{2}$ and
\[
V^{k-1}\overset{\perp}{\oplus}W^{k} = V^{k} \text{ for all } k\in\{3,\dots,n\} \qquad\text{so that}\qquad \Space{\Sn} = V^{0}\oplus\bigoplus_{k=2}^{n}W^{k}.
\]
One would be highly tempted to say that for $k\in\{2,\dots,n\}$, $W^{k}$ localizes the part of information specific to scale $k$ and $V^{k}$ localizes the part of information of scales lower or equal than $k$. Fortunately, the following theorem establishes this statement.

\begin{theorem}[Decomposition associated to the alternative embedding]
\label{th:connection-other-MRA}
One has 
\[
V^{0} = \phi'_{\n}(H^{0})\qquad\text{and}\qquad W^{k} = \phi'_{\n}(H^{k}) \text{ for all } k\in\{2,\dots,n\}.
\]
In addition, $\phi'_{\n}$ establishes an isomorphism of representations of $\Sn$ between $\mathbb{H}_{n}$ and $\Space{\Sn}$, so that
\[
V^{0}\cong S^{(n)} \qquad\text{and}\qquad W^{k}\cong H^{k} \text{ for all } k\in\{2,\dots,n\}.
\]
\end{theorem}

Refer to the Appendix for the proof of Theorem \ref{th:connection-other-MRA}. The latter draws the connection between the MRA decomposition and the models that involve the embedding operator $\phi'_{\n}$. In particular it allows to say that for $\pi\in\Gn$, the indicator function $\1{\Sn(\pi)}$ contains absolute rank information up to scale $\vert\pi\vert$. It also naturally recovers some already known results. For instance applying Theorem \ref{th:decomposition-H-k} to $W^{2}$ gives Proposition 16 in \citet{HGG09}, or applying Property 1. of Proposition \ref{prop:properties-multiplicities} to $W^{k}$ can be seen as a corollary of Proposition 7 in \citet{Kondor2010}.

The spaces $W^{k}$ also have an interesting connection with card shuffling, more specifically with random-to-random shuffles. The analysis of card shuffling was introduced in the seminal contributions \citet{AD86} and \citet{BD92}. It sees a configuration of a deck of $n$ cards as a permutation of $\n$. The uncertainty about the configuration is then captured by a probability distribution over $\Sn$. The principle of the analysis of card shuffling is to study the properties of a Markov chain on $\Sn$ that represents a particular shuffle. The \textit{random-to-random shuffle}, studied in depth in \citet{UR2002}, consists in picking a card at random from the deck and replacing it at random in the deck. More generally for $k\in\{1,\dots,n-2\}$, the $k$-random-to-random shuffle consists in picking $k$ cards at random from the deck and replacing them at random positions (and in a random order) in the deck. It happens that the transition matrices of the $k$-random-to-random shuffles can be expressed with incomplete rankings.

\begin{proposition}[Connection with card shuffling]
\label{prop:connection-card-shuffling}
For $k\in\{2,\dots,n-1\}$, the transition matrix $R_{k}$ of the $(n-k)$-random-to-random shuffling satisfies for any $f\in\Space{\Sn}$:
\[
R_{k}f = (n-k)!\left(\frac{k!}{n!}\right)^{2}\sum_{\pi\in\Gamma^{k}}\left\langle f,\1{\Sn(\pi)}\right\rangle\1{\Sn(\pi)}.
\]
\end{proposition}

\begin{proof}
If one picks $n-k$ cards from a configuration $\sigma\in\Sn$, the configuration of the remaining deck is $\sigma_{\vert A}$, where $A$ is the subset of $k$ cards that were not picked. Then replacing the $n-k$ cards at random positions and in a random order in the deck can lead to any configuration $\sigma\in\Sn(\pi)$. The $n-k$-random-to-random shuffle applied to the Dirac function $\delta_{\sigma}$ therefore decomposes as the sequence of mappings
\[
\delta_{\sigma} \qquad\mapsto\qquad \frac{1}{\binom{n}{k}}\sum_{A\subset\n,\ \vert A\vert = k}\delta_{\sigma_{\vert A}} \qquad\mapsto\qquad \frac{1}{\binom{n}{k}}\sum_{A\subset\n,\ \vert A\vert = k}\frac{k!}{n!}\1{\Sn(\sigma_{\vert A})}.
\]
Thus for $f = \sum_{\sigma\in\Sn}f(\sigma)\delta_{\sigma}$, one has
\[
R_{k}f = \sum_{\sigma\in\Sn}f(\sigma)\frac{1}{\binom{n}{k}}\sum_{A\subset\n,\ \vert A\vert = k}\frac{k!}{n!}\1{\Sn(\sigma_{\vert A})} = \frac{1}{\binom{n}{k}}\frac{k!}{n!}\sum_{\pi\in\Gamma^{k}}\1{\Sn(\pi)}\sum_{\sigma\in\Sn}f(\sigma)\mathbb{I}\{\pi\subset\sigma\}.
\]
This concludes the proof.
\end{proof}

By Proposition \ref{prop:connection-card-shuffling}, it is clear that for $k\in\{2,\dots,n-1\}$ the image space of $R_{k}$ is included in $V^{k}$ and that its null space contains all spaces $W^{j}$ for $k < j \leq n$: $\image R_{k} \subset V^{k}$ and $\ker R_{k} \supset \bigoplus_{j=k+1}^{n}W^{j}$. These results can actually be refined using the results from \citet{Reiner13}. The connection is established via the following proposition.

\begin{proposition}[Connection with matrices from \cite{Reiner13}]
\label{prop:connection-Reiner}
Let $k\in\{2,\dots,n-1\}$. For $\sigma,\sigma'\in\Sn$, $R_{k}(\sigma,\sigma')$ is proportional to the number of subwords of size $k$ that $\sigma$ and $\sigma'$ have in common:
\[
R_{k}(\sigma,\sigma') = (n-k)!\left(\frac{k!}{n!}\right)^{2} \vert\{A\subset\n \text{ with } \vert A\vert = k \ \vert\ \sigma_{\vert A} = \sigma'_{\vert A}\}\vert.
\]
\end{proposition}

\begin{proof}
Noticing that $\left\langle\delta_{\sigma},\1{\Sn(\pi)}\right\rangle = \1{\Sn(\pi)}(\sigma) = \mathbb{I}\{\pi\subset\sigma\}$ for any $\pi\in\Gn$ and $\sigma\in\Sn$, one obtains
\[
R_{k}(\sigma,\sigma') = R_{k}\delta_{\sigma'}(\sigma) = (n-k)!\left(\frac{k!}{n!}\right)^{2}\sum_{\pi\in\Gamma^{k}}\mathbb{I}\{\pi\subset\sigma'\}\mathbb{I}\{\pi\subset\pi\},
\]
which gives the desired result.
\end{proof}

The number $\vert\{A\subset\n \text{ with } \vert A\vert = k \ \vert\ \sigma_{\vert A} = \sigma'_{\vert A}\}\vert$ of subwords of size $k\in\{2,\dots,n-1\}$ that $\sigma\in\Sn$ and $\sigma'\in\Sn$ have in common is equal to $\noninv_{k}(\sigma'^{-1}\sigma)$ where $\noninv_{k}$ is the statistics on $\Sn$ defined in \citet{Reiner13}. Proposition \ref{prop:connection-Reiner} thus says that the matrix $R_{k}$ is proportional to the matrix $\nu_{(k,1^{n-k})}$ considered by the authors of \citet{Reiner13}. Now, one of their major results is that these matrices are symmetric positive semidefinite and pairwise commute. They can thus be simultaneously diagonalized and the following result establishes a connection between their eigenspaces and the $W^{k}$'s.

\begin{theorem}[Null spaces of the matrices $R_{k}$]
\label{th:eigenspaces}
Each of the spaces $V^{0}$, $W^{2}$, \dots, $W^{n}$ is stable for all the matrices $R_{k}$ for $k\in\{2,\dots,n-1\}$. It is thus a direct sum of their eigenspaces. In addition, one has
\[
\ker R_{k} = \bigoplus_{j=k+1}^{n}W^{j} \qquad\text{for all } k\in\{2,\dots,n-1\}.
\]
\end{theorem}

\begin{proof}
It is proven in \citet{UR2002} that $\dim\ker R_{n-1} = d_{n}$, the number of derangements on a set of $n$ elements. Since $W^{n}\subset\ker R_{n-1}$ and $\dim W^{n} = d_{n}$ by Theorem \ref{th:connection-other-MRA}, one has $\ker R_{n-1} = W^{n}$ and $\image R_{n-1} = V^{n-1}$. Now , in \citet{Reiner13}, the authors define in equation (22) the space $V_{n,j} = \ker R_{n-j-1}\cap\image R_{n-j}$ for $j\in\{1,\dots,n-2\}$. They show that each space $V_{n,j}$ is stable for all matrices $R_{k}$. They show in addition that for each $j\in\{1,\dots,n-2\}$, $\dim V_{n,j} = \binom{n}{j}d_{n-j}$. For $j=1$ one then has
\[
V_{n,1} = \ker R_{n-2} \cap V^{n-1} \quad\text{and}\quad \ker R_{n-2}\supset W^{n-1}\oplus W^{n} \quad\text{so that}\quad V_{n,1}\supset W^{n-1}.
\]
Again, by Theorem \ref{th:connection-other-MRA}, $\dim W^{n-1} = n d_{n-1} = \dim V_{n,1}$ so that $V_{n,1} = W^{n-1}$ and therefore $\ker R_{n-2} = W^{n-1}\oplus W^{n}$. By induction, one obtains that for all $j\in\{1,\dots, n-2\}$, $V_{n,j} = W^{n-j}$ and $\ker R_{n-j} = \bigoplus_{i=0}^{j-1}W^{n-i}$. This concludes the proof.
\end{proof}

The goal of a shuffle is to mix cards so that the configuration of the deck after several iterations is closest to a purely random configuration. By definition, the component of a probability distribution over $\Sn$ that lies in the null space of a shuffle is mixed after one iteration (on average). Theorem \ref{th:eigenspaces} therefore says that the space $W^{k}$ localizes the part of information that is preserved by the $j$-random-to-random shuffles for $1\leq j\leq n-k$ but mixed by the $j$-random-to-random shuffles for $n-k+1\leq j\leq n-2$.

Finally, notice that by Proposition \ref{prop:connection-Reiner}, $R_{2}(\cdot,\cdot)$ is proportional to $\binom{n}{2} - d_{KT}(\cdot,\cdot)$ where $d_{KT}$ is the Kendall's tau distance. It thus has the same null space as the matrix $(d_{KT}(\sigma,\sigma'))_{\sigma,\sigma'\in\Sn}$. More generally for $k\in\{2,\dots,n-1\}$, Proposition \ref{prop:connection-Reiner} gives
\[
R_{k}(\sigma,\sigma') = (n-k)!\left(\frac{k!}{n!}\right)^{2}\left(\binom{n}{k} - d^{k}(\sigma,\sigma')\right),
\]
where $d^{k}(\sigma,\sigma') := \vert\{A\subset\n\;\vert\;\vert A\vert = k \text{ and } \sigma_{\vert A}\neq \sigma_{\vert A'}\}$ is the number of $k$-wise disagreements between $\sigma$ and $\sigma'$, and can therefore be seen as an extension of the Kendall's tau distance. Hence the matrices of the distances $d^{k}$ for $k\in\{2,\dots,n-1\}$ pairwise commute and their null spaces are given by Theorem \ref{th:eigenspaces}.

To conclude this section, we summarize the interpretations that can be given to the spaces $V^{0}, W^{2}, \dots, W^{n}$ and thus to the different scales of the MRA. For $k\in\{2,\dots,n\}$:
\begin{itemize}
	\item $W^{k}$ is the space spanned by the $\1{\Sn(\pi)}$'s for $\pi\in\Gamma^{k}$ that localizes the part of absolute rank information specific to scale $k$.
	\item $W^{k}$ localizes the part of information preserved by the $j$-random-to-random shuffles for $1\leq j\leq n-k$ but mixed by the $j$-random-to-random shuffles for $n-k+1\leq j\leq n-2$.
	\item $W^{k}$ localizes the part of additional information captured by the distance $d^{k}$ compared to $d^{k-1}$.
\end{itemize}

\subsection{The specific case of absolute rank information at scale $2$}
\label{subsec:pairwise-comparisons}

Here we analyze in particular the part of absolute rank information at scale $2$ or in other words the part of information contained in pairwise marginals. By Theorem \ref{th:decomposition-H-k}, one has the isomorphism $H^{2}\cong S^{(n-1,1)}\oplus S^{(n-2,1,1)}$. We give an explicit construction of subspaces of $H^{2}$ that correspond to this decomposition. First we recall that $H^{2} = \bigoplus_{\{a,b\}\subset\n}H_{\{a,b\}}$ with $\dim H_{\{a,b\}} = 1$ for each pair $\{a,b\}\subset\n$, so that $\dim H^{2} = \binom{n}{2}$. The following proposition gives an explicit basis for each $H_{\{a,b\}}$ and thus for $H^{2}$. Its proof is straightforward and left to the reader.

\begin{proposition}[Canonical basis of $H^{2}$]
\label{prop:canonical-basis}
For any $a,b\in\n$ with $a\neq b$, the element $x_{a\succ b} = \delta_{ab} - \delta_{ba}$ generates the space $H_{\{a,b\}}$. By convention, we choose for each pair $\{a,b\}\subset\n$ with $a < b$ the element $x_{a\succ b}$ to be the canonical basis of $H_{\{a,b\}}$. The canonical basis of $H^{2}$ is then given by the family $\left( x_{a\succ b}\right)_{1\leq a < b \leq n}$.
\end{proposition}

We use the canonical basis introduced in Proposition \ref{prop:canonical-basis} to construct the two following subspaces of $H^{2}$:
\begin{equation}
\label{eq:subspaces-H-2}
\begin{aligned}
H^{2}_{1} &= \Span\Bigg\{e_{a} := \sum_{b\in\n,\ b\neq a}x_{a\succ b} \quad \Bigg\vert\quad  a\in\n\Bigg\}\\
\text{and}\qquad H^{2}_{2} &= \Span\Bigg\{f_{(a,b)} := \sum_{c\in\n,\ c\not\in\{a,b\}}\left( x_{a\succ b} + x_{b\succ c} + x_{c\succ a}\right) \quad \Bigg\vert\quad \{a,b\}\subset\n \Bigg\}.
\end{aligned}
\end{equation}
The following theorem shows that they provide a decomposition of $H^{2}$ isomorphic to $S^{(n-1,1)}\oplus S^{(n-2,1,1)}$. Its proof is left in Appendix.

\begin{theorem}[Explicit decomposition of $H^{2}$]
\label{th:explicit-decomposition-H-2}
The spaces $H^{2}_{1}$ and $H^{2}_{2}$ defined in Equation \eqref{eq:subspaces-H-2} satisfy the following properties:
\[
H^{2} = H^{2}_{1}\oplus H^{2}_{2} \qquad\text{with}\qquad H^{2}_{1} \cong S^{(n-1,1)} \qquad\text{and}\qquad H^{2}_{2}\cong S^{(n-2,1,1)}.
\]
\end{theorem}

It happens that the spaces $H^{2}_{1}$ and $H^{2}_{2}$ defined in Equation \eqref{eq:subspaces-H-2} appear in several other mathematical constructions and therefore have different interpretation. We first detail the connection with social choice theory. These spaces are indeed closely related to the ``Borda space'' $\mathcal{B}_{n}$ and the ``Condorcet space'' $\mathcal{C}_{n}$ introduced in \citet{Sibony2014}. More specifically, it is easy to see that one has $\mathcal{B}_{n} = \phi'_{\n}(H^{2}_{1})$ and $\mathcal{C}_{n} = \phi'_{\n}(H^{2}_{2})$. Since $\phi'_{\n}$ is an isomorphism between $\mathbb{H}_{n}$ and $\Space{\Sn}$ by Theorem \ref{th:connection-other-MRA}, this implies that $H^{2}_{1}$ and $H^{2}_{2}$ respectively localize the same part of information as $\mathcal{B}_{n}$ and $\mathcal{C}_{n}$. Hence, as explained in \citet{Sibony2014}, $H^{2}_{1}$ localizes the part of information captured by the \textit{Borda count} and $H^{2}_{2}$ localizes the part of information of ``pairwise voting inconsistencies'' responsible for the \textit{Condorcet paradox}, refer for instance to \citet{Saari2000}, \citet{CS2013} or \citet{Crisman2014} for more details.

The second connection we detail is with the HodgeRank framework. Introduced in \cite{JLYY11}, it models a collection of pairwise comparisons as an oriented flow on the graph with vertices $\n$ where two items are linked if the pair appears at least once in the comparisons. The collection of observed pairwise comparisons is he observation design $\mathcal{A}$ of our present setting. The space of edge flows considered in \cite{JLYY11} is then equal to the space $\mathbb{H}(\mathcal{A}) = \bigoplus_{\{a,b\}\in\mathcal{A}}H_{\{a,b\}}$. The HodgeRank framework then decomposes any element of this space as the sum of three components: a ``gradient flow'' that corresponds to globally consistent rankings, a ``curl flow'' that corresponds to locally inconsistent rankings, and a ``harmonic flow'', that corresponds to globally inconsistent but locally consistent rankings. The following proposition establishes the connection with the present work. Its proof is left in Appendix.

\begin{proposition}[Connection with HodgeRank]
\label{prop:connection-HodgeRank}
In the particular case where $\mathcal{A} = \{\{a,b\}\subset\n \}$, the space of edge flows in the HodgeRank framework is equal to $H^{2}$, the space of gradient flows to $H_{1}^{2}$, the space of curl flows to $H_{2}^{2}$ and the space of harmonic flows is null. The Hodge decomposition then boils down to $H^{2} = H^{2}_{1}\oplus H^{2}_{2}$. There is no particular connection in the general case.
\end{proposition}

\section{Discussion and future directions}
\label{sec:discussion}

Here we describe some relevant developments for the MRA framework. We do not consider its application to the different statistical problems mentioned in Subsection \ref{subsec:problems} as each requires a specific treatment. Instead we focus on general results that would be useful for all applications.

\subsection{The need for regularization}

As already mentioned in Subsection \ref{subsec:general-method}, the MRA framework needs to be applied together with a regularization procedure in order to be most efficient. This is all the more true than the number $n$ of items or the maximal size $K$ of a ranking increase. A simple way to see it is by comparing the number of degrees of freedom of the wavelet empirical estimator $\mathbf{\widehat{X}}$ defined in Equation \eqref{eq:wavelet-estimator} with the minimal number of parameters required to store $\mathcal{D}_{N}$. By construction $\mathbf{\widehat{X}} = (\widehat{X}_{B})_{B\in\SubsetsWE{\mathcal{A}}}$ is an element of $\mathbb{H}(\SubsetsWE{\mathcal{A}})$ and therefore the former is equal to $\sum_{B\in\SubsetsWE{\mathcal{A}}}d_{\vert B\vert}$, whereas the latter is equal to $\min(N,\sum_{A\in\mathcal{A}}\vert A\vert !)$ by Lemma \ref{lem:dataset-storage}. In a case where one observes all subsets of items of size lower or equal than $k$, that is $\mathcal{A} = \{A\subset\n \;\vert\; 2\leq \vert A\vert \leq K\}$, one has $\Subsets{\mathcal{A}} = \mathcal{A}$ and thus
\[
\sum_{B\in\SubsetsWE{\mathcal{A}}}d_{\vert B\vert} = 1 + \sum_{k=2}^{K}\binom{n}{k}d_{k} \qquad\text{whereas}\qquad \sum_{A\in\mathcal{A}}\vert A\vert ! = \sum_{k=2}^{K}\binom{n}{k}k!.
\]
Since $d_{k} \geq k!/3$ for all $k\geq 0$ by a classic result from elementary combinatorics, this means that the initial dimension of the data is at most reduced by a factor $3$ in the wavelet empirical estimator. In particular the number of degrees of freedom of the latter remains in $O(K!n^{K})$. This shows that using the wavelet empirical estimator alone reduces very little the dimension of the problem and therefore does not provide a strong generalization effect. 

We point out that the MRA representation first allows to exploit the consistency assumption \eqref{eq:consistency-assumption} efficiently. It overcomes the statistical challenge of dealing with the heterogeneity of incomplete rankings and the computational challenge of manipulating a ranking model. But by construction, it only transfers information between included subsets of items. Transferring more information requires an additional regularization procedure. Also, it does not provide sparse representations of usual ranking models. This is for instance illustrated by Figure \ref{fig:approximations}, where one can see that setting the unobserved wavelet projections $X_{B}$ to $0$ does not provide better results than setting them randomly. Constructing a basis where usual ranking models would be sparse needs indeed an additional regularity assumption. However, the MRA representation provides a general and flexible framework to define such regularization procedures and regularity assumptions.

\subsection{Regularization procedures}

Here we describe some regularization that one may consider but the list is of course non exhaustive. Our suggestions are based on intuition and analogy with classic regularization procedures on other types of data. Hence they do not come with any theoretical guarantees. Finding a good regularity assumption and the associated regularization procedure in the feature space $\mathbb{H}_{n}$ largely remains an open problem.

\ \\
\noindent 
{\bf Kernel-based smoothing.} The most usual way to define a notion of regularity is to say that a function $f$ is regular if ``$f(x)\simeq f(y)$'' for ``$x\simeq y$''. In this case, the knowledge of $f(x)$ can be used to infer some knowledge about $f(y)$. Thus if one has an estimation of $f$ at some point $x$ and assumes that $f$ is regular, she can obtain estimations for points $y\simeq x$. A typical approach is then to regularize an initial estimator by applying a smoothing kernel $K_{h}$ that will ``diffuse'' the knowledge of $f(x)$ to points $y$ close to $x$. The parameter $h$ is usually a window parameter that controls both the ``speed and the range of the diffusion''. As we detailed in Subsection \ref{subsec:existing-approaches}, kernel smoothing for incomplete rankings is already used in \cite{Kondor2010} and \cite{Sun2012}. The difference here is that we propose to define kernels on the feature space $\mathbb{H}_{n}$ instead of the space $\Space{\Sn}$.

Here we propose an approach to transpose these ideas for the feature space $\Hn$. By analogy, one wants to say that an element $\mathbf{X} = (X_{B})_{B\in\SubsetsWE{\n}}$ is regular if ``$X_{B}\simeq X_{B'}$'' for ``$B\simeq B'$''. The first step is therefore to define relevant meanings for ``$X_{B}\simeq X_{B'}$'' and ``$B\simeq B'$''. We assert that the MRA representation already exploits the consistency assumption to transfer information between included subsets and therefore between different scales. Transferring information between elements $X_{B}$ and $X_{B'}$ indexed by two subsets of different size is then not relevant. Hence we define a notion of regularity for each subspace $H^{k}$ and from now on we fix $k\in\{0,2,\dots,n\}$. First we propose to consider the distance $D_{k}$ defined for $B,B'\in\SubsetsWE{\n}$ with $\vert B\vert = \vert B'\vert = k$ by
\[
D_{k}(B,B') = \frac{1}{2}\left(k - \vert B\cap B'\vert\right)
\] 
(the proof that $D_{k}$ is a distance on $\{B\subset\n \;\vert\; \vert B\vert = k\}$ is left to the reader). Two subsets $B,B'$ with $\vert B\vert = \vert B'\vert = k$ thus have distance $1$ if they have $k-1$ items in common, $2$ if they have $k-2$ items in common, \dots, and $k$ if they have no item in common. The distance $D_{k}$ is also the distance on the graph with set of nodes $\{B\subset\n \;\vert\; \vert B\vert = k\}$ and where $B$ and $B'$ are connected if they have $k-1$ items in common. An illustration of this graph for $n = 5$ and $k = 2$ is provided on Figure \ref{fig:graph-subsets}.

\begin{figure}
\centering
\begin{tikzpicture}[scale = 2]
\node (12) at (0,1) {$\{1,2\}$};
\node (13) at (0.87,0.5) {$\{1,3\}$};
\node (14) at (0.87,-0.5) {$\{1,4\}$};
\node (23) at (0,-1) {$\{2,3\}$};
\node (24) at (-0.87,-0.5) {$\{2,4\}$};
\node (34) at (-0.87,0.5) {$\{3,4\}$};
\draw 	(12) -- (13)
		(12) -- (14)
		(12) -- (23)
		(12) -- (24)
		(13) -- (14)
		(13) -- (23)
		(13) -- (34)
		(14) -- (24)
		(14) -- (34)
		(23) -- (24)
		(23) -- (34)
		(24) -- (34);
\end{tikzpicture}
\caption{Graph on pairs of items for $n=5$}
\label{fig:graph-subsets}
\end{figure}
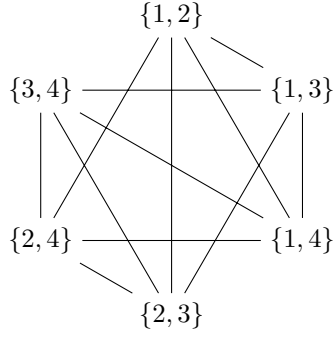

We now define a relevant meaning for ``$X_{B}\simeq X_{B'}$''. The difficulty is that for $B\neq B'$, the elements $X_{B}$ and $X_{B'}$ lie in different spaces and how they should be compared is not obvious. To tackle this problem we propose to send one to the space of the other and then to compare them. For $B,B'\in\SubsetsWE{\n}$ we define the set
\[
\Bij(B,B') = \{\tau:B\rightarrow B' \text{ bijection } \;\vert\; \tau(b) = b \text{ for all } b\in B\cap B'\}.
\]
For $\tau\in\Bij(B,B')$ we denote by $\tau(\pi_{1}\dots\pi_{k}) := \tau(\pi_{1})\dots\tau(\pi_{k})$ and define for $X_{B}\in H_{B}$ the element $\tau\cdot X_{B} := \sum_{\pi\in\Rank{B}}X_{B}(\pi)\delta_{\tau(\pi)}$. With a proof similar to the one of \ref{prop:action-on-H}, it is easy to show that $\tau\cdot X_{B}\in H_{B'}$. We then say that ``$X_{B}\simeq X_{B'}$'' if 
\[
X_{B'} \simeq \frac{1}{\vert\Bij(B,B')\vert}\sum_{\tau\in\Bij(B,B')}\tau\cdot X_{B} \qquad\text{in } H_{B'}.
\]
The kernels associated to the regularity assumption ``$X_{B}\simeq X_{B'}$'' for ``$B\simeq B'$'' are then functions $K_{h}:H^{k}\rightarrow H^{k}$ defined by
\[
K_{h}:X_{B} \mapsto\sum_{\vert B'\vert = k}\frac{q_{h}(D_{k}(B,B'))}{\vert\Bij(B,B')\vert}\sum_{\tau\in\Bij(B,B')}\tau\cdot X_{B},
\]
where $q_{h}:\mathbb{N}\rightarrow\mathbb{R}$ is a nonnegative function such that $\sum_{\pi\in\Gamma^{k}}K_{h}X_{B}(\pi) = \sum_{\pi\in\Rank{B}}X_{B}(\pi)$. Since for any $B'\subset\n$ with $\vert B'\vert = k$ and $\tau\in\Bij(B,B')$, $\sum_{\pi\in\Gamma^{k}}\tau\cdot X_{B}(\pi) = \sum_{\pi\in\Rank{B}}X_{B}(\pi)$, the condition on $q_{h}$ boils down to 
\[
\sum_{\vert B'\vert = k}q_{h}(D_{k}(B,B')) = 1 \qquad\textit{i.e.}\qquad \sum_{j=0}^{k}q_{h}(j)\binom{k}{j}\binom{n-k}{j} = 1.
\]
One can take for instance $q_{h}(j) = [(h+1)\binom{k}{j}\binom{n-k}{j}]^{-1}$ if $0\leq j\leq h$ and $0$ otherwise.

\ \\
\noindent
{\bf Penalty minimization and sparsity.} As already mentioned in Subsection \ref{subsec:general-method}, another classic approach to define regularization procedure is through the minimization of a penalty function. One chooses a dissimilarity measure $\Delta$ on $\Hn$, and then defines a regularized version of an initial element $\mathbf{X}\in\Hn$ as the solution of a minimization problem of the form
\begin{equation}
\label{eq:penalty-minimization}
\min_{\mathbf{X'}\in\Hn} \Delta(\mathbf{X},\mathbf{X'}) + \lambda\Omega(\mathbf{X'}),
\end{equation}
where $\Omega:\Hn\rightarrow\mathbb{R}$ is a penalty function and $\lambda > 0$ is a regularization parameter. As $\Hn$ is constructed as $\bigoplus_{B\in\SubsetsWE{\n}}H_{B}$, it is natural to define a dissimilarity measure $\Delta$ of the form
\[
\Delta(\mathbf{X},\mathbf{X'}) = \sum_{B\in\SubsetsWE{\n}}\Delta_{B}(X_{B},X'_{B}),
\] 
where for each $B\in\SubsetsWE{\n}$, $\Delta_{B}$ is a dissimilarity measure on $H_{B}$. If one takes $\Delta_{B} = \Vert\cdot\Vert_{B}^{2}$ then $\Delta := \Vert\cdot\Vert^{2}_{\Gn}$, the Euclidean norm on $\Space{\Gn}$. The challenge in this approach lies more in the definition of a ``good'' penalty function $\Omega$. If one wants to enforce the regularity assumption described previously, one can use the Tikhonov regularization approach and take $\Omega(\mathbf{X'}) = \Vert K_{h}\mathbf{X'} - \mathbf{X'} \Vert_{\Gn}^{2}$. The use of a penalty function can also force the solution of \eqref{eq:penalty-minimization} to be sparse in a certain basis or dictionary. The first challenge is then to define a dictionary where ``regular'' elements of $\Hn$ should be sparse in. As explained previously, such a dictionary should not contain elements that lie in one single space $H_{B}$ only. In other words, ``regular'' elements of $\Hn$ should not have the form $\sum_{i=1}^{s}\alpha_{i}X_{B_{i}}$ with a small $s$, where for $i\in\{1,\dots,s\}$, $B_{i}\in\SubsetsWE{\n}$, $X_{B_{i}}\in H_{B_{i}}$ and $\alpha_{i}\in\mathbb{R}$. Instead, we advocate to define atoms of the form $\mathbf{X}_{\mathcal{B}}^{k} = \sum_{B\in\mathcal{B}}X_{B}$ with $\mathcal{B}\subset \{B\subset\n \;\vert\; \vert B\vert = k \}$ and $X_{B}\in H_{B}$ for each $B\in\mathcal{B}$. As an example, we consider for distinct items $a,b\in\n$ the following element (defined in Proposition \ref{prop:canonical-basis}):
\[
x_{a\succ b} = \delta_{ab} - \delta_{ba} \in H_{\{a,b\}}.
\]
Then one can consider a dictionary with atoms 
\[
\mathbf{X}^{2}_{a,B} = \sum_{b\in B}x_{a\succ b}\in \bigoplus_{b\in B}H_{\{a,b\}} \qquad\text{for } a\in\n \text{ and } B\subset\n\setminus\{a\}.
\]
Such an atom localizes the part of rank information that says that item $a$ is preferred to each of the items of $B$ in pairwise comparisons in the sense that for $i,j\in\n$ with $i\neq j$,
\[
\phi_{\{i,j\}}\mathbf{X}^{2}_{a,B} =
\left\{
\begin{aligned}
\delta_{ab} - \delta_{ba} 	&\qquad\text{if } \{i,j\} = \{a,b\} \text{ with } b\in B\\
0							&\qquad\text{otherwise.}
\end{aligned}
\right.
\]

\ \\
\noindent
{\bf Fourier band-limited approximation.} Another classic regularization procedure is to compute the Fourier transform of a function, truncate it to the low frequencies, and output its inverse. The performance of this procedure for functions on Euclidean spaces stems from the fact that the Fourier spectrum of irregularities is usually localized in high frequencies. Keeping only the low frequencies of the Fourier spectrum of a function $f$ therefore leads to a regularized version of $f$. The analogue of this approach can been applied for functions on the symmetric group, using $\Sn$-based harmonic analysis. The additional challenge is that ``frequencies'' are then partitions of $n$ (see Subsection \ref{subsec:background-harmonic-analysis}) and thus are not naturally ordered. Fortunately the dominance order (defined in Subsection \ref{subsec:absolute-rank-information}) is a partial order on partitions of $n$ that orders Fourier coefficients by a certain level of ``smoothness''. Hence the band-limited approximation procedure has been proven to be efficient on real datasets \citep[see][]{HGG07, Irurozki2011}.

This regularization procedure can also be applied to the statistical analysis of incomplete rankings:
\begin{enumerate}
	\item Compute the wavelet empirical estimator $\mathbf{\widehat{X}}\in\Hn$
	\item Apply the procedure to $\phi_{\n}\mathbf{\widehat{X}}\in\Space{\Sn}$
	\item Compute its wavelet transform to obtain a regularized wavelet estimator $\mathbf{\tilde{X}}\in\Hn$
\end{enumerate}
This procedure is theoretical because it would not lead to tractable computations. For that, one needs to find a way to obtain $\mathbf{\tilde{X}}$ from $\mathbf{\widehat{X}}$ without passing by $\phi_{\n}\mathbf{\widehat{X}}$. We point out this direction however because we assert that this regularization procedure gains a new interpretation when applied to the statistical analysis of incomplete rankings: it allows to regularize small pieces of relative rank information into global parts of absolute rank information. Assume for instance that one observes pairwise comparisons and keeps only absolute rank information of level $1$. Besides the piece of rank information of level $0$, there are $n(n-1)/2$ potential degrees of freedom in the data, one for the piece of relative rank information related to each pair in $\n$. By contrast, there are only $n-1$ degrees of freedom in the part of absolute rank information localized in the copy of $S^{(n-1,1)}$ that appears in the decomposition of $H^{2}$ (see Subsection \ref{subsec:two-decompositions}). Keeping only this component therefore allows to enforce the regularity constraints of absolute rank information on the pieces of relative rank information captured by $\mathbf{\widehat{X}}$.

\ \\
\noindent
{\bf Local regularization.} In some applications, one is only interested in using an estimator to make local predictions on small subsets of items. One then does not have to regularize the full wavelet empirical estimator $\mathbf{\widehat{X}}$ but can regularize only the coefficients involved in each prediction. For $A\in\Subsets{\n}$, we recall that the estimation of the marginal $P_{A}$ of the true ranking model $p$ provided by $\mathbf{\widehat{X}}$ is equal to $\phi_{A}\sum_{B\in\SubsetsWE{A}}\widehat{X}_{B}$. One therefore only needs to regularize the coefficients $(\widehat{X}_{B})_{B\in\SubsetsWE{A}}\in\mathbb{H}(\SubsetsWE{A})$ to improve the estimation of $P_{A}$. Thanks to the multi-scale nature of $\Subsets{\n}$, the three aforementioned families of regularization procedures naturally apply to $\mathbb{H}(\SubsetsWE{A})$. Notice however that if one wants to apply the the Fourier band-limited approximation procedure, she will have to use the Fourier transform based on $\Sym{A}$, the group of permutations of $A$. The regularization then will involve ``absolute rank information on $A$'' and not absolute rank information on $\n$.

The drawbacks of a local regularization procedure is of course that it does not allow to transfer information from subsets of items not included in $A$ to subsets of items included in $A$. The major advantage however is the much lower computational cost: the parameter $n$ that would appear in any of the procedures when regularizing globally becomes $\vert A\vert$ when regularizing locally, which is much smaller in practical applications.

\subsection{Wavelet basis and connection with Hopf algebras}

As already mentioned, one major difference between the MRA representation and classic multiresolution analysis on a Euclidean space is that the wavelet projections $\Psi_{B}F$ of a function $F\in\Space{\Gn}$ are not scalar coefficients but vectors. In other words, they project on subspaces of the feature space $\Hn$ but not on subspaces of dimension $1$. To this purpose one would need to refine the decomposition $\Hn = \bigoplus_{B\in\SubsetsWE{\n}}H_{B}$ and construct a basis of $\Hn$ consistent with it, in the sense that it is equal to the concatenation of the bases of each $H_{B}$ for $B\in\SubsetsWE{\n}$.

An example of such a basis is introduced in \citet{CJS2014}. We did not recall it in the present article because it is not required for the definition and use of the MRA framework. We however point out some interesting observations about it. The basis is generated in \citet{CJS2014} by an algorithm which is a slight variation from the algorithm introduced in \citet{RT11} to define a basis for the top homology space of the complex of injective words over the field $\mathbb{F}_{2} = \mathbb{Z}/2\mathbb{Z}$ of two elements. Now, it happens that a sub-procedure of the algorithm used in \citet{CJS2014} is exactly the same algorithm as the one used to generate the Lyndon words of length $n$ \citep[see][]{CFL58}. Such Lyndon words can be used to construct a basis for the $n$th homogeneous component of the free Lie algebra and are also involved in the study of Hopf algebras, in particular in \citet{DPR14} where they are used to define eigenvectors for the transition matrix of the Markov chain associated to a certain card shuffle. At last, the basis obtained for $H_{\n}$ in \citet{CJS2014} is exactly the same as the basis computed in \citet{AL11} for what the authors call ``the Hopf kernel of the canonical morphism of Hopf monoids between the species of linear orders and the exponential species'' (see part 5.3). These connections may bring new results or insights to the MRA framework.

\subsection{Extension to the analysis of incomplete rankings with ties}

In practical applications, one may observe incomplete rankings with ties. For instance if a user chooses some items $a_{1},\dots, a_{k}$ among a selection of proposed items $\{a_{1},\dots,a_{k},b_{1},\dots,b_{l}\}$ then one can model her preference by the ranking $a_{1},\dots,a_{k} \succ b_{1},\dots,b_{l}$. More generally, incomplete rankings with ties are partial orders of the form
\begin{equation}
\label{eq:incomplete-ranking-with-ties}
a_{1,1},\dots,a_{n_{1},1}\succ\dots\succ a_{1,r},\dots, a_{n_{r},r} \qquad\text{with } r\geq 1 \text{ and } \sum_{i=1}^{r}n_{i} < n.
\end{equation}
Observations then cannot be represented as incomplete rankings anymore, but as incomplete rankings with ties, and the MRA framework needs to be extended before it can be applied. To do so,observe that an incomplete ranking with ties of the form \eqref{eq:incomplete-ranking-with-ties} can be seen as a partial ranking on the subset of items $\{a_{1,1},\dots,a_{n_{r},r}\}$. We therefore propose to extend the MRA framework as follows:
\begin{enumerate}
	\item Construct an estimator $\widehat{Q}_{A}$ on each observed subset of items $A$ using any method to analyze partial rankings from the literature
	\item Compute the wavelet transforms of all the $\widehat{Q}_{A}$'s and average them to obtain a wavelet estimator $\mathbf{\tilde{X}}$
	\item Perform the task related to the considered application in the feature space $\Hn$ using $\mathbf{\tilde{X}}$ as empirical distribution 
\end{enumerate}
Of course, this extended framework needs to be developed for each statistical application with respect to the considered method to analyze partial rankings.

\section{Conclusion}

This article introduces a novel general framework for the statistical analysis of incomplete rankings. The latter problem is defined here in a rigorous setting: incomplete rankings on a set of $n$ items $\n$ are drawn from one ranking model $p$ over $\Sn$ and observed on random subsets of items drawn independently. Each ranking thus provides information about the marginal of $p$ over the involved subset of items only, and the purpose of the statistical analysis of incomplete rankings is to consolidate and transfer information between observations to recover some part of $p$. Performing this task efficiently represents a great statistical challenge as there is no simple and general way to handle such heterogeneous data. It also presents a great computational challenge as the computation of marginals become by far intractable in modern applications where $n$ is around $10^{4}$ and the rankings involve less than $10$ items. 

The MRA framework we have introduced overcomes these challenges by sending the data into a tailor-made feature space. Procedures can then be defined to infer only the parameters of $p$ than are accessible/identifiable, with a complexity shown to be of the same order as that related to the storage of the dataset. The MRA framework thus offers a general and flexible approach to tackle any statistical application based on incomplete rankings. These advantages stem from the strong localization properties of the MRA representation, established in the present article. The latter decomposes any function of rankings into components that each localize the part of information specific to the marginal on one subset. This decomposition thus fits the multi-scale structure of the marginals and has a natural multiresolution interpretation.
 
We then established multiple connections between the MRA representation and other mathematical constructions. In particular we showed that if the latter can be interpreted as localizing \textit{relative rank information}, $\Sn$-based harmonic analysis can be interpreted as localizing \textit{absolute rank information} in contrast. We then provided a precise relationship between the part of information that contains all relative rank information of a given scale in the MRA representation and the corresponding pieces of absolute rank information expressed in the $\Sn$-based harmonic analysis framework.

We believe that the contributions of this article have several interests. From a fundamental point of view, the MRA decomposition introduces a novel yet natural way to decompose rank information. The connections we establish also provide important insights about the numerous mathematical objects involved in the analysis of ranking data. All these results should be of great interest to gain a better understanding of the latter and obtain new theoretical guarantees about its associated procedures. From a methodological point of view, the MRA framework provides a novel, efficient and general approach to analyze incomplete rankings. Though this article only settles the basis of the framework, we described at length many directions to extend it, in particular through the design of efficient regularization procedures, as well as to apply it to the relevant statistical problems. We therefore believe that it paves the way to many future developments in the statistical analysis of ranking data.


\appendix

\gdef\thesection{Appendix \Alph{section}.}

\section{Commutation with translations}

\begin{proposition}
	\label{prop:translation-invariance}
	For any $A,B\in\SubsetsWE{\n}$ and $\tau\in\Sn$,
	\[
	T_{\tau}\phi_{A} = \phi_{\tau(A)}T_{\tau} \qquad\text{and}\qquad T_{\tau}\Psi_{B} = \Psi_{\tau(B)}T_{\tau}.
	\]
\end{proposition}

\begin{proof}
	For $\pi\in\Gn$ and $\pi'\in\Rank{A}$ it is clear that $\pi\sqsubset\pi' \Rightarrow \tau(\pi)\sqsubset\tau(\pi')$. Hence, the mapping $\pi'\mapsto\tau(\pi')$ being injective, $\tau(\{\pi'\in\Rank{A} \;\vert\; \pi\sqsubset\pi'\}) = \{\pi'\in\Rank{\tau(A)} \;\vert\; \tau(\pi)\sqsubset\pi'\}$ and one has
	\begin{align*}
	(\vert A\vert - \vert\pi\vert + 1)!T_{\tau}\phi_{A}\delta_{\pi} 
	&= T_{\tau}\1{\{\pi'\in\Rank{A} \;\vert\; \pi\sqsubset\pi'\}}\\
	&= \1{\tau(\{\pi'\in\Rank{A} \;\vert\; \pi\sqsubset\pi'\})}\\
	&= \1{\{\pi'\in\Rank{\tau(A)} \;\vert\; \tau(\pi)\sqsubset\pi'\}}\\
	&= (\vert A\vert - \vert\pi\vert + 1)!\phi_{\tau(A)}\delta_{\tau(\pi)}.
	\end{align*}
	This proves the first part. For the second part, let $B'\in\SubsetsWE{\n}$ with $B\subset B'$ and $F\in\Space{\Rank{B'}}$. By Theorem \ref{th:MRA-general} one has $F = \phi_{B'}\sum_{B\in\SubsetsWE{B'}}\Psi_{B}F$. Applying the operator $T_{\tau}$ and using the previous result one obtains
	\[
	T_{\tau}F = T_{\tau}\phi_{B'}\sum_{B\in\SubsetsWE{B'}}\Psi_{B}F = \phi_{\tau(B')}\sum_{B\in\SubsetsWE{B'}}T_{\tau}\Psi_{B}F
	\]
	where for each $B\in\SubsetsWE{B'}$, $T_{\tau}\Psi_{B}F\in H_{\tau(B)}$ by Proposition \ref{prop:action-on-H}. On the other hand, applying Theorem \ref{th:MRA-general} to $T_{\tau}F\in\Space{\Rank{\tau(B')}}$ gives
	\[
	T_{\tau}F = \phi_{\tau(B')}\sum_{B\in\SubsetsWE{\tau(B')}}\Psi_{B}T_{\tau}F = \phi_{\tau(B')}\sum_{B\in\SubsetsWE{B'}}\Psi_{\tau(B)}T_{\tau}F.
	\]
	The uniqueness of the MRA decomposition concludes the proof.
\end{proof}

\section{Proofs of Propositions \ref{prop:complexity-alpha-coefficients} and \ref{prop:complexity-synthesis-operator}}

We first show Proposition \ref{prop:complexity-alpha-coefficients} and give at the same time a method to compute the coefficients $\alpha_{B}(\pi,\pi')$ for $\pi,\pi'\in B$ and $B\in\Subsets{\n}$ with $\vert B\vert \leq k$; for $k\in\{2,\dots,n\}$. The first simplification in their computation stems from the following lemma. For $\pi = \pi_{A}\dots\pi_{k}\in\Gamma_{n}$ and $\tau\in\Sn$, we denote by $\tau(\pi)$ the word $\tau(\pi_{1})\dots\tau(\pi_{k})\in\Rank{\tau(c(\pi))}$, as in Subsection \ref{subsec:two-decompositions}.

\begin{lemma}
	\label{lem:alpha-translation-invariance}
	Let $B\in\Subsets{\n}$ and $\tau\in\Sn$ a permutation that keeps the order of the items in $B$, \textit{i.e.} such that for all $b,b^{\prime}\in B$, $b < b^{\prime} \Rightarrow \tau(b) < \tau(b^{\prime})$. Then for all $\pi,\pi^{\prime}\in\Rank{B}$, 
	\[ 
	\alpha_{B}(\pi,\pi^{\prime}) = \alpha_{\tau(B)}(\tau(\pi), \tau(\pi^{\prime}))
	\]
\end{lemma}

\begin{proof}
	By Definition \ref{def:alpha-coefficients} one has $\Psi_{B}\delta_{\pi'} = \sum_{\pi\in\Rank{B}}\alpha_{B}(\pi,\pi')\delta_{\pi}$. Applying $T_{\tau}$ then gives 
	\[
	T_{\tau}\Psi_{B}\delta_{\pi'} = \sum_{\pi\in\Rank{B}}\alpha_{B}(\pi,\pi')\delta_{\tau(\pi)}.
	\]
	On the other hand, Proposition \ref{prop:translation-invariance} gives
	\[
	T_{\tau}\Psi_{B}\delta_{\pi'} = \Psi_{\tau(B)}\delta_{\tau(\pi')} = \sum_{\pi\in\Rank{\tau(B)}}\alpha_{\tau(B)}(\pi,\tau(\pi'))\delta_{\pi} = \sum_{\pi\in\Rank{B}}\alpha_{\tau(B)}(\tau(\pi),\tau(\pi'))\delta_{\tau(\pi)}.
	\]
	Identifying the coefficients concludes the proof.
\end{proof}

Property 3. of Lemma \ref{lem:alpha-translation-invariance} implies two simplifications:
\begin{itemize}
	\item First, for $k\in\{2,\dots,n\}$, the coefficients $(\alpha_{B}(\pi,\pi^{\prime}))_{\pi,\pi^{\prime}\in\Rank{B}}$ are obtained directly from the $(\alpha_{\set{k}}(\pi,\pi^{\prime}))_{\pi,\pi^{\prime}\in\Rank{\set{k}}}$ for all $B\subset\n$ with $\vert B\vert = k$.
	\item Second, for $B = \{b_{1},\dots,b_{k}\}\in\Subsets{\n}$ with $b_{1} < \dots < b_{k}$, the coefficients $(\alpha_{B}(\pi,\pi'))_{\pi'\in\Rank{B}}$ are obtained directly from the $(\alpha_{B}(b_{1}\dots b_{k},\pi'))_{\pi'\in\Rank{B}}$ for any $\pi\in\Rank{B}$.
\end{itemize} 

\begin{example}
	Let $B = \{2,4,5\}$ and $\tau\in\Sn$ such that $\tau(2) = 1, \tau(4) = 2$ and $\tau(5) = 3$. Then for $\pi,\pi^{\prime}\in\Rank{\{2,4,5\}}$, $\alpha_{\{2,4,5\}}(\pi,\pi^{\prime}) = \alpha_{\{1,2,3\}}(\tau(\pi), \tau(\pi^{\prime}))$.
\end{example}

With the precedent simplifications, one only needs to compute and store the $j!$ coefficients $(\alpha_{\set{j}}(12\dots j,\pi))_{\pi\in\Rank{\set{j}}}$ for each $j\in\{2,\dots, k\}$. These coefficients are computed using the recursive formula from Theorem \ref{th:alpha-recursive-formula}. Let $j\in\{2,\dots,k\}$. If all coefficients the $\alpha_{\set{j'}}(12\dots j',\pi)$ for $\pi\in\Rank{\set{j'}}$ and $2\leq j'\leq j-1$, it is easy to see that the computation of each $\alpha_{\set{j}}(12\dots j,\pi)$ for $\pi\in\Rank{\set{j}}$ then has complexity bounded by $\binom{j}{2}$. The global complexity of the computation of the coefficients $(\alpha_{\{1,\dots,j\}}(12\dots j,\pi))_{\pi\in\Rank{\set{j}}, 2\leq j\leq k}$ is therefore bounded by
\[
\sum_{j=2}^{k}\binom{j}{2}j! \leq \frac{k-1}{2}\sum_{j=2}^{k}\left[(j+1)!-j!\right] \leq \frac{1}{2}k^{2}k!.
\]
This establishes Proposition \ref{prop:complexity-alpha-coefficients}. Proposition \ref{prop:complexity-synthesis-operator} is a direct consequence of Lemma \ref{lem:useful-lemma}.

\section{Proof of Lemma \ref{lem:commutativity}}

Lemma \ref{lem:commutativity} is a cornerstone in the construction of the MRA representation. Its proof relies on the exploitation of the combinatorial structure of the wavelet synthesis operators. This requires some more definitions. Let $\GGn := \Gn\cup\n$ be the set of all injective words on $\n$, including the words of length $1$ of the form $a$, with $a\in\n$.

\begin{definition}
	\label{def:elementary-operations}
	Let $\pi,\pi'\in\GGn$ such that $c(\pi)\cap c(\pi') = \emptyset$. Their concatenation product is then defined by
	\[
	\pi\pi' := \pi_{1}\dots\pi_{\vert\pi\vert}\pi'_{1}\dots\pi_{\vert\pi'\vert}'.
	\]
\end{definition}

The following lemma gives a combinatorial expression for the wavelet synthesis operator.

\begin{lemma}
	\label{lem:combinatorial-expression}
	Let $\pi\in\Gamma_{n}$ and $A\in\Subsets{\n}$ such that $c(\pi)\subset A$. Then one has
	\[
	\phi_{A}\delta_{\pi} = \frac{1}{(\vert A\vert - \vert\pi\vert + 1)!}\sum_{\substack{A_{1},A_{2}\subset A \\ A_{1}\sqcup A_{2} = A\setminus c(\pi)}}\sum_{\substack{\omega\in\Rank{A_{1}} \\ \omega'\in\Rank{A_{2}}}}\delta_{\omega\pi\omega'}.
	\]
\end{lemma}

\begin{proof}
	The proof only consists in noticing that
	\[
	\{ \sigma\in\Rank{A} \;\vert\; \pi\sqsubset \sigma \} = \{ \omega\pi\omega' \;\vert\; (\omega,\omega')\in\Rank{A_{1}}\times\Rank{A_{2}} \text{ with } A_{1}\sqcup A_{2} = A\setminus c(\pi) \}.
	\]
\end{proof}

Lemma \ref{lem:commutativity} then relies on the following lemma, the proof of which is straightforward and left to the reader.

\begin{lemma}
	\label{lem:elementary-commutativity}
	Let $A,A\subset\n$ be two disjoint subsets and $(\pi,\pi')\in\Rank{A}\times\Rank{A'}$. Then for any $B\in\Subsets{\n}$ such that $B\cap A \neq \emptyset$ and $B\cap A'\neq \emptyset$ one has
	\[
	(\pi\pi')_{\vert B} = \pi_{\vert B\cap A}\pi'_{\vert B\cap A'}.
	\]
\end{lemma}

\begin{proof}[Proof of Lemma \ref{lem:commutativity}]
	Let $A,B,C\in\Subsets{\n}$ such that $A\cup B\subset C$ and $\pi\in\Rank{A}$. We need to prove that $M_{B}\phi_{C}\delta_{\pi} = \phi_{B}M_{A\cap B}\delta_{\pi}$. Lemma \ref{lem:combinatorial-expression} gives on the one hand
	\[
	\phi_{B}M_{A\cap B}\delta_{\pi} = \phi_{B}\delta_{\pi_{\vert A\cap B}} = \frac{1}{(\vert B\vert - \vert A\cap B\vert + 1)!}\sum_{\substack{B_{1},B_{2}\subset B \\ B_{1}\sqcup B_{2} = B\setminus A}}\sum_{\substack{\omega\in\Rank{B_{1}} \\ \omega'\in\Rank{B_{2}}}}\delta_{\omega\pi_{\vert A\cap B}\omega'}
	\]
	and on the other hand
	\[
	(\vert C\vert - \vert A\vert + 1)!M_{B}\phi_{C}\delta_{\pi} = M_{B}\sum_{\substack{C_{1},C_{2}\subset C \\ C_{1}\sqcup C_{2} = C\setminus A}}\sum_{\substack{\omega\in\Rank{C_{1}} \\ \omega'\in\Rank{C_{2}}}}\delta_{\omega\pi\omega'} = \sum_{\substack{C_{1},C_{2}\subset C \\ C_{1}\sqcup C_{2} = C\setminus A}}\sum_{\substack{\omega\in\Rank{C_{1}} \\ \omega'\in\Rank{C_{2}}}}\delta_{(\omega\pi\omega')_{\vert B}}.
	\]
	Now, by Lemma \ref{lem:elementary-commutativity}, one has for any $C_{1},C_{2}\subset C$ such that $C_{1}\sqcup C_{2} = C\setminus A$ and $(\omega,\omega')\in\Rank{C_{1}}\times\Rank{C_{2}}$,
	\[
	(\omega\pi\omega')_{\vert B} = \omega_{\vert B\cap C_{1}}\pi_{\vert A\cap B}\,\omega'_{\vert B\cap C_{2}}.
	\]
	Therefore, doing the change of variables $B_{1} := C_{1}\cap B$, $B_{2} := C_{2}\cap B$, $\upsilon := \omega_{\vert B\cap C_{1}}$ and $\upsilon' := \omega'_{\vert B\cap C_{2}}$, one obtains
	\[
	M_{B}\phi_{C}\delta_{\pi} = \frac{1}{(\vert C\vert - \vert A\vert + 1)!}\sum_{\substack{B_{1},B_{2}\subset B \\ B_{1}\sqcup B_{2} = B\setminus A}}\sum_{\substack{\upsilon\in\Rank{B_{1}} \\ \upsilon'\in\Rank{B_{2}}}}c(B_{1},B_{2},\upsilon,\upsilon')\delta_{\upsilon\pi_{\vert A\cap B}\upsilon'},
	\]
	where the coefficient $c(B_{1},B_{2},\upsilon,\upsilon')$ is given by
	\begin{align*}
	c(B_{1},B_{2},\upsilon,\upsilon') 
	&= \sum_{\substack{C_{1},C_{2}\subset C \\ C_{1}\sqcup C_{2} = C\setminus A}}\sum_{\substack{\omega\in\Rank{C_{1}} \\ \omega'\in\Rank{C_{2}}}}\mathbb{I}\{C_{1}\cap B = B_{1}, C_{2}\cap B = B_{2}, \omega_{\vert B_{1}} = \upsilon, \omega'_{\vert B_{2}} = \upsilon'\}\\
	&= \sum_{\substack{C_{1},C_{2}\subset C \\ C_{1}\sqcup C_{2} = C\setminus A}}\mathbb{I}\{C_{1}\cap B = B_{1}, C_{2}\cap B = B_{2}\}\frac{\vert C_{1}\vert !}{\vert B_{1}\vert !}\frac{\vert C_{2}\vert !}{\vert B_{2}\vert !}\\
	&= \frac{\vert C_{1}\vert !}{\vert B_{1}\vert !}\frac{\vert C_{2}\vert !}{\vert B_{2}\vert !}\sum_{k=0}^{\vert C\vert - \vert A\cup B\vert}(k+\vert B_{1}\vert)!(\vert C\vert - \vert A\cup B\vert - k + \vert B_{2}\vert)!\\
	&= (\vert C\vert - \vert A\cup B\vert)!\sum_{k=0}^{\vert C\vert - \vert A\cup B\vert}\binom{k + \vert B_{1}\vert}{\vert B_{1}\vert}\binom{\vert C\vert - \vert A\cup B\vert - k + \vert B_{2}\vert}{\vert B_{2}\vert}\\
	&= (\vert C\vert - \vert A\cup B\vert)!\binom{\vert C\vert - \vert A\cup B\vert + \vert B_{1}\vert + \vert B_{2}\vert + 1}{\vert C\vert - \vert A\cup B\vert},
	\end{align*}
	where the last equality is given by Lemma \ref{lem:combinatorial-identity} below for $n := \vert C\vert - \vert A\cup B\vert$, $r := \vert B_{1}\vert$ and $s := \vert B_{2}\vert$. The proof is concluded by noticing that for $B_{1},B_{2}\subset B$ such that $B_{1}\sqcup B_{2} = B\setminus A$, $\vert B_{1}\vert + \vert B_{2}\vert = \vert B\vert - \vert A\cap B\vert$ and $\vert A\cup B\vert - \vert B_{1}\vert - \vert B_{2}\vert = \vert A\vert$, so that
	\[
	\binom{\vert C\vert - \vert A\cup B\vert + \vert B_{1}\vert + \vert B_{2}\vert + 1}{\vert C\vert - \vert A\cup B\vert} = \frac{(\vert C\vert - \vert A\vert + 1)!}{(\vert C\vert - \vert A\cup B\vert)!(\vert B\vert - \vert A\cap B\vert + 1)!}.
	\]
\end{proof}

\begin{lemma}
	\label{lem:combinatorial-identity}
	For any $n,r,s\in\mathbb{N}$, one has the identity
	\[
	\sum_{k=0}^{n}\binom{k+r}{r}\binom{n-k+s}{s} = \binom{n+r+s+1}{n}
	\]
\end{lemma}

\begin{proof}
	Denote the sum by $S_{n}(r,s)$. By Pascal's rule, one has
	\begin{align*}
	S_{n}(r+1,s) 
	&= \sum_{k=0}^{n}\binom{k+r+1}{k}\binom{n-k+s}{s}\\
	&= \binom{n+s}{s} + \sum_{k=1}^{n}\binom{k+r}{k}\binom{n-k+s}{s} + \sum_{k=1}^{n}\binom{k+r}{k-1}\binom{n-k+s}{s}\\
	&= \sum_{k=0}^{n}\binom{k+r}{k}\binom{n-k+s}{s} + \sum_{k=0}^{n-1}\binom{k+r+1}{k}\binom{n-1-k+s}{s}\\
	&= S_{n}(r,s) + S_{n-1}(r+1,s).
	\end{align*}
	One thus has $S_{n}(r+1,s) - S_{n-1}(r+1,s) = S_{n}(r,s)$ and, noticing that $S_{0}(r,s) = 1$ for all $r,s\in\mathbb{N}$, one obtains by a telescoping sum
	\[
	S_{n}(r+1,s) = \sum_{k=0}^{n}S_{k}(r,s).
	\]
	The identity is now proven by induction on $r$ using the well-known identity 
	\begin{equation}
	\label{eq:well-known-identity}
	\sum_{j=k}^{n}\binom{j}{k} = \binom{n+1}{k+1}
	\end{equation} 
	(it can be proven by induction on $n$ with Pascal's rule). For $r=0$ one has
	\[
	S_{n}(0,s) = \sum_{k=0}^{n}\binom{n-k+s}{s} = \sum_{j=s}^{n+s}\binom{j}{s} = \binom{n+s+1}{s+1},
	\]
	which satisfies the identity. Assuming the identity true for all $k\leq r$, one has
	\[
	S_{n}(r+1,s) = \sum_{k=0}^{n}\binom{k+r+s+1}{r+s+1} = \sum_{j=r+s+1}^{n+r+s+1}\binom{j}{r+s+1} = \binom{n+r+s+2}{r+s+2},
	\]
	where the last equality also stems from identity \eqref{eq:well-known-identity}. This concludes the proof.
\end{proof}

\section{Proofs of Proposition \ref{prop:connection-Fourier-MRA} and Theorem \ref{th:connection-other-MRA}}

A full ranking $\pi_{1}\succ\dots\succ \pi_{n}$ is either seen as the word $\pi_{1}\dots \pi_{n}\in\Rank{\n}$ or as the permutation $\sigma\in\Sn$ defined by $\sigma(\pi_{i}) = i$ for all $i\in\{1,\dots,n\}$. The action of $\Sn$ on $\Gn$ defined in Subsection \ref{subsec:two-decompositions} by $\tau\cdot\pi = \tau(\pi)$ is transitive on $\Rank{\n}$. If $\sigma\in\Sn$ is the permutation associated to the full ranking $\pi\in\Rank{\n}$ then the permutation associated to $\tau\cdot\pi$ is the permutation $\sigma'\in\Sn$ defined by $\sigma'(\tau(\pi_{i})) = i$. In other words it is such that $\sigma'\tau = \sigma$, or equivalently it is given by $\sigma' = \sigma\tau^{-1}$. Hence, through the identification of $\Rank{\n}$ and $\Sn$, the natural action of $\Sn$ on $\Rank{\n}$ is  equivalent to the classic right translation $\sigma\mapsto\sigma\tau^{-1}$ on $\Sn$. We therefore use this action and the associated representation on $\Space{\Sn}$, namely the right regular representation.\\

\begin{proof}[Proof of Proposition \ref{prop:connection-Fourier-MRA}]
	Theorem \ref{th:MRA-decomposition} shows that $\phi_{\n}$ is a linear isomorphism between $\mathbb{H}_{n}$ and $\Space{\Sn}$, and Proposition \ref{prop:translation-invariance} shows that for any $\tau\in\Sn$, $T_{\tau}\phi_{\n} = \phi_{\tau(\n)}T_{\tau} = \phi_{\n}T_{\tau}$. This concludes the proof.
\end{proof}

The proof of Theorem \ref{th:connection-other-MRA} relies on the properties of the embedding operator $\phi'_{A}$, given by the following lemma. For $A\in\Subsets{\n}$ and $\pi'\in\Gamma^{\vert A\vert}$, we define the operator $T_{A\rightarrow\pi'}:\Space{\Gn}\rightarrow\Gn$ that maps the Dirac function of a ranking $\pi\in\Gn$ to the Dirac function of the ranking obtained by replacing $\pi_{\vert A}$ by $\pi'$ if $A\subset c(\pi)$ or to $0$ otherwise.

\begin{lemma}
	\label{lem:properties-alternative-embedding}
	Let $A\in\Subsets{\n}$ and $\pi\in\Rank{A}$. The following properties hold.
	\begin{enumerate}
		\item For all $A',C\in\Subsets{\n}$ such that $A\subset A'\subset C$,
		\[
		\phi'_{C}\delta_{\pi} = \phi'_{C}\phi'_{A'}\delta_{\pi}.
		\]
		\item For all $B,C\in\Subsets{\n}$ such that $A\cup B \subset C$,
		\[
		M_{B}\phi'_{C}\delta_{\pi} = M_{B}\phi'_{A\cup B}\delta_{\pi}.
		\]
		\item For all $B\in\Subsets{\n}$,
		\[
		M_{B}\phi'_{A\cup B}\delta_{\pi} = \sum_{\substack{A_{1}\subset A\setminus B \\ B_{1}\subset B\setminus A \\ \vert A_{1}\vert = \vert B_{1}\vert}}\lambda_{\vert B_{1}\vert}\sum_{\pi'\in\Rank{B_{1}}}\phi'_{B}M_{(A\cap B)\sqcup B_{1}}T_{A_{1}\rightarrow\pi'}\delta_{\pi},
		\]
		where $\lambda_{t} = (\vert A\vert !\vert B\vert !)/(\vert A\cup B\vert ! (\vert A\cap B\vert + t)!)$ for any $t\in\mathbb{N}$.
		\item For all $\tau\in\Sn$
		\[
		T_{\tau}\phi'_{A} = \phi'_{\tau(A)}T_{\tau}
		\]
	\end{enumerate}
\end{lemma}

\begin{proof}
	We prove the properties in the order.\\
	$1.$ Let $A',C\in\Subsets{\n}$ such that $A\subset A'\subset C$. One has
	\[
	\phi'_{C}\phi'_{A'}\delta_{\pi} = \frac{\vert A\vert !}{\vert A'\vert !}\phi'_{C}\sum_{\substack{\pi'\in\Rank{A'} \\ \pi\subset\sigma}}\delta_{\pi'} = \frac{\vert A\vert !}{\vert A'\vert !}\frac{\vert A'\vert !}{\vert C\vert !}\sum_{\substack{\pi'\in\Rank{A'} \\ \pi\subset\pi'}}\sum_{\substack{\sigma\in\Rank{C} \\ \pi'\subset\sigma}}\delta_{\sigma} = \frac{\vert A\vert !}{\vert C\vert !}\sum_{\substack{\sigma\in\Rank{C} \\ \pi\subset\sigma}}\delta_{\sigma} = \phi'_{C}\delta_{\pi}.
	\]
	$2.$ Let $B,C\in\Subsets{\n}$ such that $A\cup B \subset C$. By definition of the marginal operator and by Property $1.$, one has
	\[
	M_{B}\phi'_{C}\delta_{\pi} = M_{B}M_{A\cup B}\phi'_{C}\phi'_{A\cup B}\delta_{\pi}.
	\]
	Now, for any $A'\in\Subsets{C}$ and $\pi'\in\Rank{A'}$, it is clear that $M_{A'}\phi'_{C}\delta_{\pi'} = \delta_{\pi'}$. Applied to $A\cup B$, this concludes the proof of Property $2.$\\
	$3.$ This is certainly the longest part of the proof. We introduce two new operators. First, the deletion operator
	\[
	\varrho_{a} : \delta_{\pi} \mapsto \delta_{\pi\setminus\{a\}} \qquad\text{for } a\in c(\pi),
	\]
	where $\pi\setminus\{a\}$ is the ranking obtained by deleting the item $a$ in $\pi$. Second, the insertion operator 
	\[
	\varrho^{\ast}_{b} : \delta_{\pi} \rightarrow \sum_{i=1}^{\vert\pi\vert + 1}\delta_{\pi\lhd_{i}b} \qquad\text{for } b\not\in c(\pi),
	\]
	where $\pi\lhd_{i}b$ is the ranking obtained by inserting item $b$ at the $i^{th}$ position. Then for $A'\in\Subsets{A}$ with $A\setminus A' = \{a_{1},\dots, a_{r}\}$, and $B$ such that $A\subset B$ with $B\setminus A = \{b_{1},\dots,b_{s}\}$, one has
	\[
	M_{A'}\delta_{\pi} = \varrho_{a_{1}}\dots\varrho_{a_{r}}\delta_{\pi} \qquad\text{and}\qquad \phi'_{B}\delta_{\pi} = \frac{\vert A\vert !}{\vert B\vert !}\varrho^{\ast}_{b_{1}}\dots\varrho^{\ast}_{b_{s}}.
	\]
	Property $3.$ is then equivalent for any $B\in\Subsets{\n}$ to
	\begin{multline}
	\label{eq:hyp-de-rec-0}
	\varrho_{a_{1}}\dots\varrho_{a_{r}}\varrho^{\ast}_{b_{1}}\dots\varrho^{\ast}_{b_{s}}\delta_{\pi} = \\
	\sum_{k=0}^{\min(r,s)}\sum_{\substack{A_{1}\sqcup\{a_{i_{1}},\dots,a_{i_{r-k}}\}=\{a_{1},\dots,a_{r}\} \\ B_{1}\sqcup\{b_{j_{1}},\dots, b_{j_{s-k}}\} = \{b_{1},\dots,b_{s}\} \\ \vert A_{1}\vert = \vert B_{1}\vert = k}}
	\sum_{\pi'\in\Rank{B_{1}}}
	\varrho^{\ast}_{b_{j_{1}}}\dots\varrho^{\ast}_{b_{j_{s-k}}}\varrho_{a_{i_{1}}}\dots\varrho_{a_{i_{r-k}}} T_{A_{1}\rightarrow\pi'}\delta_{\pi},
	\end{multline}
	where $\{a_{1},\dots, a_{r}\} = A\setminus B$ and $\{b_{1},\dots,b_{s}\} = B\setminus A$. We prove Formula \eqref{eq:hyp-de-rec-0} in three steps. First for $r = s = 1$, one has
	\[
	\varrho_{a}\varrho^{\ast}_{b}\delta_{\pi} = \sum_{i=1}^{\vert A\vert + 1}\varrho_{a}\delta_{\pi\lhd_{i}b} = \delta_{\pi\lhd_{1}b\setminus\{a\}} + \dots + \delta_{\pi\lhd_{\pi(a)}b\setminus\{a\}} + \delta_{\pi\lhd_{\pi(a)+1}b\setminus\{a\}} + \dots + \delta_{\pi\lhd_{1}b\setminus\{a\}}.
	\]
	The ranking $\pi\lhd_{\pi(a)}b\setminus\{a\}$ is the ranking obtained by inserting $b$ at the left of $a$ in $\pi$ and then by deleting $a$. The ranking $\pi\lhd_{\pi(a)+1}b\setminus\{a\}$ is the ranking obtained by inserting $b$ at the right of $a$ in $\pi$ and then by deleting $a$. It is clear that they are both equal to the ranking $\pi_{\{a\}\rightarrow b}$ obtained by changing $a$ to $b$ in $\pi$. Hence one has
	\[
	\varrho_{a}\varrho^{\ast}_{b}\delta_{\pi} = \varrho^{\ast}_{b}\varrho_{a}\delta_{\pi} + T_{\{a\}\rightarrow b}\delta_{\pi}
	\]
	and Formula \eqref{eq:hyp-de-rec-0} is satisfied. We now show by induction on $s\in\{1,\dots,\vert B\setminus A\vert\}$ that
	\begin{equation}
	\label{eq:hyp-de-rec-1}
	\varrho_{a}\varrho_{b_{1}}^{\ast}\dots\varrho_{b_{s}}^{\ast}\delta_{\pi} = \varrho_{b_{1}}^{\ast}\dots\varrho_{b_{s}}^{\ast}\varrho_{a}\delta_{\pi} + \sum_{i=1}^{s}\varrho_{b_{1}}^{\ast}\dots\varrho_{b_{i-1}}^{\ast}\varrho_{b_{i+1}}^{\ast}\dots\varrho_{b_{s}}^{\ast}T_{\{a\}\rightarrow b_{i}}\delta_{\pi}.
	\end{equation}
	Notice that for any $A_{1}\varsubsetneq A\setminus B$, $\pi'\in\Rank{B_{1}}$ with $B_{1}\subset B\setminus A$, $a\in A\setminus (A_{1}\sqcup B)$ and $b\in B\setminus(A\sqcup B_{1})$ one clearly has
	\begin{equation}
	\label{eq:commutation}
	\varrho_{a}T_{A_{1}\rightarrow \pi'}\delta_{\pi} = T_{A_{1}\rightarrow \pi'}\varrho_{a}\delta_{\pi} \qquad\text{\qquad } \varrho^{\ast}_{b}T_{A_{1}\rightarrow \pi'}\delta_{\pi} = T_{A_{1}\rightarrow \pi'}\varrho^{\ast}_{b}\delta_{\pi}.
	\end{equation}
	Therefore, assuming \eqref{eq:hyp-de-rec-1} true for $s\leq \vert B\setminus A\vert - 1$, one has
	\begin{align*}
	\varrho_{a}\varrho_{b_{1}}^{\ast}\dots\varrho_{b_{s+1}}^{\ast}\delta_{\pi} 
	&= \varrho_{a}\varrho_{b_{1}}^{\ast}\dots\varrho_{b_{s}}^{\ast}\left(\varrho_{b_{s+1}}^{\ast}\delta_{\pi}\right)\\
	&= \varrho_{b_{1}}^{\ast}\dots\varrho_{b_{s}}^{\ast}\varrho_{a}\left(\varrho_{b_{s+1}}^{\ast}\delta_{\pi}\right) + \sum_{i=1}^{s}\varrho_{b_{1}}^{\ast}\dots\varrho_{b_{i-1}}^{\ast}\varrho_{b_{i+1}}^{\ast}\dots\varrho_{b_{s}}^{\ast}T_{\{a\}\rightarrow b_{i}}\left(\varrho_{b_{s+1}}^{\ast}\delta_{\pi}\right)\\
	&= \varrho_{b_{1}}^{\ast}\dots\varrho_{b_{s+1}}^{\ast}\varrho_{a}\delta_{\pi} + \varrho_{b_{1}}^{\ast}\dots\varrho_{b_{s}}^{\ast}T_{\{a\}\rightarrow b_{s+1}} + \sum_{i=1}^{s}\varrho_{b_{1}}^{\ast}\dots\varrho_{b_{i-1}}^{\ast}\varrho_{b_{i+1}}^{\ast}\dots\varrho_{b_{s+1}}^{\ast}T_{\{a\}\rightarrow b_{i}}\delta_{\pi}\\
	&= \varrho_{b_{1}}^{\ast}\dots\varrho_{b_{s+1}}^{\ast}\varrho_{a}\delta_{\pi} + \sum_{i=1}^{s+1}\varrho_{b_{1}}^{\ast}\dots\varrho_{b_{i-1}}^{\ast}\varrho_{b_{i+1}}^{\ast}\dots\varrho_{b_{s+1}}^{\ast}T_{\{a\}\rightarrow b_{i}}\delta_{\pi},
	\end{align*}
	which concludes the proof of \eqref{eq:hyp-de-rec-1}. At last, we show \eqref{eq:hyp-de-rec-0} by induction on $r\in\{1,\dots, \vert A\setminus B\vert\}$. Assuming it true for $r\leq \vert A\setminus B\vert - 1$, one has
	\begin{align*}
	&\varrho_{a_{1}}\dots\varrho_{a_{r+1}}\varrho^{\ast}_{b_{1}}\dots\varrho^{\ast}_{b_{s}}\delta_{\pi}\\
	&= \varrho_{a_{r+1}}\left[\varrho_{a_{1}}\dots\varrho_{a_{r}}\varrho^{\ast}_{b_{1}}\dots\varrho^{\ast}_{b_{s}}\delta_{\pi}\right]\\
	&= \varrho_{a_{r+1}}\left[
	\sum_{k=0}^{\min(r,s)}\sum_{\substack{A_{1}\sqcup\{a_{i_{1}},\dots,a_{i_{r-k}}\}=\{a_{1},\dots,a_{r}\} \\ B_{1}\sqcup\{b_{j_{1}},\dots, b_{j_{s-k}}\} = \{b_{1},\dots,b_{s}\} \\ \vert A_{1}\vert = \vert B_{1}\vert = k}}
	\sum_{\pi'\in\Rank{B_{1}}}
	\varrho^{\ast}_{b_{j_{1}}}\dots\varrho^{\ast}_{b_{j_{s-k}}}\varrho_{a_{i_{1}}}\dots\varrho_{a_{i_{r-k}}} T_{A_{1}\rightarrow\pi'}\delta_{\pi}
	\right].
	\end{align*}
	If $r\leq s$, Equations \eqref{eq:hyp-de-rec-1} and \eqref{eq:commutation} give
	\begin{align*}
	&\varrho_{a_{1}}\dots\varrho_{a_{r+1}}\varrho^{\ast}_{b_{1}}\dots\varrho^{\ast}_{b_{s}}\delta_{\pi}\\
	&= 
	\sum_{k=0}^{r}\sum_{\substack{A_{1}\sqcup\{a_{i_{1}},\dots,a_{i_{r-k}}\}=\{a_{1},\dots,a_{r}\} \\ B_{1}\sqcup\{b_{j_{1}},\dots, b_{j_{s-k}}\} = \{b_{1},\dots,b_{s}\} \\ \vert A_{1}\vert = \vert B_{1}\vert = k}}
	\sum_{\pi'\in\Rank{B_{1}}}
	\left[\varrho^{\ast}_{b_{j_{1}}}\dots\varrho^{\ast}_{b_{j_{s-k}}}\varrho_{a_{r+1}}\varrho_{a_{i_{1}}}\dots\varrho_{a_{i_{r-k}}} T_{A_{1}\rightarrow\pi'}\delta_{\pi}\right. \\
	&\qquad\qquad\qquad\qquad\left. + \sum_{i=1}^{s-k}\varrho^{\ast}_{b_{j_{1}}}\dots\varrho^{\ast}_{b_{j_{i-1}}}\varrho^{\ast}_{b_{j_{i+1}}}\dots\varrho^{\ast}_{b_{j_{s-k}}}T_{\{a_{r+1}\}\rightarrow b_{j_{i}}}\varrho_{a_{i_{1}}}\dots\varrho_{a_{i_{r-k}}} T_{A_{1}\rightarrow\pi'}\delta_{\pi}\right]\\
	&= 
	\sum_{k=0}^{r}\sum_{\substack{A_{1}\sqcup\{a_{i_{1}},\dots,a_{i_{r+1-k}}\}=\{a_{1},\dots,a_{r+1}\} \\ B_{1}\sqcup\{b_{j_{1}},\dots, b_{j_{s-k}}\} = \{b_{1},\dots,b_{s}\} \\ \vert A_{1}\vert = \vert B_{1}\vert = k \\ a_{r+1}\not\in A_{1}}}
	\sum_{\pi'\in\Rank{B_{1}}}
	\varrho^{\ast}_{b_{j_{1}}}\dots\varrho^{\ast}_{b_{j_{s-k}}}\varrho_{a_{i_{1}}}\dots\varrho_{a_{i_{r+1-k}}} T_{A_{1}\rightarrow\pi'}\delta_{\pi} \\
	&+ 
	\sum_{k=1}^{r+1}\sum_{\substack{A_{1}\sqcup\{a_{i_{1}},\dots,a_{i_{r-k}}\}=\{a_{1},\dots,a_{r+1}\} \\ B_{1}\sqcup\{b_{j'_{1}},\dots, b_{j'_{s-k-1}}\} = \{b_{1},\dots,b_{s}\} \\ \vert A_{1}\vert = \vert B_{1}\vert = k+1 \\ a_{r+1}\in A_{1}}}
	\sum_{\pi'\in\Rank{B_{1}}}
	\varrho^{\ast}_{b_{j'_{1}}}\dots\varrho^{\ast}_{b_{j'_{s-k-1}}}\varrho_{a_{i_{1}}}\dots\varrho_{a_{i_{r-k}}} T_{A_{1}\rightarrow\pi'}\delta_{\pi}\\
	&= 
	\sum_{k=0}^{r+1}\sum_{\substack{A_{1}\sqcup\{a_{i_{1}},\dots,a_{i_{r+1-k}}\}=\{a_{1},\dots,a_{r+1}\} \\ B_{1}\sqcup\{b_{j_{1}},\dots, b_{j_{s-k}}\} = \{b_{1},\dots,b_{s}\} \\ \vert A_{1}\vert = \vert B_{1}\vert = k}}
	\sum_{\pi'\in\Rank{B_{1}}}
	\varrho^{\ast}_{b_{j_{1}}}\dots\varrho^{\ast}_{b_{j_{s-k}}}\varrho_{a_{i_{1}}}\dots\varrho_{a_{i_{r+1-k}}} T_{A_{1}\rightarrow\pi'}\delta_{\pi}.
	\end{align*}
	If $s < r$, Equations \eqref{eq:hyp-de-rec-1} and \eqref{eq:commutation} give
	\begin{align*}
	&\varrho_{a_{1}}\dots\varrho_{a_{r+1}}\varrho^{\ast}_{b_{1}}\dots\varrho^{\ast}_{b_{s}}\delta_{\pi}\\
	&= 
	\sum_{k=0}^{s-1}\sum_{\substack{A_{1}\sqcup\{a_{i_{1}},\dots,a_{i_{r-k}}\}=\{a_{1},\dots,a_{r}\} \\ B_{1}\sqcup\{b_{j_{1}},\dots, b_{j_{s-k}}\} = \{b_{1},\dots,b_{s}\} \\ \vert A_{1}\vert = \vert B_{1}\vert = k}}
	\sum_{\pi'\in\Rank{B_{1}}}
	\left[\varrho^{\ast}_{b_{j_{1}}}\dots\varrho^{\ast}_{b_{j_{s-k}}}\varrho_{a_{r+1}}\varrho_{a_{i_{1}}}\dots\varrho_{a_{i_{r-k}}} T_{A_{1}\rightarrow\pi'}\delta_{\pi}\right. \\
	&\left. + \sum_{i=1}^{s-k}\varrho^{\ast}_{b_{j_{1}}}\dots\varrho^{\ast}_{b_{j_{i-1}}}\varrho^{\ast}_{b_{j_{i+1}}}\dots\varrho^{\ast}_{b_{j_{s-k}}}T_{\{a_{r+1}\}\rightarrow b_{j_{i}}}\varrho_{a_{i_{1}}}\dots\varrho_{a_{i_{r-k}}} T_{A_{1}\rightarrow\pi'}\delta_{\pi}\right] \\
	&+ \sum_{\substack{A_{1}\sqcup\{a_{i_{1}},\dots,a_{i_{r-s}}\}=\{a_{1},\dots,a_{r}\} \\ \vert A_{1}\vert = s}}
	\sum_{\pi'\in\Rank{\{b_{1},\dots,b_{s}\}}}\varrho_{a_{r+1}}\varrho_{a_{i_{1}}}\dots\varrho_{a_{i_{r-s}}} T_{A_{1}\rightarrow\pi'}\delta_{\pi}\\
	&= 
	\sum_{k=0}^{s}\sum_{\substack{A_{1}\sqcup\{a_{i_{1}},\dots,a_{i_{r+1-k}}\}=\{a_{1},\dots,a_{r+1}\} \\ B_{1}\sqcup\{b_{j_{1}},\dots, b_{j_{s-k}}\} = \{b_{1},\dots,b_{s}\} \\ \vert A_{1}\vert = \vert B_{1}\vert = k}}
	\sum_{\pi'\in\Rank{B_{1}}}
	\varrho^{\ast}_{b_{j_{1}}}\dots\varrho^{\ast}_{b_{j_{s-k}}}\varrho_{a_{i_{1}}}\dots\varrho_{a_{i_{r+1-k}}} T_{A_{1}\rightarrow\pi'}\delta_{\pi}.
	\end{align*}
	In both cases the proof is concluded.\\
	$4.$ The proof of Property $4.$ is fully analogous to the one of Proposition \ref{prop:translation-invariance}. It is left to the reader.
\end{proof}

Property $3$ from Lemma \ref{lem:properties-alternative-embedding} is the analogue of Lemma \ref{lem:commutativity}. It allows to prove Theorem \ref{th:connection-other-MRA}.

\begin{proof}[Proof of Theorem \ref{th:connection-other-MRA}]
	One clearly has $\phi'_{\n}(H^{0}) = V^{0}$ and $V^{0}\cong S^{(n)}$. Let $k\in\{2,\dots,n\}$ and $A\in\Subsets{\n}$ with $\vert A\vert = k$. We define the space $W_{A}^{k} = W^{k}\cap\Span\{\1{\Sn(\pi)} \;\vert\; \pi\in\Rank{A}\}$. We first prove that $\phi'_{\n}(H_{A})\subset W^{k}_{A}$. Let $F\in H_{A}$ and let $B\in\Subsets{\n}$ with $\vert B\vert \leq k-1$. By definition $\phi'_{\n}(H_{A})\subset\Span\{\1{\Sn(\pi)} \;\vert\; \pi\in\Rank{A}\}$. We then need to prove that $M_{B}\phi'_{\n}F = 0$. Properties $2.$ and $3.$ of Lemma \ref{lem:properties-alternative-embedding} give
	\[
	M_{B}\phi'_{\n}F = M_{B}\phi'_{A\cup B}F = \sum_{\substack{A_{1}\subset A\setminus B \\ B_{1}\subset B\setminus A \\ \vert A_{1}\vert = \vert B_{1}\vert}}\sum_{\pi'\in\Rank{B_{1}}}\phi'_{B}M_{(A\cap B)\sqcup B_{1}}T_{A_{1}\rightarrow\pi'}F.
	\]
	The space $H^{k}$ being stable under translations, one has $T_{A_{1}\rightarrow\pi'}F\in H^{k}$ for any $A_{1}\subset A$ and $\pi'\in\Gamma^{\vert A_{1}\vert}$. Now, for any $B_{1}\subset B\setminus A$, $\vert (A\cap B)\sqcup B_{1}\vert = \vert A\cap B\vert + \vert B_{1}\vert \leq \vert B\vert \leq k-1$. Hence $M_{(A\cap B)\sqcup B_{1}}T_{A_{1}\rightarrow\pi'}F = 0$ and $M_{B}\phi'_{\n}F = 0$. One therefore has $\phi'_{\n}(H_{A})\subset W^{k}_{A}$. In addition, for $F\in H_{A}$ such that $\phi'_{\n}F = 0$, property $2.$ of Lemma \ref{lem:properties-alternative-embedding} gives $0 = M_{A}\phi'_{\n}F = F$. The operator $\phi'_{\n}$ is thus an injection from $H_{A}$ to $W_{A}^{k}$ and thus $\dim W_{A}^{k}\geq d_{k}$ by Theorem \ref{th:topology}. Now, by construction $W^{k} = \bigoplus_{\vert A\vert = k}W_{A}^{k}$, so that
	\[
	n! = \dim\left(V^{0}\oplus\bigoplus_{k=2}^{n}\bigoplus_{\vert A\vert = k}W_{A}^{k}\right) \leq 1 + \sum_{k=2}^{n}\binom{n}{k}d_{k} = n!.
	\]
	Hence all the inequalities are equalities and therefore $\phi'_{\n}(H^{k}) = W^{k}$. Property $4.$ of Lemma \ref{lem:properties-alternative-embedding} then ensures that $W^{k}\cong H^{k}$.
\end{proof}

\section{Technical proofs of Subsection \ref{subsec:pairwise-comparisons}}

The proofs of Theorem \ref{th:explicit-decomposition-H-2} and Proposition \ref{prop:connection-HodgeRank} require the two following lemmas. The proof of the first one is straightforward and left to the reader.

\begin{lemma}
	\label{lem:inner-products}
	For $a,b,c\in\n$ with $b\neq c$ one has
	\[
	\left\langle e_{a},e_{b}\right\rangle =
	\left\{
	\begin{aligned}
	-1 &\qquad\text{if } a \neq b\\
	n-1 &\qquad\text{if } a = b
	\end{aligned}\right.
	\qquad\text{and}\qquad
	\left\langle e_{a},x_{b\succ c}\right\rangle = \left\{
	\begin{aligned}
	1 &\qquad\text{if } a = b\\
	-1 &\qquad \text{if } a = c\\
	0 &\qquad \text{if } a\not\in\{b,c\}
	\end{aligned}\right.
	\]
\end{lemma}

\begin{lemma}
	\label{lem:other-expressions}
	For $a,b\in\n$ with $a\neq b$ and $s\in\mathbb{R}^{n}$ one has
	\[
	\sum_{1\leq i < j \leq n}(s_{i} - s_{j})x_{i\succ j} = \sum_{i\in\n}s_{i}e_{i} \qquad\text{and}\qquad f_{(a,b)} = n\, x_{a\succ b} + e_{b} - e_{a}.
	\]
\end{lemma}

\begin{proof}
	Recalling that for any $i,j\in\n$ with $i\neq j$, $x_{j\succ i} = -x_{i\succ j}$, straightforward calculations give
	\begin{align*}
	\sum_{1\leq i < j \leq n}(s_{i} - s_{j})x_{i\succ j} = \frac{1}{2}\sum_{1\leq i\neq j\leq n}(s_{i} - s_{j})x_{i\succ j} = \sum_{i\in\n}s_{i}\sum_{j\neq i}x_{i\succ j} + \sum_{j\in\n}s_{j}\sum_{i\neq j}x_{j\succ i} = \sum_{i\in\n}s_{i}e_{i}
	\end{align*}
	and
	\[
	f_{(a,b)} = \sum_{c\not\in\{a,b\}}(x_{a\succ b} + x_{b\succ c} + x_{c\succ a}) = (n-2)x_{a\succ b} + (e_{b} - x_{b\succ a}) - (e_{a} - x_{a\succ b}) = n\, x_{a\succ b} + e_{b} - e_{a}.
	\]
\end{proof}

\begin{proof}[Proof of Theorem \ref{th:explicit-decomposition-H-2}]
	We first show that the spaces $H_{1}^{2}$ and $H_{2}^{2}$ are orthogonal. Let $a,b,c\in\n$ with $b\neq c$. By Lemmas \ref{lem:inner-products} and \ref{lem:other-expressions}, one has
	\[
	\left\langle e_{a}, f_{(b,c)}\right\rangle = n\left\langle e_{a},x_{b\succ c}\right\rangle + \left\langle e_{a},e_{c}\right\rangle - \left\langle e_{a},e_{b}\right\rangle = 
	\left\{
	\begin{aligned}
	n - 1 - (n-1) = 0 &\qquad\text{if } a = b\\
	-n + (n-1) +1 = 0 &\qquad\text{if } a = c\\
	0 -1 +1 = 0 &\qquad\text{if } a\not\in\{b,c\}\\
	\end{aligned}\right.
	.
	\]
	Next we prove that $H_{1}^{2}$ and $H_{2}^{2}$ are both representations of $\Sn$, or equivalently stable under translations. For $a,b\in\n$ with $a\neq b$ and $\tau\in\Sn$ one has by definition $T_{\tau}x_{a\succ b} = x_{\tau(a)\succ \tau(b)}$, so that
	\[
	T_{\tau}e_{a} = \sum_{c\neq a}x_{\tau(a)\succ \tau(c)} = \sum_{c\neq a}x_{\tau(a)\succ c} = e_{\tau(a)}
	\]
	and
	\begin{multline*}
	T_{\tau}f_{(a,b)} = \sum_{c\not\in\{a,b\}}(x_{\tau(a)\succ \tau(b)} + x_{\tau(b)\succ \tau(c)} + x_{\tau(c)\succ \tau(a)})\\
	= \sum_{c\not\in\{a,b\}}T_{\tau}(x_{\tau(a)\succ \tau(b)} + x_{\tau(b)\succ c} + x_{c\succ \tau(a)}) = f_{(\tau(a),\tau(b)}.
	\end{multline*}
	Now, Theorem \ref{th:decomposition-H-k} ensures that $H^{2}\cong S^{(n-1,1)}\oplus S^{(n-2,1,1)}$ as representations of $\Sn$, where $S^{(n-1,1)}$ and $S^{(n-2,1,1)}$ are both irreducible representations. Since $H_{1}^{2}\neq\{0\}$, one then necessarily has $H_{1}^{2}\cong S^{(n-1,1)}$ and $H_{2}^{2}\cong S^{(n-2,1,1)}$ or $H_{2}^{2}\cong S^{(n-1,1)}$ and $H_{1}^{2}\cong S^{(n-2,1,1)}$. To conclude, notice that since $H_{1}^{2} = \Span\{e_{a} \;\vert\; a\in\n\}$, $\dim H_{1}^{2}\leq n < \binom{n-1}{2}$ and one cannot have $H_{1}^{2}\cong S^{(n-2,1,1)}$. Hence the other alternative is true and this concludes the proof.
\end{proof}

\begin{proof}[Proof of Proposition \ref{prop:connection-HodgeRank}]
	Following the notations of \citet{JLYY11}, we denote by $G$ the complete graph on $\n$ and by $K_{G}$ its clique complex. The space of ``edge flows'' on $G$ is defined by $C^{1}(K_{G},\mathbb{R}) := \{(X_{i,j})_{i,j}\in\mathbb{R}^{n\times n} \;\vert\; X_{i,j} = - X_{j,i}\}$. Identifying index $(i,j)$ with $ij$, one clearly has $C^{1}(K_{G},\mathbb{R}) = H^{2}$. The HodgeRank decomposition, established by theorem 2 in \cite{JLYY11}, is then given by 
	\[
	H^{2} = \image(\text{grad})\overset{\perp}{\oplus}\image(\text{curl}^{\ast}) = \image(\text{grad})\overset{\perp}{\oplus}\image(\text{grad})^{\perp},
	\]
	where by definition $\image(\text{grad}) = \{\sum_{1\leq i < j \leq n}(s_{i}-s_{j})x_{i\succ j} \;\vert\; s\in\mathbb{R}^{n}\}$. Now, Lemma \ref{lem:other-expressions} shows that for any $s\in\mathbb{R}^{n}$, an element of the form $\{\sum_{1\leq i < j \leq n}(s_{i}-s_{j})x_{i\succ j}$ is of the form $\sum_{i\in\n}s_{i}e_{i}$ and reciprocally. This means that $\image(\text{grad}) = H_{1}^{2}$, which concludes the proof.
\end{proof}

\bibliographystyle{apalike}
\bibliography{biblio}

\end{document}